\documentclass[final,onefignum,onetabnum]{siamart190516}

\usepackage{lipsum}
\usepackage{amsfonts}
\usepackage{graphicx}
\usepackage{algorithmic}
\usepackage{multirow}
\usepackage{tikz}
\usetikzlibrary{positioning,fit}

\usepackage{cleveref}

\DeclareGraphicsExtensions{.pdf,.png,.jpg}

\usepackage{enumitem}
\setlist[enumerate]{leftmargin=.5in}
\setlist[itemize]{leftmargin=.3in}

\newcommand{\lb}{\lambda}
\newcommand{\argmin}{\arg\min}

\newcommand{\st}{\textnormal{s.t.}}
\newcommand{\zz}{^\top}
\newcommand{\dist}{\textnormal{dist}}

\newcommand{\vgamma}{{\boldsymbol{\gamma}}}
\newcommand{\vxi}{{\boldsymbol{\xi}}}
\newcommand{\vzeta}{{\boldsymbol{\zeta}}}

\usepackage{mathtools}

\usepackage[normalem]{ulem} 

\newcommand{\va}{{\mathbf{a}}}

\newcommand{\ve}{{\mathbf{e}}}

\newcommand{\vg}{{\mathbf{g}}}

\newcommand{\vt}{{\mathbf{t}}}
\newcommand{\vu}{{\mathbf{u}}}

\newcommand{\vw}{{\mathbf{w}}}
\newcommand{\vx}{{\mathbf{x}}}
\newcommand{\vy}{{\mathbf{y}}}
\newcommand{\vz}{{\mathbf{z}}}

\newcommand{\cA}{{\mathcal{A}}}
\newcommand{\cB}{{\mathcal{B}}}

\newcommand{\cD}{{\mathcal{D}}}

\newcommand{\cJ}{{\mathcal{J}}}

\newcommand{\cN}{{\mathcal{N}}}

\newcommand{\cX}{{\mathcal{X}}}

\newcommand{\EE}{\mathbb{E}} 
\newcommand{\RR}{\mathbb{R}} 
\newcommand{\vzero}{\mathbf{0}} 
\newcommand{\dom}{{\mathrm{dom}}} 

\newcommand{\bc}{\begin{center}}
	\newcommand{\ec}{\end{center}}

\newcommand{\bdm}{\begin{displaymath}}
	\newcommand{\edm}{\end{displaymath}}

\newcommand{\beq}{\begin{equation}}
	\newcommand{\eeq}{\end{equation}}

\newcommand{\bfl}{\begin{flushleft}}
	\newcommand{\efl}{\end{flushleft}}

\newcommand{\bt}{\begin{tabbing}}
	\newcommand{\et}{\end{tabbing}}

\newcommand{\beqn}{\begin{eqnarray}}
	\newcommand{\eeqn}{\end{eqnarray}}

\newcommand{\beqs}{\begin{align*}} 
	\newcommand{\eeqs}{\end{align*}}  

\newtheorem{assumption}{Assumption}

\newtheorem{example}{Example}[section]

\newsiamremark{remark}{Remark}
\newsiamremark{hypothesis}{Hypothesis}
\crefname{hypothesis}{Hypothesis}{Hypotheses}
\newsiamthm{claim}{Claim}

\title{A SPIDER-type stochastic subgradient method for expectation-constrained nonconvex nonsmooth optimization\thanks{This work is supported in part by NSF grant DMS-2208394.}
        }

\author{Wei Liu\thanks{Department of Mathematical Sciences, Rensselaer Polytechnic Institute, Troy, NY
		(\email{lwdsdqqb@gmail.com, xuy21@rpi.edu}).} \and Yangyang Xu\footnotemark[2]}

\headers{Stochastic method for expectation-constrained optimization}{WEI LIU,  Yangyang Xu}

\usepackage{amsopn}
\allowdisplaybreaks[4]

\usepackage{subfig}
\begin{document}
	
		\maketitle
		\begin{abstract}	
			Many real-world problems, such as those with fairness constraints, involve complex expectation constraints and large datasets, necessitating the design of efficient stochastic methods to solve them. Most existing research focuses on cases,  {which have no constraint, or have constraints that admit an easy projection, or only have deterministic constraints.} 
			In this paper, we consider {weakly convex} stochastic optimization problems with expectation constraints, for which we build an exact penalty model.  
			We first show the relationship between the penalty model and the original problem. 
			Then on solving the penalty problem, we present 
			a SPIDER-type stochastic subgradient
			 method, which utilizes the subgradients of 
			both the objective and constraint functions, as well as the constraint function value at each iteration. 
		Under {the Slater-type constraint qualification (CQ)}, we establish an iteration complexity result of $O(\epsilon^{-4})$  to reach
			a near-$\epsilon$ stationary point of the penalized problem in expectation, matching the lower bound for such tasks. Building on the exact penalization, an $(\epsilon,\epsilon)$-KKT point of the original problem is obtained. 
            For a few scenarios, our complexity of either the {objective} 
            sample subgradient or the constraint sample function values can be lower than the state-of-the-art results in \cite{boob2023stochastic,huang2023single,ma2020quadratically} by a factor of $\epsilon^{-2}$.  
            Moreover, 
            on solving two fairness-constrained problems and {a multi-class Neyman-Pearson classification problem}, our method is significantly (up to 466 times in terms of data pass) faster than the state-of-the-art algorithms, including the switching subgradient method  in~\cite{huang2023single} and the inexact proximal point methods in~\cite{boob2023stochastic,ma2020quadratically}.
		\end{abstract}
		
		
		
		\section{Introduction}
	Stochastic optimization is fundamental to the statistical learning 
	and serves as the backbone of data-driven learning processes. 
	In this paper, we study this topic by designing first-order methods (FOMs) for solving the \textit{expectation-constrained nonconvex programming}: 
	\begin{equation}\label{eq:model}
		\min_{ \vx\in\mathcal{X}} \,\, 
			{f(\vx)}, \text{ s.t. } \underbrace{\mathbb{E}_{\xi_{\vg} \sim \mathbb{P}_{\vg}}[\vg(\vx, \xi_{\vg})]}_{=:\vg(\vx)}\leq 0. 
			\tag{P}
		\end{equation}
		Here, {$\mathcal{X}\subset \RR^d$ is a closed convex set that allows a computationally easy projection operator,  \(\vg : \RR^d \to \RR^m\) is a vector-valued function consisting of \(m\) components \(g_1, \ldots, g_m\), and \(\mathbb{P}_{\vg}\) is an unknown distribution from which we can draw independent samples. 
		We will consider both general distributions and the special case where \(\mathbb{P}_{\vg}\) is uniform over a finite sample set; for the latter we can obtain deterministic subgradient of each $g_i$ if needed. For $f$, we assume only its unbiased stochastic subgradient is available.} 
		
		Substantial research has been devoted to FOMs for solving function-constrained optimization problems. 
		Among these FOMs, most are developed for deterministic function-constrained problems or for expectation-constrained convex optimization. This is well-documented in the literature for convex scenarios~\cite{aybat2011first,bayandina2018mirror,lan2013iteration-penalty,lan2016iteration-alm,lan2020algorithms,lin2018level-ICML,liu2016iALM,xu2020primal,xu2021first-ALM,yu2017simple} and nonconvex scenarios~\cite{andreani2011sequential,andreani2018strict,cartis2014complexity,cui2025two,facchinei2021ghost,grimmer2025goldstein,jia2025first,liu2022linearly,liu2022inexact,martinez2003practical,shi2025momentum}. Only a few works (e.g.,~\cite{alacaoglu2024complexity, boob2023stochastic, huang2023single, li2024stochastic}) have studied FOMs for expectation-constrained nonconvex optimization. 
		We focus on nonconvex nonsmooth expectation-constrained problems in the form of~\cref{eq:model}. 
		Throughout this paper, we make the following  assumption.
		\begin{assumption}
			\label{ass:problemsetup}
			The following statements hold. 
			
			$\mathrm{(a)}$  The set {$\mathcal{X} \cap \{\vx: \vg(\vx) \le 0\}$} is nonempty,  and $\mathcal{X}$ is bounded.  Specifically, there exists $D>0$ such that  $\|\vx_1-\vx_2\|\leq D$ for any $\vx_1,\vx_2\in \mathcal{X}$.

			$\mathrm{(b)}$ The function  $f$  is $\rho_f$-weakly convex {on~$\cX$},  and $g_i$  is $\rho_g$-weakly convex on~$\cX$ for some $\rho_f \ge 0$, $\rho_g\ge 0$ and each $i\in[m]:=\{1,\ldots,m\}$. That is, $f+\frac{\rho_f}{2}\|\cdot\|^2$ and $g_i+\frac{\rho_g}{2}\|\cdot\|^2$ are convex {on~$\cX$} for each $i\in[m]$.
			
			$\mathrm{(c)}$ 
			The subdifferential sets of $f$ and $g_i$ are nonempty {in their respective domains} for each $i\in[m]$. In addition, it holds that $\|\vzeta_f\| \leq l_f$ and  $\|\boldsymbol{\vzeta}_{g_i}\| \leq l_g$ for all $\vzeta_f\in\partial f(\vx)$ and $\boldsymbol{\vzeta}_{g_i}\in\partial g_i(\vx)$ for any~{$\vx\in\mathcal{X}$ and any $i\in[m]$.}  
		\end{assumption}
		From items (a) and (c) in the above assumption, we have 
		\begin{equation}\label{eq:bd-f-diff}
		|f(\vx_1) -f(\vx_2)| \leq l_f D=: B_f,\text{ for any }\vx_1,\vx_2\in \mathcal{X}.
		\end{equation}
		{In item (c), $\partial$ refers to the   subdifferential; see~\cref{def: subdiff}. Item (c) holds naturally when $f$ and $\{g_i\}_{i\in[m]}$ are all Lipschitz continuous.}
		
		\vspace{1mm}
		
		\noindent\textbf{Complexity {measures}.}~~We will focus on three 
		complexity measures: {objective subgradient complexity (OGC)}, {constraint subgradient complexity (CGC)} and {constraint function value complexity (CFC)}
			, detailed as follows:
\vspace{1mm}			
		\begin{center}
			OGC and CGC denote the number of calls to the sample subgradients of \(f\) and~\(\vg\), respectively, while CFC represents the number of evaluations of the sample constraint function \(\vg(\cdot, \xi_{\vg})\).
		\end{center}
		
		\vspace{1mm}
		
		The nonsmooth expectation constraints in~\cref{eq:model} pose significant challenges for designing efficient and reliable stochastic methods, as the projection onto an expectation-constraint set is prohibitively expensive.   
		To {circumvent} the costly projection step, {several} methods such as online convex method~\cite{yu2017online}, augmented Lagrangian method (ALM)~\cite{li2024stochastic}, inexact proximal point method (IPP)~\cite{boob2023stochastic,jia2025first, ma2020quadratically} and switching subgradient method (SSG)~\cite{huang2023single,lan2020algorithms} have been explored. {Among these methods, the online convex method, ALM, and IPP all require at least an inner subsolver to address their respective subproblems.
		The online convex method~\cite{yu2017online} is developed specifically for problems with convex constraints. 
		With nonconvex constraints and under suitable regularity conditions,}  Li et al.~\cite{li2024stochastic} and Xiao et al.~\cite{xiao2024developing} demonstrate the effectiveness of the ALM applied to problem~\cref{eq:model}, yielding competitive numerical results.
		Xiao et al.~\cite{xiao2024developing} establishes convergence to a KKT point. To derive complexity results to an $(\epsilon,\epsilon)$-KKT point  (see~\Cref{def:nearkkt}), Li et al.~\cite{li2024stochastic} assume additional smoothness on~$\vg$ and exploit this structure to efficiently solve each nonconvex ALM subproblem by an accelerated method.
		{Under a certain constraint qualification (CQ), Boob et al.~\cite{boob2023stochastic}  and Ma et al.~\cite{ma2020quadratically} respectively provide convergence results and complexity analysis for stochastic IPP methods to find an {$(\epsilon,\epsilon)$-KKT point}. These methods iteratively solve strongly convex subproblems by incorporating a proximal term into the objective and constraint functions. Notably, the complexity results in~\cite{boob2023stochastic} {represent} the state-of-the-art results  
			for problem~\cref{eq:model} under nonsmooth weakly convex settings.}
		
		{The aforementioned methods} utilize optimal subroutines for solving inner subproblems to achieve low-order complexity results. Implementing them needs extra efforts to design appropriate checkable stopping criteria for inner subsolvers or requires careful tuning of the inner iteration counts to reach a satisfactory empirical performance.  
		This disadvantage is especially pronounced in stochastic settings, where the randomness complicates the verification of a stopping condition and efficient tuning. 
		{Therefore, algorithms that avoid solving subproblems by an iterative subroutine and instead update the solution directly at each iteration are often more desirable for solving expectation-constrained problems in practice.}
		{However}, such methods are still limited. Both~\cite{alacaoglu2024complexity,shi2025momentum} design FOMs based on penalty or ALM for constrained stochastic programming but they require smoothness of the objective (possibly with a proximable nonsmooth convex term) and constraint functions.
		A concurrent work~\cite{yang2025single}   studies problem~\eqref{eq:model} via an $\ell_1$ penalty-based method, under a uniform non-degeneracy condition: specifically, they assume there is a constant $\theta>0$ such that
			$
			\operatorname{dist}\left(0, \frac{1}{m}\sum_{i=1}^m \partial [g_i(\vx)]_{+}\right)\ge \theta$ for all $\vx\in \{\vx: \max_i g_i(\vx)>0\}.
			$
			While this condition holds under a strong LICQ-type assumption for smooth constraints~\cite{yang2025single}, it is significantly more restrictive in the nonsmooth setting, ruling out infeasible points with aggregated subgradients arbitrarily close to zero. Under this assumption, they derive complexity bounds of order $\tilde{\mathcal{O}}(\epsilon^{-6})$ for OGC, CGC, and CFC.	Recently, Huang and Lin~\cite{huang2023single} apply 
		 SSG~\cite{Polyak1967ssg} to problem~\cref{eq:model}.	
		The SSG is  
		easier to implement while delivering competitive numerical performance.
		Under a (uniform) Slater-type~CQ  (see~\Cref{def:2}),  
		the SSG adapts its stepsize depending on whether the current iterate is near feasible or not, switching between fixed and adaptive steps based on Polyak’s rules~\cite{Polyak1967ssg}. 
		This adaptation is crucial in maintaining near feasibility to achieve a near KKT point.  
		However, the SSG requires frequent checks for near-feasibility at generated iterates, resulting in a higher CFC and overall computational complexity compared to the IPP method in~\cite{boob2023stochastic} when solving the stochastic nonconvex problem~\cref{eq:model}. Furthermore, in the stochastic setting, Huang et al.~\cite{huang2023single} establish convergence guarantees only under the assumptions that $\vg$ is convex and satisfies a light-tail condition.
		
		Without the need {for} feasibility checking, we propose a simple 
		\texttt{S}PIDER-type \texttt{S}tochastic \texttt{S}ubgradient FOM for solving the \texttt{E}xpectation \texttt{con}strained problem~\cref{eq:model}, named \verb|3S-Econ|. At the $k$-th iteration, \verb|3S-Econ| requires unbiased stochastic subgradient $\boldsymbol{\zeta}_f^{(k)}$ of $f$ and $\boldsymbol{\zeta}_{g_i}^{(k)}$ of $g_i$, evaluated at the iterate $\vx^{(k)}$, along with an approximation $u_i^{(k)}$ of  $g_i(\vx^{(k)})$, \text{for all $i\in[m]$}.  {Using} this information, it performs the following update step:
		{\begin{equation}
			\label{eq:newupdate}
			\vx^{(k+1)} =
			\mathrm{Proj}_{\mathcal{X}}\left(\vx^{(k)} -\alpha_k \left(\boldsymbol{\zeta}_f^{(k)} + \beta \sum_{i=1}^{m} \mathrm{Proj}_{[0,1]}\left(\textstyle \frac{u_i^{(k)}}{\nu}\right) \boldsymbol{\zeta}_{g_i}^{(k)} \right)\right),
		\end{equation}}where {$\beta>0$ is a 
		penalty parameter, $\nu>0$ is a 
		smoothing parameter}, and $\alpha_k>0$ denotes the step size at iteration $k$.	We will show that \verb|3S-Econ| is not only  easy to implement but also enjoys lower-order complexity than the best-known results to produce an $(\epsilon, \epsilon)$-KKT point of problem~\cref{eq:model}; 
		see~\cref{tab:1}.

		\subsection{Applications}
		Many applications can be formulated into the nonconvex nonsmooth problem~\cref{eq:model}. We highlight two representative examples: Neyman-Pearson classification~\cite{rigollet2011neyman,yan2022adaptive} and fairness-constrained classification~\cite{lin2022complexity}.  
		\begin{example}
			\label{exam:1}
			\textbf{ Neyman-Pearson classification.} 
			{Consider a classification problem involving training data partitioned into  $M$ classes, denoted by \mbox{$\mathcal{D}_i \subseteq \mathbb{R}^d$} for $i\in[M]$.
			Let $\phi $ be a loss function, $\{\vx_i\}_{i\in[M]}$ denote model variables,
			$D>0$, and $\kappa$ be a prescribed threshold. 
			The multi-class  Neyman-Pearson classification aims to minimize the average pairwise loss for a specified class (e.g., class $\cD_1$),
			 while controlling the average pairwise losses for all other classes below a threshold. Formally, the optimization problem is given by
				\begin{align}
				\label{eq:example1}
					&	\min _{\left\|\mathbf{x}_i\right\| \leq D, i\in[M]}   \frac{1}{\left|\mathcal{D}_1\right|} \sum_{l>1} \sum_{\va_1 \in \mathcal{D}_1} \phi\left(\mathbf{x}_1^{\top} \va_1-\mathbf{x}_l^{\top} \va_1\right), \\
					& \text { s.t. } \frac{1}{\left|\mathcal{D}_i\right|} \sum^M_{l=2, l \neq i} \sum_{\va_i \in \mathcal{D}_i} \phi\left(\mathbf{x}_i^{\top} \va_i-\mathbf{x}_l^{\top} \va_i\right) \leq \kappa, \quad i=2,3, \ldots, M.
						\notag
				\end{align} 
			Clearly,~\cref{eq:example1} is a special instance of \cref{eq:model} with $f$ having a finite-sum structure,   and~$\mathbb{P}_{\vg}$ corresponds to a uniform distribution over sample sets associated with $\{\cD_i\}_{i=2,\ldots,M}$.}
		\end{example}
		
		\begin{example}
			\label{exam:2}
			\textbf{Fairness-constrained classification}. This application aims to ensure that classification models treat different groups equitably, avoiding biased outcomes. Traditional models focus on maximizing accuracy, which can lead to unfair performance across different groups of data, such as males and females. Fairness constraints are applied to balance outcomes for all groups, using principles like demographic parity~\cite{agarwal2018reductions,li2024model} or equalized odds~\cite{hardt2016equality}. Achieving fairness requires trade-offs, as improving fairness can reduce overall accuracy. This research area is particularly important in fields like healthcare and criminal justice~\cite{chouldechova2017fair}. 
			
			Let us consider 
			a classification problem with receiver operating characteristic (ROC)-based fairness~\cite{vogel2021learning}. {Suppose the data are partitioned into three parts: a dataset \(\mathcal{D} = \left\{ (\mathbf{a}_i, b_i) \right\}_{i=1}^n\), and two labeled subgroups \(\mathcal{D}_{\mathrm{p}}=\left\{\left(\mathbf{a}^{\mathrm{p}}_i, b^{\mathrm{p}}_i\right)\right\}_{i=1}^{n_{\mathrm{p}}}\) and \mbox{\(\mathcal{D}_{\mathrm{u}}=\left\{\left(\mathbf{a}^{\mathrm{u}}_i, b^{\mathrm{u}}_i\right)\right\}_{i=1}^{n_{\mathrm{u}}}\)}, defined according to a binary protected attribute (e.g., race or gender). Here, group~$\mathcal{D}_{\mathrm{p}}$ corresponds to the minority subgroup with historically less favorable treatment, while group $\mathcal{D}_{\mathrm{u}}$ represents the majority subgroup.}
			The fairness-constrained problem can be formulated as   
			\begin{equation} 
				\min_{\vx: \|\vx\|\leq D} \Phi_{\Theta}(\vx), \,\, \text{\st}\,\, \Phi(\vx)\leq \Phi^*+\kappa_1,
			\end{equation}
			where $\Phi_{\Theta}(\vx):=\max_{\theta \in\Theta} \left| \frac{1}{n_{\mathrm{p}}} \sum_{i=1}^{n_{\mathrm{p}}} \mathbb{I}\left(\mathbf{x}^{\top} \mathbf{a}_i^{\mathrm{p}}\geq \theta\right)-\frac{1}{n_{\mathrm{u}}} \sum_{i=1}^{n_{\mathrm{u}}} \mathbb{I}\left(\mathbf{x}^{\top} \mathbf{a}_i^{\mathrm{u}}\geq\theta\right)  \right|$ is a ROC-based fairness, $\Phi^*:= \min_{ \vx }\Phi(\vx)$ with $\Phi(\vx):= \frac{1}{n} \sum_{i=1}^n\left[1-b_i \cdot \vx^{\top} \va_i\right]_{+}$,  $\Theta$ is a finite threshold set, and $\kappa_1>0$ is a slackness parameter. Using the sigmoid function $\sigma(\cdot)$  
			as a  continuous relaxation of the indicator function, we reformulate the problem as
			\begin{equation}
				\label{eq:fair1}
				\min_{\vx: \|\vx\|\leq D} \Psi_{\Theta}(\vx), \,\, \text{\st}\,\, \Phi(\vx)\leq \Phi^*+\kappa_1,
			\end{equation}
			where 
			$ \Psi_{\Theta}(\vx) : = \max_{\theta \in\Theta} \left| \frac{1}{n_{\mathrm{p}}} \sum_{i=1}^{n_{\mathrm{p}}} \sigma\left(\mathbf{x}^{\top} \mathbf{a}_i^{\mathrm{p}}-\theta\right)-\frac{1}{n_{\mathrm{u}}} \sum_{i=1}^{n_{\mathrm{u}}} \sigma\left(\mathbf{x}^{\top} \mathbf{a}_i^{\mathrm{u}}-\theta\right)  \right|$.  
			Since the max operator respects the weak convexity,  
			\cref{eq:fair1} is also a specific instance of~\cref{eq:model}.  
	\end{example}
	
	\subsection{Tools: exact penalization, smoothing, and variance reduction}\label{sec:motivation}
	A key tool in our approach is the exact penalization. 
    This tool
	transforms problem~\cref{eq:model} into {a problem without function constraints}. Notably, when the penalty parameter is sufficiently large, the stationary point of the penalized problem {can} align with the KKT point of the original problem.
	Building on the exact penalization, we then design a stochastic FOM to solve the resulting problem. 
	
	{Let the penalty function $h(\vz)=\ve\zz[\vz]_+$, {where \(\ve \in \RR^m\) is the all-ones vector and 
	\([\vz]_+ := (\max\{z_1, 0\}, \ldots, \max\{z_m, 0\})^\top\) for any \(\vz \in \RR^m\).} Then an exact penalty reformulation~\cite{bertsekas1997nonlinear,nocedal2006numerical,di1989exact,wang2017penalty,yang2025single,zhang2022stochastic} of problem~\cref{eq:model}  is}
	\begin{equation}\label{eq:penaltymodel}
		\textstyle
		\min_{ \vx\in\mathcal{X}} \,\, F(\vx):=f(\vx) + \beta g_+(\vx), \text{ with }g_+(\vx) = h(\vg(\vx)) = \sum_{i=1}^m [g_i(\vx)]_+,
	\end{equation}
	where $\beta>0$ is the penalty parameter. 
	The exact penalization results regarding global minimizers are established in~\cite{nocedal2006numerical}.
	{Note that \(\sum_{i=1}^{m}\mathrm{Sign}([g_i(\mathbf{x})]_+) \boldsymbol{\zeta}_{g_i} \) belongs to \(\partial g_+(\mathbf{x}) \), where \( \boldsymbol{\zeta}_{g_i} \in \partial g_i(\mathbf{x} )\),}
	 and $\mathrm{Sign}:\RR\mapsto \{-1,0,1\}$  refers to the sign operator.  
	Due to the stochastic nature of \( \vg \), it is impractical to compute the exact value of  
	\( \mathrm{Sign}([g_i(\mathbf{x})]_+) \) for each $i$. 
	Approximating this value is also challenging, because $\mathrm{Sign}$ is discontinuous, making it impossible for the discrepancies between \( \mathrm{Sign}([g_i(\mathbf{x})]_+) \) and \( \mathrm{Sign}([g_i(\mathbf{x} , \mathcal{B})]_+) \) to be consistently small enough. 
	Here, $\mathcal{B}$ is a sample set of $\xi_{\vg}$ over the distribution $\mathbb{P}_{\vg}$, and  we define
	\begin{equation}\label{eq:g-approx}
		g_i(\mathbf{x}, \mathcal{B})=\frac{1}{|\mathcal{B}|}\sum_{\xi_{\vg}\in \mathcal{B}} g_i(\vx, \xi_{\vg}), \forall \, i\in [m],\quad \vg(\vx,\cB) =\left(g_1(\mathbf{x}, \mathcal{B}), \ldots, g_m(\mathbf{x}, \mathcal{B})\right)^\top,
	\end{equation} 
	with $|\mathcal{B}|$ denoting the cardinality of $\mathcal{B}$.  
	Consequently, solving problem~\cref{eq:penaltymodel} directly by a stochastic method is challenging. 
	
	To resolve  this issue,  {we employ the Huber-type\footnote{{The smoothing technique is a powerful tool for handling nonsmooth problems, and we refer interested readers to~\cite{burke2013epi,burke2017epi,chen2012smoothing} for further details and illustrative examples. Specifically, the function \([\,\cdot\,]_+\) admits various smooth approximations, including the Softplus function~\cite{li2023softplus,li2025new}, the square-root function~\cite{chen1996class}, and the log-sum-exp function~\cite{chen2012smoothing,liu2025stochastic}. Our complexity results remain valid as long as the smoothing function \(h^\nu\) is Lipschitz continuous, its gradient has a Lipschitz constant that grows linearly with \(1/\nu\), and the composition \(h^\nu(\vg(\cdot))\) is weakly convex. The mentioned Huber-type, Softplus and log-sum-exp smoothing functions  satisfy all these conditions. We choose the Huber-type smoothing because its gradient takes the particularly simple form \(\mathrm{Proj}_{[0,1]}(\cdot/\nu)\), which facilitates a clean and efficient algorithmic implementation.}} smoothing~\cite{chen2016penalty,nedich2023huber,tatarenko2021smooth} function~$h^\nu$ given as follows,}
     \begin{align*} 
		h^\nu(\vz) & :=\sum_{i=1}^m H^\nu(z_i), \text{ with } H^\nu(z_i) = \max _{0\leq y\leq 1} \left[\langle y, z_i \rangle-\frac{\nu}{2} y^2\right] 
=
		\left\{
		\begin{aligned}
			\frac{z_i^2}{2\nu},\quad &\text{ if } 0\leq z_i\leq \nu,\\
			z_i -\frac{\nu}{2}, &\text{ if }  z_i>\nu,\\
			0 , \quad\,\,\,&\text{ otherwise}, 
		\end{aligned}\right.
	\end{align*}
where $\nu>0$ is a smoothing parameter. Then instead of directly working on \cref{eq:penaltymodel}, we present a new penalty model of problem~\cref{eq:model}:
	\begin{equation}
		\label{eq:smoothed}
		\tag{PP}
		\min_{ \vx \in\mathcal{X} } \,\, F^{\nu}(\vx)=f(\vx) + \beta h^{\nu} (\vg(\vx)).
	\end{equation}
	Since $h^{\nu}$ is smooth and {$\nabla h^\nu(\vz)  = 
	\mathrm{Proj}_{[0,1]^m}\left(\frac{\vz}{\nu}\right)$}, we can access a subgradient of $ h^{\nu} (\vg(\vx))$ via $\sum_{i=1}^{m}\mathrm{Proj}_{[0,1]}\left(\frac{g_i(\vx)}{\nu}\right) \boldsymbol{\zeta}_{g_i} $, where \( \boldsymbol{\zeta}_{g_i} \in\partial g_i(\mathbf{x} )\) for each $i\in [m]$. 
    In addition, as $\mathrm{Proj}_{[0,1]}$ is Lipschitz continuous, the difference between $\mathrm{Proj}_{[0,1]}\left(\frac{g_i(\vx)}{\nu}\right) \boldsymbol{\zeta}_{g_i} $ and $\mathrm{Proj}_{[0,1]}\left(\frac{u_i}{\nu}\right) \boldsymbol{\zeta}_{g_i} $ 
    can be well controlled if $u_i$ is a good approximation of $g_i(\vx)$. 
	
    For each $i\in[m]$ and each~$k$, we apply a SPIDER-type (or SARAH)~\cite{fang2018spider, nguyen2017sarah} variance reduction technique to estimate $g_i(\vx^{(k)})$ to obtain $u_i^{(k)}$ and perform the update in~\cref{eq:newupdate}. The variance reduction tool, together with the smoothing tool, enables us to achieve better complexity results for our method than state-of-the-art results to produce {an \mbox{$(\epsilon,\epsilon)$-KKT}} point of \cref{eq:model}. 
	
	\subsection{Contributions}
	We address key challenges in solving nonconvex nonsmooth expectation-constrained optimization problems. Our contributions are threefold.
	
	First,
		to address the expectation constraint involved in~\cref{eq:model}, we introduce a  Huber-type  smoothing penalty function~$h^{\nu} \circ \vg$. Leveraging this function, we propose solving problem~\cref{eq:smoothed}.   
		For any given $\epsilon > 0$, we demonstrate that, under easily verifiable conditions,  
        any \mbox{near-$\epsilon$} stationary point of problem~\cref{eq:smoothed} (in expectation) is an $(O(\epsilon), O(\epsilon))$-KKT point of problem~\cref{eq:model} (in expectation).
		This property   enables us to pursue an $(\epsilon, \epsilon)$-KKT point of problem~\cref{eq:model} by finding a near-$O(\epsilon)$ stationary point of problem~\cref{eq:smoothed}. 
		Our framework is well-suited for many applications such as problem~\cref{eq:fair1},  
        which involves a nonconvex objective and convex expectation constraints satisfying the required assumptions. 
	
	Second, we exploit the structure of problem~\cref{eq:smoothed} and propose a  SPIDER-type stochastic FOM to solve it, with updates given in~\cref{eq:newupdate}. At each iteration, the SPIDER technique~\cite{fang2018spider} is employed to generate an approximation \( u_i^{(k)} \) of the constraint function value $g_i(\vx^{(k)})$ for each $i\in[m]$.  
	Using the tool of Moreau envelope, we demonstrate that the proposed method can, in expectation, find a near-$\epsilon$ stationary point of problem~\cref{eq:smoothed} within \(O(\epsilon^{-4})\) iterations. This result matches the known lower bound for such tasks~\cite{jordan2023deterministic}, 
    {highlighting} the efficiency of our approach. Combining this with the theory of exact penalization, we further analyze the complexity of obtaining an \((\epsilon, \epsilon)\)-KKT point of problem~\cref{eq:model} in expectation.

	Based on the structure of \(\vg\),	 we categorize the complexity results of our method as follows: 
	\begin{enumerate}
	\item {A generic stochastic \(\vg\):} OGC is \(O(\epsilon^{-4})\), CGC is \(O(\epsilon^{-4})\), and CFC is \(O(\epsilon^{-6})\) .
	\item {A finite-sum structured\footnote{
    {Here,} 
    we assume that \(\mathbb{P}_{\vg}\) is a uniform distribution over a finite set of size \(N\), where \(N < \epsilon^{-4}\).} \(\vg\):} 
	OGC is \(O(\epsilon^{-4})\), CGC is \(O(\epsilon^{-4})\),  and CFC is \(O(\sqrt{N}\epsilon^{-4})\).
	\end{enumerate}
    {When $\vg$ is generic stochastic, our complexity result is lower than state-of-the-art results by a factor of $\epsilon^{-2}$ in either OGC or CFC; when $\vg$ enjoys a finite-sum structure, the highest order among our results of OGC, CGC, and CFC is \(O(\sqrt{N}\epsilon^{-4})\), representing an improvement over the previously best-known $O(\min\{N\epsilon^{-4},\epsilon^{-6}\})$ bound.}
    Detailed comparisons are provided in~\cref{tab:1}.  

	Third, our approach does not need to solve subproblems by an iterative subroutine and thus  
	eliminates the need for tuning inner iterations, distinguishing our method from  those in~\cite{boob2023stochastic,li2024stochastic,ma2020quadratically}. This simplification reduces computational overhead and makes the method more accessible for 
    {real-world} applications. In numerical experiments, we test both the deterministic and stochastic versions of our method, for three problems across three datasets.
	Notably, the deterministic version is significantly faster than state-of-the-art methods, and the stochastic version achieves even more impressive speedups.  
	\begin{table}[h]
	\centering
	\scriptsize
	\begin{tabular}{ccccccc}
		\hline Methods &   RIS$^\P$ &  structure of $\vg$  & CQ on $\vg$ & 
		OGC & CGC 
		&  CFC 
		\\ 
		\hline
		\multirow{4}{*}{IPP~\cite{boob2023stochastic}}	& \multirow{4}{*}{{yes}} & stochastic convex&  MFCQ, and  &  {$O(\epsilon^{-6})$}  & \multicolumn{2}{|c}{$O(\epsilon^{-6})$}  
		\\
		& & finite-sum$^\dagger$ convex  &   Slater's condition  & {$O(\epsilon^{-4})$} & \multicolumn{2}{|c}{$O(\min\{N\epsilon^{-4},\epsilon^{-6}\})$} 
		\\
		  \cline{3-7}
		&  & stochastic WC$^\P$& 
MFCQ, and 		& {$O(\epsilon^{-6})$}  & \multicolumn{2}{|c}{$O(\epsilon^{-6})$}
		\\
		& & finite-sum$^\dagger$ WC  & {strong feasibility$^*$}  & {$O(\epsilon^{-4})$}  & \multicolumn{2}{|c}{$O(\min\{N\epsilon^{-4},\epsilon^{-6}\})$}\\
		\hline
		\multirow{3}{*}{SSG~\cite{huang2023single}}
		& \multirow{3}{*}{{no}} &stochastic convex& \multirow{3}{*}{Slater-type CQ$^*$} & $O(\epsilon^{-4})$ & $O(\epsilon^{-4})$  & {$O(\epsilon^{-8})$}
		\\
		& & finite-sum$^\dagger$ convex  &   & $O(\epsilon^{-4})$ & $O(\epsilon^{-4})$ & $O(N\epsilon^{-4})$\\
		&  & WC&  &   \multicolumn{3}{c}{none$^\ddagger$}
		\\
		\hline
		\multirow{5}{*}{Ours}
		& \multirow{5}{*}{{no}} &  stochastic convex& \multirow{2}{*}{Slater-type CQ$^*$} 
		& \multirow{2}{*}{$O(\epsilon^{-4})$}  &  $O(\epsilon^{-4})$ & $O(\epsilon^{-6})$
		\\
		& & finite-sum$^\dagger$ convex  &  
		& & $O(\epsilon^{-4})$ & $O(\sqrt{N}\epsilon^{-4})$
            \\
            \cline{3-7}
		&  & stochastic WC&  {Slater-type CQ, with}  & \multirow{2}{*}{$O(\epsilon^{-4})$}  &  $O(\epsilon^{-4})$ & $O(\epsilon^{-6})$
		\\
		&  & finite-sum$^\dagger$ WC& $B_g > m\rho_g  D^2$$^*$  &    & $O(\epsilon^{-4})$ & $O(\sqrt{N}\epsilon^{-4})$
		\\ 
		\hline
	\end{tabular} 
	\caption{ Comparison {of} complexity results among our method and existing ones for obtaining an $(\epsilon,\epsilon)$-KKT point of problem~\cref{eq:model} with a stochastic {nonsmooth} weakly convex objective function $f$.
		$^\P${RIS is short for ``Require Inner Solver'', and WC is for ``weakly convex''.}
		$^*${Strong feasibility implies the Slater-type CQ with any $B_g>0$; see~\Cref{prop:regurela}(ii). 
         {The Slater-type CQ is directly assumed in \cite{huang2023single} when $\{g_i\}$ are weakly convex. {Under the setting where $g_i$ is convex for each $i\in [m]$, Slater's condition is equivalent to the \mbox{Slater-type CQ}; see~\Cref{prop:regurela}(i).}}
		$^\dagger$In the finite-sum setting, we assume that \(\mathbb{P}_{\vg}\) is a uniform distribution over a finite set of size \(N < \epsilon^{-4}\). 
		$^\ddagger${When~\(\{g_i\}_{i\in[m]}\) are weakly convex, Huang et al.~\cite{huang2023single} provide complexity results only for the case where both \(f\) and~\(\vg\) are deterministic, excluding the setting where \(f\) is stochastic.}} 
	}
	\label{tab:1} 
	\end{table}

	\subsection{Notations and definitions}
	    {We use $\dom(\cdot)$ to denote the domain of a function.}
	$[M]$ denotes $\{1,\ldots,M\}$ for a positive integer \(M\). 
	 {The $0$-$\infty$ indicator function of a set $\cX$ is denoted as $\iota_\cX$.}
	The \(l_2\)-norm is {written as} 
    \(\|\cdot\|\).  {\({\mathrm{co}(\cdot)}\) and $\mathrm{relint}(\cdot)$  represent the closed convex hull and relative interior of a set, respectively.
	$\mathcal{N}_{\mathcal{X}}(\vx)$   refers to the normal cone  of a closed convex set \(\mathcal{X}\) at~$\vx$.}
	We use \(\mathrm{Proj}\), \(\mathrm{Prob}\), and \(\mathrm{dist}\) to denote operators of the projection, the probability, and the distance, respectively.
	
	{
	\begin{definition}[subdifferential]\label{def: subdiff}
	For a convex function $f$, its subdifferential at $\vx\in \dom(f)$ is defined as~\cite{clarke1990optimization}
		\begin{align*}
			\partial f(\vx) = \big\{\vxi: 
			f(\vx') - f(\vx) \geq \langle \vxi, \vx' - \vx \rangle \ \text{ for all } \vx' \in \dom(f)\big\}.
		\end{align*}
		The subdifferential of a $\rho$-weakly convex function $f$ at $\vx\in \dom(f)$ is given by
		$$\partial f(\vx) := \textstyle \partial \big(f(\cdot) + \frac{\rho}{2}\|\cdot\|^2\big)(\vx) - \rho \vx.$$
	\end{definition}}
	
	\begin{definition}[Moreau Envelope~\cite{drusvyatskiy2019efficiency}]\label{def:weak-cvx} 
	For a \(\rho_{\phi}\)-weakly convex function \(\phi\), its Moreau envelope and proximal mapping for any \(t \in (0, \frac{1}{\rho_{\phi}})\) {on $\mathcal{X}$}  are defined by
	\mbox{$
	\phi_t({\vx})=\min_{\vz \in\mathcal{X}}\left\{\phi({\vz})+\frac{1}{2 t}\|{\vz}-{\vx}\|^2\right\}
	$, $\operatorname{Prox}_{t \phi}({\vx}) =\argmin_{\vz\in\mathcal{X} }\{\phi({\vz})+\frac{1}{2 t}\|{\vz}-{\vx}\|^2\}.	
	$}
	\end{definition}

	We give the definition of stationarity as follows.  In deterministic cases, where randomness is not involved, the expectation operator \(\mathbb{E}\) should be removed from the definition below.
	 {\begin{definition}[stationarity]
	\label{def:nearkkt}
	Let $\epsilon>0$. Suppose that $\{\psi_i\}_{i=1}^m$ are all $\rho_\psi$-weakly convex and $\phi$ is $\rho_\phi$-weakly convex {on $\cX$}.
	\\[1mm]
	$\mathrm{(a)}$ $\vx^*$ is called a \textit{near-$\epsilon$ stationary point} of $\min_{\vx\in \cX}\phi(\vx)$  in expectation, if $\vx^*\in\cX$, and there exists $0<t<\frac{1}{\rho_{\phi}}$, such that 
	$
	\mathbb{E}[\|\nabla \phi_{t}(\vx^*)\|^2] 
	\leq~\epsilon^2.
	$
	\\[1mm]
	$\mathrm{(b)}$ $\vx^*$ is called an \textit{$\epsilon$-KKT point} of 
	$\min_{\vx\in\mathcal{X}, \boldsymbol{\psi}(\vx)\leq 0} \phi(\vx)$ in expectation, if 
	$\vx^*\in\mathcal{X}$ and there exists $\boldsymbol{\lb}_{\vx^*}\ge0$ such that  $\mathbb{E}\left[\psi_i(\vx^*)\right]\leq \epsilon,$  \mbox{$\mathbb{E}\left[| ({\lb}_{\vx^*})_i {\psi}_i(\vx^*)|\right]\leq \epsilon$} for all $i\in[m]$, and
	$
		\mathbb{E}\left[\dist(\vzero,\partial \phi(\vx^*) + \sum_{i=1}^{m} (\lb_{\vx^*})_i \partial \psi_i(\vx^*)  + \mathcal{N}_{\mathcal{X}}(\vx^*))\right] \leq \epsilon.
	$ 
	\\[1mm]
	$\mathrm{(c)}$ $\vx^*$ is called an \textit{$(\text{{$\delta$}}, \epsilon)$-KKT point} of 
	$\min_{\vx\in\mathcal{X},\boldsymbol{\psi}(\vx)\leq 0} \phi(\vx)$ in expectation, if 
	\mbox{$\vx^*\in\mathcal{X}$} and $\mathbb{E}[\|\vx^*-\widetilde{\vx}^*\|]\leq \text{{$\delta$}}$ with $\widetilde{\vx}^*$ being an $\epsilon$-KKT point of 
	$\min_{\vx\in\mathcal{X}, \boldsymbol{\psi}(\vx)\leq 0} \phi(\vx)$ in expectation.  
	\end{definition}
}

\begin{remark}
	We choose $\delta=\epsilon$ in Definition~\ref{def:nearkkt}(c) to enable a direct comparison with prior work. This choice is standard and can be justified via near-stationarity arguments based on the Moreau envelope~\cite{davis2018stochastic,drusvyatskiy2019efficiency}. In particular, the gradient norm of the Moreau envelope directly quantifies the proximity between a computed point and its associated proximal projection, thereby naturally linking the near-stationarity tolerance to the tolerances in feasibility and optimality conditions. With this choice, the complexity analysis cleanly yields an $(\epsilon,\epsilon)$-KKT point, consistent with the definitions and complexity targets in the closely related literature, e.g.,~\cite{huang2023single,ma2020quadratically} (which target $(\epsilon,\epsilon)$-KKT points) and~\cite{boob2023stochastic} (which considers an \mbox{$(\epsilon,\mathcal{O}(\epsilon))$-KKT} point).
\end{remark}

	\subsection{Organization}
	
	The rest of this paper is organized as follows.~\Cref{sec:exact} introduces the exact penalty models. In~\Cref{sec:algorithm}, we propose a   SPIDER-type stochastic subgradient method and establish its iteration complexity results. Numerical results are presented in~\Cref{sec:numer}, showcasing the performance of our method. Finally,~\Cref{sec:conclu} concludes the paper.
	
			\section{Model analysis}	
		\label{sec:exact}
		In this section, we  
		establish the relationships among models~\cref{eq:model},~\cref{eq:penaltymodel} and~\cref{eq:smoothed}, which are presented in the following diagram for {the convenience of readers.} 
        Details are given in the mentioned lemma and theorem.
		\begin{equation*}\label{eq:relation}
			\small
			\boxed{\begin{aligned}
					&  \quad\quad \quad\quad\quad \,\,
					\text{${\epsilon}$-stationary point of~\cref{eq:penaltymodel}  } \\
					&\textbf{\cref{lem:2}: }\quad\quad  \quad\Downarrow\\
					& \quad\quad\quad\quad\quad\quad 
					\text{$\epsilon$-KKT point of~\cref{eq:model}}   \\
			\end{aligned}}\quad 
			\boxed{\begin{aligned}
					& \quad\quad\quad\text{near-$\epsilon$ stationary point of }~\cref{eq:smoothed} \,\:\quad  \\
					&\textbf{\Cref{thm:main}: }\quad\quad\quad\quad\,\Downarrow\\
					&\quad\quad\quad\text{$(O(\epsilon), O(\epsilon))$-KKT point of problem~\cref{eq:model}}
			\end{aligned}}
		\end{equation*} 

		\subsection{Regularity conditions}\label{sec:regurela}
		 {To derive the exact penalization result, we assume the following (uniform)  Slater-type CQ, {which holds for tested instances of Examples~\ref{exam:1} and \ref{exam:2} in \Cref{sec:numer}.} 
			\begin{assumption}[(uniform) Slater-type CQ~\cite{huang2023single,ma2020quadratically}]
					\label{def:2}
					There exist \(B_g>~0 \), \(B >0\), and \(\overline{\rho} > \rho_g \ge 0\), such that for 
						any \(\vx{ \in\cX}\) satisfying $0<g_+(\vx)\leq B_g$ with $g_+$ defined in \eqref{eq:penaltymodel},  it holds
						\begin{equation}\label{eq:sl-type-cq}
						\max_{i\in[m]}  g_i(\mathbf{y}) + \frac{\overline{\rho}}{2}\|\mathbf{y} - \mathbf{x}\|^2 \leq-B,\text{ for some }\mathbf{y}\in\mathrm{relint}(\mathcal{X}).
						\end{equation}
				\end{assumption}
 

{		
		Below we give the relationships among the Slater-type CQ, Slater's condition and the strong feasibility condition assumed in~\cite{boob2023stochastic},  where the latter two conditions are stated as follows:
		\begin{align}
			\label{eq:sla}
			\textstyle
			&\text{Slater's condition: } \exists\,\vy_{\mathrm{feas}}\in\mathrm{relint}(\cX) \text{ such that } \max_{i\in[m]} g_i(\vy_{\mathrm{feas}})<0;
			\\\textstyle
			\label{eq:str-feas}
			&\text{strong feasibility: }\exists\, \vy_{\mathrm{feas}} \in \mathcal{X} \text{ such that } \max_{i\in[m]} g_i(\vy_{\mathrm{feas}}) \le -2 \rho_g D^2.
		\end{align}
		
		\begin{proposition}
			\label{prop:regurela}
			Under~\cref{ass:problemsetup}, the following statements hold:
			\begin{itemize}
				\item[$(\mathrm{i})$] If $\{g_i\}_{i\in[m]}$ are   convex, Slater's condition and the Slater-type CQ are equivalent.
				\item[$(\mathrm{ii})$] If the strong feasibility condition holds, so does the Slater-type CQ. 
			\end{itemize}
		\end{proposition}
		
		\begin{proof}
		(i) 
			Suppose Slater's condition holds. Let $
			\vy=\vy_{\mathrm{feas}},$
			$
			\overline{\rho}=\frac{-\max_{i\in[m]} g_i(\vy_{\mathrm{feas}})}{D^2},$ $
			B=\frac{-\max_{i\in[m]} g_i(\vy_{\mathrm{feas}})}{2}$, and $B_g$ be any positive constant. 
			Then \cref{eq:sl-type-cq} trivially holds with the chosen $\vy$, $\overline{\rho}$, $B$, and any $\vx\in \cX$. Thus Slater-type CQ holds 
			since $\rho_g=0<\overline{\rho}$.
			
			Conversely, if the Slater-type CQ holds, then by definition there exist $\vy\in\mathrm{relint}(\cX)$ and $B>0$ such that
			$
			\max_{i\in[m]} g_i(\vy) \le -B <0,
			$
			which immediately implies Slater's condition. Hence, we 
			establish the equivalence.
			
			(ii) Because $\cX\subset \RR^d$ is nonempty convex and $\vy_{\mathrm{feas}} \in \mathcal{X}$, there must exist $\vy\in\mathrm{relint}(\cX)$ such that $\|\vy - \vy_{\mathrm{feas}}\|\leq \frac{\rho_g D^2}{2l_g}$. 
			By the $l_g$-Lipschitz continuity of each $g_i$ from \cref{ass:problemsetup}(c), it holds $g_i(\mathbf{y}) - g_i(\mathbf{y}_{\mathrm{feas}}) \le l_g \|\vy - \vy_{\mathrm{feas}}\| \le \rho_g D^2/2,\, \forall\, i\in [m]$. Thus \mbox{$\max_{i\in[m]}g_i(\mathbf{y}) \leq \max_{i\in[m]}g_i(\mathbf{y}_{\mathrm{feas}})  + \rho_g D^2/2$}.
			Now let $\overline{\rho} = 2\rho_g$. We have for any $\vx\in\cX$ that 
			\begin{align*}
				\max_{i\in[m]} g_i(\mathbf{y}) + \frac{\overline{\rho}}{2}\|\mathbf{y} - \mathbf{x}\|^2 & \le \max_{i\in[m]} g_i(\mathbf{y}_{\mathrm{feas}}) + \frac{\rho_g D^2}{2} +  \rho_g\|\mathbf{y} - \mathbf{x}\|^2  \\ &\leq -2 \rho_g D^2+ \frac{\rho_g D^2}{2}+  \rho_g D^2=-\frac{\rho_g D^2}{2}.
			\end{align*}
			 Hence, the Slater-type CQ holds with $B=\frac{\rho_g D^2}{2}$, $\overline{\rho} = 2\rho_g$, and any $B_g>0$.			
		\end{proof}
		We highlight that the strong feasibility condition in \cref{eq:str-feas} ensures the \mbox{Slater-type CQ} with any $B_g>0$. Later in the proof of~\cref{thm:main}, 
		we assume \mbox{Slater-type CQ} with $B_g > m\rho_g  D^2$, which is still implied by the strong feasibility condition. 
		However, the converse does not hold in general. The following example illustrates that 
		a   \mbox{Slater-type CQ} with $B_g > m\rho_g  D^2$ does not imply the strong feasibility condition.
		\begin{example}
			Consider the nonconvex function  
			$
			g(x) = -6+8x+x^2 - x^3$ over the set $\mathcal{X} = [0,1].
			$ Since $g''(x) = 2 - 6x$, it holds $g''(x) \ge -4,\, \forall x \in \cX$.
			Then $g$ is $4$-weakly convex on $\cX$.
			Take $y=0.01\in \mathrm{relint}(\cX)$. It is straightforward to verify that 
			$g(y) + \frac{5}{2}|x-y|^2 \leq -3$ for any $x\in\mathcal{X}$. Thus,
			the Slater-type CQ holds with parameters  
			$
			B_g = 1000,   B = 3$,    and $\overline{\rho} = 5 > 4 = \rho_g.
			$
			Thus, we have \(B_g > \rho_g  m D^2\), where $m=1$, and \(D =  1\).
			However, the strong feasibility condition fails because  it holds that 
			$
			\min_{x \in [0,1]} g(x) = -6 \; > \; -8 = -2 \rho_g D^2.
			$ 
		\end{example}
%
}	
		\subsection{Exact penalization with penalty function $[\cdot]_+$}
		\label{sec:exact1}
		In this subsection, we study the relationship between \cref{eq:model} and \cref{eq:penaltymodel}. 
        {Since each $g_i$  is $\rho_g$-weakly convex for $i\in[m]$, and $f$  is $\rho_f$-weakly convex,   $F $ in \cref{eq:penaltymodel} is $\rho$-weakly convex with $\rho := \rho_f + \beta m\rho_g$.}
		 
		First, we present a key lemma as follows, whose proof is given in \cref{appen:1}.
		\begin{lemma}
			\label{lem:necessary}
			Under~\Cref{ass:problemsetup} and \cref{def:2}, 
			let $\theta = \sqrt{2B (\overline{\rho}-\rho_g)}$ and $\vx\in\mathcal{X}$ satisfy  $0< g_+(\vx) \leq B_g$. Given    $\mathcal{J}_1 \subset [m]$ such that $\mathcal{J}_1 \cap \{i : g_i(\vx) \geq 0 \} \neq \emptyset$,  let $\mathcal{J}_2 = [m] \setminus \mathcal{J}_1$ and choose $\{\lambda_i\}_{i \in \mathcal{J}_2} \subset \mathbb{R}_+$ satisfying $\lambda_i g_i(\vx) \geq 0$ for all $i \in \mathcal{J}_2$. Then, 
				\[
				\min_{ \{\boldsymbol{\zeta}_{g_i} \in \partial [g_i(\vx)]_+ \}_{i \in \mathcal{J}_1}, \{\boldsymbol{\zeta}_{g_i} \in \partial g_i(\vx)\}_{i \in \mathcal{J}_2}, \vu \in \mathcal{N}_{\mathcal{X}}(\vx)} \left\| \sum_{i \in \mathcal{J}_1} \zeta_{g_i}  + \sum_{i \in \mathcal{J}_2 } \lambda_i \zeta_{g_i} + \vu\right\| \geq \theta.
				\] 
		\end{lemma}
		
		Now, we establish the relationship between the $\epsilon$-stationary point of problem~\cref{eq:penaltymodel} and  the {$\epsilon$-KKT point} of~\cref{eq:model}.
		\begin{lemma}
			\label{lem:2}
			Suppose~\Cref{ass:problemsetup} and~\Cref{def:2} hold.
			Given  $\epsilon>0$, let the penalty parameter $\beta>\frac{1}{\theta} (l_f +\epsilon)$ in \cref{eq:penaltymodel} with $\theta = \sqrt{2B (\overline{\rho}-\rho_g)}$.
			If $\vx^*\in\cX$ satisfying 
$g_+(\vx^*)\leq B_g$	 is an $\epsilon$-stationary point of problem~\cref{eq:penaltymodel},  then $\vx^*$ satisfies $g_i(\vx^*)\leq 0$ for all $i\in[m]$, and is an $\epsilon$-KKT point of problem~\cref{eq:model}.
		\end{lemma}
		\begin{proof}
			Since $\vx^*$ is an $\epsilon$-stationary point of problem~\cref{eq:penaltymodel}, 
			there exist 
			$\vgamma\in\partial f(\vx^*)$, $\{\vzeta_{g_i}\in\partial [g_i(\vx^*)]_+\}_{i\in[m]}$, {and $\vu\in\mathcal{N}_{\mathcal{X}}(\vx^*)$} such that
			\begin{equation}
				\label{eq:KKTe}\left\|\textstyle \beta\sum_{i=1}^{m} \vzeta_{g_i}+\vu+ \vgamma\right\|\le \epsilon.
			\end{equation} 
			
			We claim 
			$g_i(\vx^*)\leq 0$ for all $i\in[m]$. 
			Otherwise, it holds  $0<g_+(\vx^*)\leq B_g$ and $[m] \cap \{i : g_i(\vx^*) \geq 0 \} \neq \emptyset$. 
			We then derive from \cref{eq:KKTe} and the triangle inequality that 
			\begin{align}
				\label{eq:vio2}
				\epsilon& \geq 
				{\beta}\left\|\textstyle \sum_{i=1}^{m} \vzeta_{g_i}+\frac{\vu}{\beta}\right\| -\|\vgamma\| \geq  {\beta}\left\|\textstyle \sum_{i=1}^{m} \vzeta_{g_i}+\frac{\vu}{\beta}\right\| -l_f \geq \beta \theta -l_f,
			\end{align}
			where 
			the second inequality holds by~\Cref{ass:problemsetup}(c), and the third one follows from~\Cref{lem:necessary} with $\cJ_1=[m]$. 
			This contradicts the condition $\beta > \frac{1}{\theta} (l_f +\epsilon)$. Hence 
			$g_i(\vx^*)\le 0$ for all $i\in[m]$.
 			
			Now we find a corresponding multiplier 
			as follows. For each $i\in [m]$, 
			if $g_i(\vx^*)=0$, we obtain  $\partial [g_i(\vx^*)]_+=  {\mathrm{co}} (\partial g_i(\vx^*)\cup \{\vzero \})$ by~\cite[Proposition 2.3.12]{clarke1990optimization} and the weak convexity of $[g_i]_+$. Thus, there exists $\lb_i\in[0,1]$ such that $  \vzeta_{g_i} \in  \lb_i \partial g_i(\vx^*)$. 
			In this case, we set $t_i =\lb_i$. 
			If $g_i(\vx^*)<0$, we have $\partial [g_i(\vx^*)]_+= \{\vzero\}$ so $  \vzeta_{g_i} =\vzero$, and 
			 we set $t_i =0 $. 
			For the constructed $\vt$, we have $t_i g_i(\vx^*) = 0, \forall\, i\in [m]$. In addition, $\beta\sum_{i=1}^{m} \vzeta_{g_i} \in \sum_{i=1}^{m}\beta t_i \partial g_i(\vx^*)$, and thus \cref{eq:KKTe} indicates $\dist\big(\vzero, \partial f(\vx^*) + \sum_{i=1}^{m}\beta t_i \partial g_i(\vx^*) + \cN_\cX(\vx^*)\big) \le \epsilon$. 
			Therefore $\vx^*$ is an $\epsilon$-KKT point of problem~\cref{eq:model} with the multiplier $\beta \vt$. 
		\end{proof}

		\subsection{Exact penalization with a ``smoothed'' penalty approach}
		\label{sec:exact2}
		In this subsection, we establish the relationship between models~\cref{eq:model} and~\cref{eq:smoothed}.
		We {begin by showing that }
        the function~$F^{\nu}$ in \cref{eq:smoothed} is  weakly convex. 
		\begin{lemma}
			\label{lem:weaklyFnu}
			Suppose~\Cref{ass:problemsetup} holds. {For any} $\nu>0$, the function $F^{\nu}$ defined in~\cref{eq:smoothed} is $\rho$-weakly convex, where $\rho = \rho_f + \beta m\rho_g$. 
		\end{lemma}
		\begin{proof}
			To obtain the desired result, we only need to prove $ h^{\nu} (\vg(\vx))$ is $m\rho_g$-weakly convex. Notice that 
			\begin{align*}
				h^{\nu} (\vg(\vx)) + \frac{m\rho_g}{2}\|\vx\|^2= \sum_{i=1}^{m}\max _{0\leq y\leq 1} \left[y \cdot g_i(\vx) -\frac{\nu}{2} y^2 +\frac{\rho_g}{2}\|\vx\|^2\right].
			\end{align*}
			Since $g_i$ is $\rho_g$-weakly convex for each $i\in[m]$, $y \cdot g_i(\vx) -\frac{\nu}{2} y^2 +\frac{\rho_g}{2}\|\vx\|^2$ is convex with respect to $\vx$ for all $0\leq y\leq 1$. Thus  {by Danskin's theorem~\cite{bertsekas1997nonlinear},}
			$h^{\nu} (\vg(\vx)) + \frac{m\rho_g}{2}\|\vx\|^2$ is convex.  
		\end{proof}
		
		Next,  we introduce some notations and key properties of the Moreau envelope.
		Denote $F_{\frac{1}{2\rho}}$ and~$F^{\nu}_{\frac{1}{2\rho}}$ as the Moreau envelope of $F $ and $F^{\nu} $ with parameter $\frac{1}{2\rho}$ {on $\mathcal{X}$,} respectively, where $\rho = \rho_f + \beta m\rho_g$.  We let
		\begin{equation}\label{eq:def-x-hat-+}
			\begin{aligned}
				&\overline\vx^*=  \argmin_{\vx\in\mathcal{X}}  F(\vx) + \rho \|\vx -\vx^*\|^2,\quad 
				 \widehat{\vx}^* =\argmin_{\vx\in\mathcal{X}}   F^{\nu}(\vx) + \rho \|\vx -\vx^*\|^2.
			\end{aligned}
		\end{equation}	
	 Then we have 
		\begin{align}
			\label{eq:gradientmore}
			&\nabla F_{\frac{1}{2\rho }}(\vx^*) = 2\rho(\vx^*-\overline\vx^*), \ \nabla F^{\nu}_{\frac{1}{2\rho }}(\vx^*) = 2\rho(\vx^*- \widehat{\vx}^*), \text{ and}\\
			&\vzero \in \partial f\left(\widehat\vx^*\right) +\beta \sum_{i=1}^{m} \mathrm{Proj}_{[0,1]}\left(\frac{g_i\left(\widehat\vx^*\right)}{\nu}\right) \partial g_i\left(\widehat\vx^*\right) +2\rho \left(\widehat\vx^*-\vx^*\right) + \mathcal{N}_{\mathcal{X}}(\widehat\vx^*).
			\label{eq:kktwidex2}
		\end{align}
		In addition, by  
		\cite[Eqn. (4.3)]{drusvyatskiy2019efficiency}, it holds that  
		\begin{equation}
			\label{eq:importantine}
			F(\overline{\vx}^*)\leq F(\vx^*),
			\text{ and }
			\dist(\vzero , \partial F(\overline{\vx}^*)+\mathcal{N}_{\mathcal{X}}(\overline\vx^*))
			\leq \big\|\nabla F_{\frac{1}{2\rho }}(\vx^*)\big\| .
		\end{equation}
		
		We now present our main results  
		to demonstrate the effectiveness of solving a smoothed exact penalty problem for both deterministic and stochastic settings.
		\begin{theorem}
			\label{thm:main}
			Suppose that~\Cref{ass:problemsetup} holds and \cref{def:2} holds with \mbox{$B_g>m \rho_g D^2$}. Given $\epsilon>0$, let $\theta = \sqrt{2B (\overline{\rho}-\rho_g)}$, $\rho = \rho_f + \beta m\rho_g$,
			\begin{equation}
				\label{eq:epsilon}			
				\begin{aligned}
					&\beta >  \max\left\{\frac{1}{\theta}(l_f + \epsilon ), \frac{2(B_f +\rho_f D^2)}{B_g - m\rho_g  D^2}\right\}, 0 < \overline{\epsilon} \leq \epsilon \min\left\{ \frac{\beta \theta - l_f}{2 B_g}, \frac{\beta \theta -l_f  }{ m\beta l_g +\beta \theta -l_f}, 2 \rho \right\},\\ 
					&
					\text{ and }  0 < \nu \leq \min\left\{B_g, \frac{\epsilon}{\beta}, \frac{B_g-m\rho_g  D^2}{2m}, \frac{\epsilon}{2}\right\}.
				\end{aligned}
			\end{equation}
			If $\vx^*$ is a near-$\overline\epsilon$ stationary
			point of problem~\cref{eq:smoothed} with $t=\frac{1}{2\rho}=\frac{1}{2(\rho_f + \beta m \rho_g)}$ 
			deterministically or in expectation, then $\|\vx^*-\widehat{\vx}^*\| \le \epsilon$ holds deterministically or in expectation, where $\widehat{\vx}^*$ is defined in~\cref{eq:def-x-hat-+}. In addition,   $\widehat{\vx}^*$   is an $\epsilon$-KKT point of problem~\cref{eq:model} deterministically or in expectation.   
		\end{theorem}	
		
		\begin{proof}
			We only need to   prove the theorem in the expectation case, {when $\vx^*$ is a random vector.} 
			For the deterministic case, one can simply drop the expectation notation $\EE$.
			
			First, 
            from \Cref{def:nearkkt} and \cref{eq:gradientmore}, it follows $
			\mathbb{E}\left[\|2 \rho (\widehat{\vx}^* - \vx^*)\|^2\right] \leq~\overline{\epsilon}^2
			$ by the near-$\overline\epsilon$ stationarity of $\vx^*$ {in expectation}. 
            {Because $\overline\epsilon \le 2\rho \epsilon$, it follows that}
			\begin{equation}\label{eq:eps-dist-xbar-x}
				\mathbb{E}\left[\|2 \rho (\widehat{\vx}^* - \vx^*)\|\right] \leq \overline{\epsilon} \text{ and  } \mathbb{E}\left[\|\widehat{\vx}^* - \vx^*\|\right] \leq \epsilon. 
			\end{equation}	
			
			Second, we define two index sets: $\mathcal{J}_2=\{i: g_i(\widehat\vx^*)/\nu\leq 1\}$ and $\mathcal{J}_1=[m]\setminus \mathcal{J}_2$. Construct the vector $\boldsymbol{\lambda}_{\widehat\vx^*}$ by setting $(\boldsymbol{\lambda}_{\widehat\vx^*})_i = \beta \mathrm{Proj}_{[0,1]}(g_i(\widehat\vx^*)/\nu) \ge 0$ for $i\in \mathcal{J}_2$ and $(\boldsymbol{\lambda}_{\widehat\vx^*})_i=0$ for $i\in \mathcal{J}_1$. We next bound $|(\boldsymbol{\lambda}_{\widehat\vx^*})_i g_i(\widehat\vx^*)|$ for each $i$. 
			For $i\in\mathcal{J}_1$, we directly have $(\boldsymbol{\lambda}_{\widehat\vx^*})_i g_i(\widehat\vx^*)=0$. For $i\in\mathcal{J}_2$, if $g_i(\widehat\vx^*)\leq 0$, then $(\boldsymbol{\lambda}_{\widehat\vx^*})_i=0$ and $(\boldsymbol{\lambda}_{\widehat\vx^*})_i g_i(\widehat\vx^*)=0$; if $g_i(\widehat\vx^*)>0$, we have from $g_i(\widehat\vx^*)/\nu\leq 1$ that $g_i(\widehat\vx^*) \leq \nu \leq \epsilon/2$, and thus 
			$|(\boldsymbol{\lambda}_{\widehat\vx^*})_i g_i(\widehat\vx^*)|=(\boldsymbol{\lambda}_{\widehat\vx^*})_i g_i(\widehat\vx^*)\leq \beta g_i(\widehat\vx^*)\leq \beta \nu\leq \epsilon.$ Hence, 
			\begin{equation}\label{eq:cs-eps}
			|(\boldsymbol{\lambda}_{\widehat\vx^*})_i g_i(\widehat\vx^*)|\leq \epsilon, \forall\, i\in[m].
			\end{equation}

			Third, we show $g_+(\widehat{\vx}^*) \le B_g$. By \cref{eq:kktwidex2} and the definition of $\cJ_1$ and $\cJ_2$, we have
			\begin{align}
				\label{eq:kktvio22}
				\vzero \in \partial f\left(\widehat\vx^*\right) +\beta  \partial \sum_{i\in\mathcal{J}_1} [g_i\left(\widehat\vx^*\right)]_+ +   \sum_{i\in\mathcal{J}_2} (\boldsymbol{\lb}_{\widehat\vx^*})_i\partial g_i\left(\widehat\vx^*\right)  +\mathcal{N}_{\mathcal{X}}(\widehat\vx^*)+2\rho \left(\widehat\vx^*-\vx^*\right).
			\end{align}
			By \cref{ass:problemsetup}, there exists $\mathbf{x}_{\mathrm{feas}}\in\cX\cap\{\vx:\vg(\vx)\leq 0\}$.
			Hence, from~\cref{eq:def-x-hat-+},  it holds
			\begin{equation}
				\begin{aligned}
					& f\left(\widehat{\mathbf{x}}^*\right)+\beta h^{\nu}\left(\vg\left(\widehat{\mathbf{x}}^*\right)\right)+\rho\left\|\widehat{\mathbf{x}}^*-\mathbf{x}^*\right\|^2 \\ \leq& f\left(\mathbf{x}_{\mathrm{feas}}\right)+\beta h^{\nu} \left(\vg\left(\mathbf{x}_{\mathrm{feas}}\right)\right)+\rho\left\|\mathbf{x}_{\mathrm{feas}}-\mathbf{x}^*\right\|^2 
					  =f\left(\mathbf{x}_{\mathrm{feas}}\right)+\rho\left\|\mathbf{x}_{\mathrm{feas}}-\mathbf{x}^*\right\|^2 .
				\end{aligned}
			\end{equation}
			 Thus,  
			$
			f(\widehat{\vx}^*) + \beta h^{\nu}(\vg(\widehat{\vx}^*) )+ \rho \|\widehat{\vx}^* - \vx^*\|^2\le f(\vx_{\mathrm{feas}}) + \rho D^2$ by \Cref{ass:problemsetup}(a).
			In addition, we have  $\beta \geq \frac{B_f + \rho_f D^2}{B_g - m\nu -  m\rho_g  D^2}$ by  $\beta>\frac{2(B_f +\rho_f D^2)}{B_g - m\rho_g  D^2}$ and $\nu\leq \frac{B_g-m\rho_g  D^2}{2m}$.
			Now further by $m\rho_g  D^2 < B_g$, 
			  $\rho = \rho_f + \beta m\rho_g$, and \cref{eq:bd-f-diff},
			it follows that
			\begin{align}
			\label{eq:hnu}
	 	\textstyle	h^{\nu}(\vg(\widehat{\vx}^*)  )
			\leq 
			\frac{1}{\beta} \Bigl(f(\vx_{\mathrm{feas}}) - f(\widehat{\vx}^*) + \rho D^2\Bigr)\leq \frac{B_f + \rho D^2}{\beta} \leq B_g - m\nu.
			\end{align}
			This implies $  g_+(\widehat{\vx}^*) \le h^{\nu}(\vg(\widehat{\vx}^*)  ) + m\nu \le B_g$ by the definition of $h^{\nu}$.
			
			We now discuss two events:  $\cA_1$ as the event of $\mathcal{J}_1=\emptyset$ and $\cA_2$ as the event of $\mathcal{J}_1\neq\emptyset$. Let $\mathrm{Prob}_1$ and $\mathrm{Prob}_2$ denote their probabilities.  
			When $\cA_1$ happens, we have  {$\dist\left(\vzero, \partial f\left(\widehat\vx^*\right) +\sum_{i=1}^{m}(\boldsymbol{\lb}_{\widehat\vx^*})_i\partial g_i\left(\widehat\vx^*\right)+\mathcal{N}_{\mathcal{X}}(\widehat\vx^*)\right) \leq 2\rho \| \widehat\vx^*-\vx^*\|$} by~\cref{eq:kktvio22}. 
			Hence, 
			\begin{equation}
				\label{eq:event12}
				\begin{aligned}
					&\dist\left(\vzero, \partial f\left(\widehat\vx^*\right) +\sum_{i=1}^{m}(\boldsymbol{\lb}_{\widehat\vx^*})_i\partial g_i\left(\widehat\vx^*\right)+\mathcal{N}_{\mathcal{X}}(\widehat\vx^*)\right) \leq  2\rho \| \widehat\vx^*-\vx^*\|, \\ &g_i(\widehat\vx^*) 
					\leq \frac{\epsilon}{2},\text{ and }|(\boldsymbol{\lb}_{\widehat\vx^*})_i g_i\left(\widehat\vx^*\right)| \leq \epsilon, \forall\,i\in[m].
				\end{aligned}
			\end{equation}			
			When $\cA_2$ happens, we have $\mathcal{J}_1 \cap \{i : g_i(\widehat\vx^*) \geq 0 \} \neq \emptyset$,  and $(\lambda_{\widehat\vx^*})_i g_i(\widehat\vx^*) \geq 0$ for all $i \in \mathcal{J}_2$.
			Using \cref{lem:necessary} and \Cref{ass:problemsetup}(c), we obtain that   
			\begin{equation}
				\label{eq:2}
				\begin{aligned} 
					&	2\rho\| \widehat\vx^* - \vx^*\|  \\ \stackrel{\cref{eq:kktvio22}}{\geq} & \dist\left(\vzero,  \partial f\left(\widehat\vx^*\right) +\beta  \partial \sum_{i\in\mathcal{J}_1} [g_i\left(\widehat\vx^*\right)]_+ +   \sum_{i\in\mathcal{J}_2} (\boldsymbol{\lb}_{\widehat\vx^*})_i\partial g_i\left(\widehat\vx^*\right)  +\mathcal{N}_{\mathcal{X}}(\widehat\vx^*)\right) 
					{\geq} \beta\theta - l_f.  
				\end{aligned}
			\end{equation} 
			
			{Next, we combine events $\mathcal{A}_1$ and $\mathcal{A}_2$. From}~\cref{eq:2} and by the definition of $\mathrm{Prob}_1$ and $\mathrm{Prob}_2$, we have 
			\begin{equation*} 
				0 \cdot \mathrm{Prob}_1 + (\beta \theta - l_f) \mathrm{Prob}_2 \leq \mathbb{E}\left[\|2\rho (\widehat{\vx}^* - \vx^*)\|\right] \overset{\cref{eq:eps-dist-xbar-x}}\leq \overline{\epsilon}.
			\end{equation*}
			Thus
			$
			\mathrm{Prob}_2 \leq \frac{\overline{\epsilon}}{\beta\theta - l_f}.
			$
			Splitting the integration for full expectation to $\cA_1$ and $\cA_2$ gives 
			\begin{equation}
				\begin{aligned}
					\label{eq:kkttotal2}
					&\mathbb{E}\left[\dist\left(\vzero, \partial f\left(\widehat\vx^*\right) +\sum_{i=1}^{m}(\boldsymbol{\lb}_{\widehat\vx^*})_i\partial g_i\left(\widehat\vx^*\right)+\mathcal{N}_{\mathcal{X}}(\widehat\vx^*)\right)\right] 
					\\&\stackrel{\cref{eq:event12},~\cref{eq:2}}{\leq} \mathbb{E}\left[\|2\rho (\widehat{\vx}^* - \vx^*)\|\right] + \mathrm{Prob}_2 m\beta l_g  \leq \overline{\epsilon} + \frac{\overline{\epsilon}m\beta l_g }{\beta \theta - l_f} \leq \epsilon,
				\end{aligned}
			\end{equation}
			where the last inequality uses $\overline{\epsilon}\leq \frac{\beta \theta -l_f  }{ m\beta l_g +\beta \theta -l_f}\epsilon$.
			{In addition}, 
			we have 
			$$
			\mathbb{E}[g_i(\widehat\vx^*)] {\leq} \nu \mathrm{Prob}_1 + B_g \mathrm{Prob}_2 \leq \nu + \frac{B_g \overline{\epsilon}}{\beta\theta - l_f} \leq \epsilon,\forall\, i\in [m] ,$$ 
			where we have used 
			$\nu \leq \frac{\epsilon}{2}$ and $\overline\epsilon \leq \frac{\beta \theta - l_f}{2 B_g}$.  
			Therefore, by recalling \cref{eq:cs-eps}, we have that $\widehat\vx^*$ is an $\epsilon$-KKT point of problem~\cref{eq:model} in expectation and  
			 complete the proof.	
		\end{proof}
		
		 {\begin{remark} 
		 		\Cref{thm:main} relies on several key assumptions, which we clarify below to highlight their implications and scope:
		 		\begin{itemize}
		 			\item[$\mathrm{(i)}$]  The 
					assumed near-$\bar\epsilon$ stationary point $\vx^*$ of the penalized problem~\cref{eq:smoothed} can be found by   
					\verb|3S-Econ|   introduced in the next section. 
		 			
		 			\item[$\mathrm{(ii)}$]  
					We now provide further intuition regarding the satisfaction of the required Slater-type CQ  with $B_g>m\rho_gD^2$. This condition naturally holds in several   scenarios:
					\begin{enumerate} 
						\item  \textit{Convex constraints:} If the constraints $\{g_i\}_{i\in[m]}$ are convex on $\cX$, then $\rho_g=0$ and the requirement simplifies to $B_g>0$, and the Slater-type CQ coincides   with the   Slater's condition (see \cref{prop:regurela}(i)), which can  be easily verified. An example is Problem~1 in \cref{eq:data}.
						
						\item  \textit{Deterministic constraints:} If the constraints are deterministic, such as in finite-sum problems where $g$ can be exactly evaluated, the   condition $B_g > m\rho_g D^2$ is unnecessary and can be omitted; see \cref{rem:deter}. In this case, only the Slater-type CQ itself is required.
						
						\item  \textit{General stochastic weakly-convex constraints:} For general stochastic constraints that are weakly convex, a practical sufficient condition is the strong feasibility condition in~\eqref{eq:str-feas}. By \cref{prop:regurela}(ii), strong feasibility implies the Slater-type CQ for \emph{any} positive value $B_g>0$. Thus, one can explicitly choose $B_g$ sufficiently large to satisfy $B_g > m\rho_g D^2$. Furthermore, in \cref{appen:2}, we provide easily verifiable sufficient conditions for Problems~2--3 in \cref{eq:data} that ensure the strong feasibility condition~\eqref{eq:str-feas}, and consequently also guarantee the \mbox{Slater-type CQ} with the condition $B_g > m\rho_g D^2$.
					\end{enumerate}

		 			\item[$\mathrm{(iii)}$]  The 
					values of parameters in~\cref{eq:epsilon} 
					depend on several constants, which may be unknown in general. 
					Nevertheless, we show in \cref{eq:data} that these constants 
					can be computed for certain applications, mitigating uncertainty and facilitating precise complexity assessments.
		 		\end{itemize}
			\end{remark}
			}

\section{A  SPIDER-type stochastic subgradient algorithm}
\label{sec:algorithm}
In this section, we present our algorithm, 
termed \verb|3S-Econ|, 
for solving problem~\cref{eq:model}. It is developed based on a subgradient method for problem~\cref{eq:smoothed} by using the SPIDER technique to estimate constraint function values. The algorithm framework is outlined in~\Cref{sec:af}. In~\Cref{sec:complexity}, we prove that  in expectation, \verb|3S-Econ| can find a near-$\epsilon$ stationary point of problem~\cref{eq:smoothed}. Building on the results presented in the previous section, 
we further establish in~\Cref{sec:all} that an \((\epsilon, \epsilon)\)-KKT point  in expectation of problem~\cref{eq:model} can be obtained.

\subsection{Algorithm framework}\label{sec:af}
We employ the SPIDER technique~\cite{fang2018spider}  to obtain an estimator 
$\vu^{(k)}\in\RR^m$ of $\vg$ at the $k$-th iteration. Specifically, every~$q$ iterations, we draw a big batch of $S_1$ samples from $\mathbb{P}_{\vg}$, 
while for other iterations, a smaller batch of~$S_2$ samples {is} 
drawn. The estimator is set as
\begin{equation}
	\label{eq:updateu}
	{\vu^{(k)} = \left\{
		\begin{array}{lll}
			\vg(\vx^{(k)}, \mathcal{B}_k), & \text{ with }|\mathcal{B}_k| = S_1, & \text{ if }\mathrm{mod}(k, q) = 0 ,\\[1mm]
			\vu^{(k-1)} + \vg(\vx^{(k)}, \mathcal{B}_k) -\vg(\vx^{(k-1)}, \mathcal{B}_{k}), & \text{ with }|\mathcal{B}_k| = S_2, & \text{ otherwise},
		\end{array}
		\right.
	}
\end{equation}
where $\vu^{(k)}=[u_1^{(k)},\ldots,u_m^{(k)}]\zz$, $\mathcal{B}_k$ contains identically independent 
samples from~$\mathbb{P}_{\vg}$, and $g_i(\mathbf{x}, \mathcal{B})$ is defined in~\cref{eq:g-approx}. With $\vu^{(k)}$, we then take a sample subgradient of $f$ and~$\{g_i\}_{i\in[m]}$ at $\vx^{(k)}$ and perform the update in~\cref{eq:newupdate}.  The pseudocode is given in~\Cref{alg:dad2-2}.
{ As detailed in \cref{thm:comple}, for a given error tolerance $\bar\epsilon>0$, we set the penalty parameter $\beta$ to a sufficiently large predefined constant, the smoothing parameter $\nu = O(\bar\epsilon)$, the step size $\alpha_k = O(\bar\epsilon^{2})$,  and $q = S_2  = \lceil \sqrt{S_1}\rceil$, with $S_1 = O(\bar\epsilon^{-4})$ in the generic stochastic case or $S_1 = N$ for the finite-sum case. 
}

\begin{algorithm}[htbp]
	\caption{ {3S-Econ}: A \texttt{S}PIDER-type \texttt{S}tochastic \texttt{S}ubgradient algorithm for solving \texttt{E}xpectation \texttt{con}strained problem \eqref{eq:model}}\label{alg:dad2-2}
	\begin{algorithmic}[1]
		\STATE{\textbf{Input:} 
		$\vx^{(0)}\in\mathcal{X}$, $\beta>0$, $\nu>0$, and positive intergers $q$, $S_1, S_2$, and $K$.}
		
		\FOR{$k=0,1, \cdots, Kq-1$} 
			\STATE Sample $\cB_k$ and set $\vu^{(k)}$ by~\cref{eq:updateu}. 
			\STATE{Choose $\alpha_k>0$, generate a (stochastic) subgradient $\boldsymbol{\zeta}_f^{(k)}$ of $f$ and (stochastic) subgradient $\boldsymbol{\zeta}_{g_i}^{(k)}$ of $g_i$ for each $i\in[m]$ that are independent of $\cB_k$.}
			\STATE{Update $\vx$ by~\cref{eq:newupdate}.} 
			\ENDFOR
		\end{algorithmic}
	\end{algorithm}
	For the convenience of analysis, we {denote}
	\begin{equation}
		\begin{aligned}
			&\boldsymbol{\zeta}_{F^{\nu}}^{(k)} = \boldsymbol{\zeta}_{f}^{(k)}  + \beta \sum_{i=1}^{m} \mathrm{Proj}_{[0,1]}\left(\frac{g_i(\vx^{(k)})}{\nu}\right)\boldsymbol{\zeta}_{g_i}^{(k)} \text{, and }\\ &\vw^{(k)} = \sum_{i=1}^{m}\left( \mathrm{Proj}_{[0,1]}\left(\frac{u_i^{(k)}}{\nu}\right) - \mathrm{Proj}_{[0,1]}\left(\frac{g_i(\vx^{(k)})}{\nu}\right)\right)\boldsymbol{\zeta}_{g_i}^{(k)}.
		\end{aligned}
		\label{eq:zetakwk}
	\end{equation}
	Then the update in 	\cref{eq:newupdate} can be {rewritten as}
	\begin{equation}\label{eq:updatex}
		 \vx^{(k+1)} = \mathrm{Proj}_{\mathcal{X}}\left(\vx^{(k)} -\alpha_k \left(\boldsymbol{\zeta}_{F^{\nu}}^{(k)} + \beta \vw^{(k)}\right)\right).
	\end{equation}
	Throughout this section, we make the following assumption, where for brevity of notation, we let $\mathbb{E}_k[\cdot]:=\mathbb{E}\left[\cdot \mid \vx^{(0)}, \vx^{(1)}, \ldots, \vx^{(k)}\right]$, {representing the conditional expectation given} 
	$\{\vx^{(0)}, \vx^{(1)}, \ldots, \vx^{(k)}\}$ for any $k\ge0$. In addition, $\mathrm{Var}(\boldsymbol{\zeta}):=\EE[\|\boldsymbol{\zeta} - \EE[\boldsymbol{\zeta}]\|^2]$ for a random vector $\boldsymbol{\zeta}$.
	\begin{assumption}
		\label{ass:gradient}
		The following conditions hold. 
		\begin{itemize}
			\item[$(a)$] \textbf{Unbiased function {estimate} and Lipschitz continuity of \(\vg\):} For all \(k \geq 0\), the samples in \(\mathcal{B}_k\) are mutually independent. It holds  
			{ $
				\mathbb{E}_k[\vg(\vx^{(k)}, \xi_{\vg})] = \vg(\vx^{(k)})
				$
				and 
				$
				\mathbb{E}\left[\left\|\vg(\vx^{(k)}, \xi_{\vg}) - \vg(\vx^{(k)})\right\|^2\right] \leq\sigma^2
				$}
			for any \(\xi_{\vg} \in \mathcal{B}_k\) and some \(\sigma > 0\). Moreover, {for all $\vx,\vy\in \cX$, it holds}  
			$
			\mathbb{E}\left[\left\|\vg(\vx, \xi_{\vg}) - \vg(\vy, \xi_{\vg})\right\|^2\right] \leq L_g^2 \|\vx - \vy\|^2.
			$
			
			\item[$(b)$] \textbf{Unbiased Stochastic Subgradients:} 
			For each $k$,
			$
			\mathbb{E}_k[\boldsymbol{\zeta}_f^{(k)}] \in \partial f(\vx^{(k)})$ and $\mathbb{E}_k[\boldsymbol{\zeta}_{g_i}^{(k)}] \in \partial g_i(\vx^{(k)})$ for each $i\in[m]$. In addition, there are $\sigma_1\ge 0$ and $\sigma_2\ge0$ such that $\mathrm{Var}(\boldsymbol{\zeta}_f^{(k)}) \le \sigma_1^2$ and $\mathrm{Var}(\boldsymbol{\zeta}_{g_i}^{(k)}) \le \sigma_2^2$  for each $i\in[m]$.  
		\end{itemize}
	\end{assumption}
	
	By~\Cref{ass:gradient}(b),
	it follows that $\mathbb{E}[\boldsymbol{\zeta}_{F^{\nu}}^{(k)}]\in\partial F^{\nu}(\vx^{(k)})$.
	
	\subsection{Convergence analysis}\label{sec:complexity}
	In this subsection, we present the convergence analysis of~\Cref{alg:dad2-2}. 
	First, we derive a bound on the difference between $\vu^{(k)}$ and $\vg(\vx^{(k)})$.
	
	\begin{lemma}
		\label{lem:boundug}
		Under~\Cref{ass:problemsetup} and~\Cref{ass:gradient},  
		let $\{\vx^{(k)}, \vu^{(k)}\}_{k=0}^{Kq-1}$ 
		be generated by~\Cref{alg:dad2-2}. Define $r_g^{(k)}:=\mathbb{E}\left[\|\vu^{(k)} -\vg(\vx^{(k)})\|^2\right]$. {Then, for all $ k\ge0$,} it  holds $r_g^{(k)} \leq  \frac{\sigma^2}{S_1}+ \sum_{i=1}^{q} \frac{\alpha_{\lfloor k/q \rfloor q+i-1}^2  L_F^2 L_g^2}{S_2}$,  where $L_{F}^2 =2(\sigma_1^2 + l_f^2 + \beta^2m^2 \sigma_2^2 + \beta^2m^2 l_g^2)$.
	\end{lemma}
	
	\begin{proof}
		{Case (i):} When $\mathrm{mod}(k,q)=0$,  
		{~\cref{eq:updateu} and~\Cref{ass:gradient}(a) give}
		\begin{equation}
			\label{eq:rqnq}
			r_g^{(k)}=\mathbb{E}\left[\left\|\vg(\vx^{(k)},\mathcal{B}_{k}) -\vg(\vx^{(k)})\right\|^2\right] \leq \frac{\sigma^2}{S_1}.
		\end{equation}
		{Case (ii):} When $\mathrm{mod}(k,q)>0$, 
		it follows that 
		\begin{align}
			\notag
			&\mathbb{E}\left[\left\|\vu^{(k)} -\vg(\vx^{(k)})\right\|^2\right]  \stackrel{\cref{eq:updateu}}{=} \mathbb{E}\left[\left\|\vu^{(k-1)} + \vg(\vx^{(k)},\mathcal{B}_k) -\vg(\vx^{(k-1)}, \mathcal{B}_{k}) -\vg(\vx^{(k)})\right\|^2\right] \\ 
			\notag	& = r_g^{(k-1)}+\mathbb{E}\left[\left\|\vg(\vx^{(k-1)})+ \vg(\vx^{(k)}, \mathcal{B}_k) -\vg(\vx^{(k-1)}, \mathcal{B}_{k}) -\vg(\vx^{(k)})\right\|^2\right]\\
			\notag	& = r_g^{(k-1)}+\frac{1}{|\mathcal{B}_k|^2}\sum_{\xi_{\vg}\in \mathcal{B}_k}\mathbb{E}\left[\left\|\vg(\vx^{(k-1)})+ \vg(\vx^{(k)}, \xi_{\vg}) -\vg(\vx^{(k-1)}, \xi_{\vg}) -\vg(\vx^{(k)})\right\|^2\right]\\
			\notag	 & \le r_g^{(k-1)}+\frac{1}{|\mathcal{B}_k|^2}\sum_{\xi_{\vg}\in \mathcal{B}_k}\mathbb{E}\left[\left\|\vg(\vx^{(k)}, \xi_{\vg}) -\vg(\vx^{(k-1)}, \xi_{\vg}) \right\|^2\right]\\
			& \leq r_g^{(k-1)}+ \frac{L_g^2}{S_2}\mathbb{E}\left[\| \vx^{(k)}-\vx^{(k-1)} \|^2\right] 
			\leq r_g^{(k-1)}+\frac{\alpha_{k-1}^2 L_F^2 L_g^2}{S_2},
			\label{eq:keyrk} 
		\end{align}
		where the second equality holds {due to} the unbiasedness of $\vg(\vx, \xi_{\vg})$, the third equality follows from the mutual independence of samples in $\mathcal{B}_k$, 
		the second inequality follows by~\Cref{ass:gradient}(a), and the last one uses  
			\begin{align}
				\notag
				&\mathbb{E}\left[\| \vx^{(k+1)}-\vx^{(k)} \|^2\right] 
				= \mathbb{E}\left[   
				\left\|  \mathrm{Proj}_{\mathcal{X}}\left(\vx^{(k)} -\alpha_k \left(\boldsymbol{\zeta}_{F^{\nu}}^{(k)} + \beta \vw^{(k)}\right)\right) - \mathrm{Proj}_{\mathcal{X}}\left( \vx^{(k)}\right)  \right\|^2
				\right] \\
				\notag&\leq \textstyle \alpha_k^2\EE\left[\left\| \left(\boldsymbol{\zeta}_f^{(k)} + \beta \sum_{i=1}^m\mathrm{Proj}_{[0,1]}\left(\frac{u_i^{(k)}}{\nu}\right) \boldsymbol{\zeta}_{g_i}^{(k)} \right)\right\|^2\right] \\\notag&\le  2 \alpha_k^2 \left(\EE\left[\|\boldsymbol{\zeta}_f^{(k)}\|^2 +  \beta^2 m \sum_{i=1}^m \left\| \boldsymbol{\zeta}_{g_i}^{(k)}\right\|^2\right]\right)
				\\\notag&\le   2 \alpha_k^2 \left(\mathrm{Var}(\boldsymbol{\zeta}_f^{(k)})+ \|\EE[\boldsymbol{\zeta}_f^{(k)}]\|^2 + \beta^2 m \sum_{i=1}^m\mathrm{Var}(\boldsymbol{\zeta}_{g_i}^{(k)}) + \beta^2 m\sum_{i=1}^m\|\EE[\boldsymbol{\zeta}_{g_i}^{(k)}]\|^2\right) \\&\le  2 \alpha_k^2(\sigma_1^2 + l_f^2 + \beta^2 m^2 \sigma_2^2 + \beta^2m^2 l_g^2)= \alpha_k^2 L_F^2.
				\label{eq:bd-x-diff-exp}
			\end{align} 
		Since $\mathrm{mod}(k,q)>0$, 
		there exist nonnegative integers $t$ and $s< q$ such that $k=tq+s$.
		Summing~\cref{eq:keyrk} over $k=tq+1, \ldots, tq+s-1 ,tq+s$, and {applying}~\cref{eq:rqnq}, we
		have  
		\begin{equation}\label{eq:bd-chain-r_g}
			r_g^{(k)} \leq r_g^{(tq)} + \sum_{i=1}^s \frac{\alpha_{tq+i-1}^2 L_F^2 L_g^2}{S_2} \leq \frac{\sigma^2}{S_1}+ \sum_{i=1}^{q} \frac{\alpha_{\lfloor k/q \rfloor q+i-1}^2  L_F^2 L_g^2}{S_2} .
		\end{equation}
		The proof is then completed by combining Case (i) and Case (ii). 
	\end{proof}
	
	The following corollary is obtained immediately from~\cref{eq:bd-chain-r_g} by noticing $r_g^{(tq)}=0$ if \(\mathbb{P}_{\vg}\) is a uniform distribution over $N$ data points and $\cB_{tq}$ contains all the data points. 
	
	\begin{corollary}
		\label{cor:boundug2}
		Under~\Cref{ass:problemsetup} and~\Cref{ass:gradient}, suppose that \(\mathbb{P}_{\vg}\) is a uniform distribution over a finite dataset~$\Xi$ of size $N$. Let $\cB_k=\Xi$ if $\mathrm{mod}(k, q)=0$.  Then $r_g^{(k)}=0$ if $\mathrm{mod}(k, q)=0$ and $r_g^{(k)} \leq   \sum_{i=1}^{q} \frac{\alpha_{\lfloor k/q \rfloor q+i-1}^2  L_F^2 L_g^2}{S_2} $ otherwise,  where $L_{F} $ is given in \cref{lem:boundug}.
	\end{corollary}
	
	Next, we bound the difference between the Moreau envelope at two consecutive iterates. {Our analysis follows 
		that of~\cite[Theorem 3.1]{davis2019stochastic}, with the key modification in how we handle and bound the term \(\mathbb{E}\left[\left\langle\widehat{\vx}^{(k)} - \vx^{(k)}, \beta \vw^{(k)}\right\rangle\right]\).}
	\begin{lemma}
		\label{lem:diffmore}
		Under~\Cref{ass:problemsetup} and~\Cref{ass:gradient}, {suppose    $\{\vx^{(k)}\}_{k=0}^{Kq-1}$ and $\{\vu^{(k)}\}_{k=0}^{Kq-1}$ are} generated by~\Cref{alg:dad2-2}. With $L_{F}$ given in \cref{lem:boundug},
		it holds   
		\begin{equation}
			\label{eq:diffmore}
			\begin{aligned}
				&\mathbb{E}\left[F^{\nu}_{\frac{1}{2\rho}}(\vx^{(k+1)})\right] \\\leq & \mathbb{E}\left[F^{\nu}_{\frac{1}{2\rho}}(\vx^{(k)})\right]-\frac{\alpha_k}{{4}} \mathbb{E}\left[\left\|\nabla F^{\nu}_{\frac{1}{2\rho}}(\vx^{(k)})\right\|^2\right]+\frac{ \alpha_k \beta^2 (l_g^2+\sigma_2^2)}{\nu^2} r_g^{(k)}+{\alpha_k^2 {\rho} L_F^2}{}.		
			\end{aligned}
		\end{equation}
	\end{lemma}
	
	\begin{proof}
		For any $k\ge0$, 
		let $\widehat{\vx}^{(k)}=\operatorname{Prox}_{\frac{F^{\nu}}{2\rho}}({\vx}^{(k)})$. 
		We 
		deduce that 
		\begin{align}
			& \notag
			\mathbb{E}_k\left[F^{\nu}_{\frac{1}{2\rho}}(\vx^{(k+1)})\right]  \leq \mathbb{E}_k\left[F^{\nu}(\widehat{\vx}^{(k)})+\rho\left\|\widehat{\vx}^{(k)}-\vx^{(k+1)}\right\|^2\right] \\
			\notag \overset{\cref{eq:updatex}}= &
			\mathbb{E}_k\left[F^{\nu}(\widehat{\vx}^{(k)})+\rho\left\|\mathrm{Proj}_{\mathcal{X}}(\widehat{\vx}^{(k)})-\mathrm{Proj}_{\mathcal{X}}\left(\vx^{(k)} -\alpha_k \left(\boldsymbol{\zeta}_{F^{\nu}}^{(k)} + \beta \vw^{(k)}\right)\right)\right\|^2\right]
			\\ \notag\leq & 
			\mathbb{E}_k\left[F^{\nu}(\widehat{\vx}^{(k)})+\rho\left\| \widehat{\vx}^{(k)}- \left(\vx^{(k)} -\alpha_k \left(\boldsymbol{\zeta}_{F^{\nu}}^{(k)} + \beta \vw^{(k)}\right)\right)\right\|^2\right]
			\\ \notag= & 
			F^{\nu}(\widehat{\vx}^{(k)})+\rho\left\|{\vx}^{(k)}-\widehat{\vx}^{(k)}\right\|^2+2\rho\alpha_k \mathbb{E}_k\left[\left\langle\widehat{\vx}^{(k)}-{\vx}^{(k)}, \boldsymbol{\zeta}_{F^{\nu}}^{(k)} + \beta \vw^{(k)}\right\rangle\right] \\ \notag&\quad\quad\quad\quad+\rho \EE_k \left[\left\|\alpha_k \left(\boldsymbol{\zeta}_{F^{\nu}}^{(k)} + \beta \vw^{(k)}\right)\right\|^2\right] \\
			\notag = & F^{\nu}_{\frac{1}{2\rho}}(\widehat{\vx}^{(k)})+2\rho \alpha_k \mathbb{E}_k\left[\left\langle\widehat{\vx}^{(k)}-{\vx}^{(k)}, \boldsymbol{\zeta}_{F^{\nu}}^{(k)} + \beta \vw^{(k)}\right\rangle\right]+\rho \EE_k \left[\left\|\alpha_k \left(\boldsymbol{\zeta}_{F^{\nu}}^{(k)} + \beta \vw^{(k)}\right)\right\|^2\right] \\
			\notag \leq & F^{\nu}_{\frac{1}{2\rho}}(\widehat{\vx}^{(k)})+2\rho \alpha_k \left( F^{\nu}(\widehat{\vx}^{(k)})-F^{\nu}({\vx}^{(k)})+\frac{\rho}{2}\left\|{\vx}^{(k)}-\widehat{\vx}^{(k)}\right\|^2\right)\\ &\quad\quad\quad\quad +2\rho \alpha_k\mathbb{E}_k\left[\left\langle\widehat{\vx}^{(k)}-{\vx}^{(k)},   \beta \vw^{(k)}\right\rangle\right]+\rho \EE_k \left[\left\|\alpha_k \left(\boldsymbol{\zeta}_{F^{\nu}}^{(k)} + \beta \vw^{(k)}\right)\right\|^2\right],
			\label{eq:diffmoreauf}
		\end{align}
		where the first inequality follows directly from the definition of $F^{\nu}_{\frac{1}{2\rho}}$, 
		and the last one uses $\mathbb{E}_k[\boldsymbol{\zeta}_{F^{\nu}}^{(k)}]\in\partial F^{\nu}(\vx^{(k)})$  by~\Cref{ass:gradient}(b) and the $\rho$-weak convexity of $F^{\nu}$ by~\cref{lem:weaklyFnu}.
		
		{Since  
			the function $\vx \mapsto F^{\nu}(\vx)+\rho\left\|\vx-{\vx}^{(k)}\right\|^2$ is $\rho$-strongly convex, we have}
		\begin{equation}			
			\label{eq:difff}
			\begin{aligned}
				&F^{\nu}({\vx}^{(k)})-F^{\nu}(\widehat{\vx}^{(k)})-\frac{\rho}{2}\left\|{\vx}^{(k)}-\widehat{\vx}^{(k)}\right\|^2 \\= & \left(F^{\nu}({\vx}^{(k)})+\rho\|{\vx}^{(k)}-{\vx}^{(k)}\|^2\right)-\left(F^{\nu}(\widehat{\vx}^{(k)})+\rho\|{\vx}^{(k)}-\widehat{\vx}^{(k)}\|^2\right) 
				+\frac{\rho}{2}\left\|{\vx}^{(k)}-\widehat{\vx}^{(k)}\right\|^2 \\
				\geq & \rho\left\|{\vx}^{(k)}-\widehat{\vx}^{(k)}\right\|^2 \stackrel{\cref{eq:gradientmore}}{=}\frac{1}{4\rho}\left\|\nabla F^{\nu}_{\frac{1}{2\rho}}({\vx}^{(k)})\right\|^2.
			\end{aligned}
		\end{equation} 
		Substituting~\cref{eq:difff} into~\cref{eq:diffmoreauf}, and using the law of expectation and~\cref{eq:bd-x-diff-exp}, we {obtain} 
		\begin{equation}
			\begin{aligned}
				\mathbb{E}\left[F^{\nu}_{\frac{1}{2\rho}}(\vx^{(k+1)})\right] &\leq   \mathbb{E} \left[F^{\nu}_{\frac{1}{2\rho}}({\vx}^{(k)})\right]-\mathbb{E}
				\left[\frac{\alpha_k}{2}\left\|\nabla F^{\nu}_{\frac{1}{2\rho}}({\vx}^{(k)})\right\|^2\right] \\ &\quad+\mathbb{E} \left[2\rho \alpha_k\left\langle\widehat{\vx}^{(k)}-{\vx}^{(k)},   \beta \vw^{(k)}\right\rangle\right]+{\alpha_k^2 \rho L_F^2}{} .
			\end{aligned}
			\label{eq:diffmoreau2}
		\end{equation}
		For the third term in the right hand side (RHS) of~\cref{eq:diffmoreau2}, {it holds that}
		\begin{align*}
			&   2\rho \alpha_k \mathbb{E} \left[\left\langle\widehat{\vx}^{(k)}-{\vx}^{(k)},   \beta \vw^{(k)}\right\rangle \right] \leq \mathbb{E} \left[{\alpha_k \rho^2} \|\widehat{\vx}^{(k)}-{\vx}^{(k)}\|^2 \right]+ \mathbb{E} \left[{\alpha_k \beta^2 }{}\|\vw^{(k)}\|^2 \right] \\
			= &   \frac{\alpha_k }{4} \mathbb{E	} \left[\left\|\nabla F^{\nu}_{\frac{1}{2\rho}}({\vx}^{(k)})\right\|^2 \right]  \\ &+ {\alpha_k \beta^2 }{}\mathbb{E} \left[\left\|\sum_{i=1}^m\left( \mathrm{Proj}_{[0,1]}\left(\frac{u_i^{(k)}}{\nu}\right) -\mathrm{Proj}_{[0,1]}\left(\frac{g_i(\vx^{(k)})}{\nu}\right)\right)\boldsymbol{\zeta}_{g_i}^{(k)}\right\|^2 \right]\\
			\leq &   \frac{\alpha_k }{4} \mathbb{E} \left[\left\|\nabla F^{\nu}_{\frac{1}{2\rho}}({\vx}^{(k)})\right\|^2\right]  \\ &+   \alpha_k \beta^2 m (l_g^2+\sigma_2^2) \mathbb{E} \left[\sum_{i=1}^m\left|\mathrm{Proj}_{[0,1]}\left(\frac{u_i^{(k)}}{\nu}\right) -\mathrm{Proj}_{[0,1]}\left(\frac{g_i(\vx^{(k)})}{\nu}\right)\right|^2\right] \\
			\leq &   \frac{\alpha_k }{4}\mathbb{E} \left[ \left\|\nabla F^{\nu}_{\frac{1}{2\rho}}({\vx}^{(k)})\right\|^2\right] +   \alpha_k \beta^2m (l_g^2+\sigma_2^2) \mathbb{E} \left[\sum_{i=1}^m\left| \frac{u_i^{(k)}}{\nu} -\frac{g_i(\vx^{(k)})}{\nu}\right|^2\right] \\ =& \frac{\alpha_k }{4} \mathbb{E} \left[\left\|\nabla F^{\nu}_{\frac{1}{2\rho}}({\vx}^{(k)})\right\|^2 \right]+ \frac{ \alpha_k \beta^2 m (l_g^2+\sigma_2^2)}{\nu^2} r_g^{(k)},
		\end{align*}
		where {the first equality follows from~\cref{eq:gradientmore} and~\cref{eq:zetakwk},} 
		and the second inequality comes from~\Cref{ass:problemsetup}(c) and~\Cref{ass:gradient}(b). 
		Substituting the above inequality {back} into~\cref{eq:diffmoreau2} yields~\cref{eq:diffmore}. 
	\end{proof}
	
	We now give the convergence result in the following theorem.
	\begin{theorem}
		\label{thm:comple}
		Under~\Cref{ass:problemsetup} and~\Cref{ass:gradient}, 
		given $\overline{\epsilon}>0$, {suppose that} $\{\vx^{(k)}\}_{k=0}^{Kq-1}$ and $\{\vu^{(k)}\}_{k=0}^{Kq-1}$ are generated by~\Cref{alg:dad2-2} with
		\begin{align}
			\label{eq:q}
			&S_1=\lceil 16\overline{\epsilon}^{-2} \nu^{-2} m (l_g^2+\sigma_2^2) \beta^2 \sigma^2\rceil,\, q=S_2 =\lceil \sqrt{S_1} \rceil, 
			\,  \\
			&\alpha_k=\alpha = \min \Big\{  \frac{\overline{\epsilon}^2}{16\rho L_F^2},   \frac{\overline{\epsilon} \nu}{4 \beta \sqrt{m(l_g^2+\sigma_2^2)} L_gL_F}  \Big\},\text{ and } K=\lceil 16\overline{\epsilon}^{-2}q^{-1}\alpha^{-1} \Delta  \rceil.
			\label{eq:K}
		\end{align} 
		Here, 
		$\Delta:= F^{\nu}_{\frac{1}{2\rho}}(\vx^{(0)}) -\min_{ \vx }F^{\nu}_{\frac{1}{2\rho}}(\vx)$, and  $L_{F}$ is given in \cref{lem:boundug}.
		{It holds}  
		\begin{equation}\label{eq:bd-sum-grad}   
			\frac{1}{Kq}\sum_{k=0}^{Kq-1} \mathbb{E}\left[\left\|\nabla F^{\nu}_{\frac{1}{2\rho}}(\vx^{(k)})\right\|^2\right] \leq \overline{\epsilon}^2.
		\end{equation} 
	\end{theorem}

	\begin{proof}
		Since $\alpha_k=\alpha$,~\cref{lem:boundug} implies
		$$	\sum_{k=0}^{Kq-1}\frac{ \alpha_k \beta^2 (l_g^2+\sigma_2^2)}{\nu^2} r_g^{(k)} \leq \frac{ Kq \alpha \beta^2 m(l_g^2+\sigma_2^2)}{\nu^2} \left(\frac{\alpha^2 q L_F^2 L_g^2}{S_2}  + \frac{\sigma^2}{S_1}\right).
		$$
		Summing~\cref{eq:diffmore} from $k=0$ to $Kq-1$ and {applying} the above bound, we derive  
		\begin{align}
			\mathbb{E}\left[F^{\nu}_{\frac{1}{2\rho}}(\vx^{(Kq)})\right] \leq &\, \mathbb{E}\left[F^{\nu}_{\frac{1}{2\rho}}(\vx^{(0)})\right]-\frac{\alpha}{{4}} \sum_{k=0}^{Kq-1} \mathbb{E}\left[\left\|\nabla F^{\nu}_{\frac{1}{2\rho}}(\vx^{(k)})\right\|^2\right]\notag\\
			& + \frac{ Kq \alpha \beta^2m (l_g^2+\sigma_2^2)}{\nu^2} \left(\frac{\alpha^2 q L_F^2 L_g^2}{S_2}  + \frac{\sigma^2}{S_1}\right)+ Kq{\alpha^2 {\rho} L_F^2}{}.
			\label{eq:diffmoresum}
		\end{align}
		Rearranging the terms and multiplying both sides of \cref{eq:diffmoresum} by $\frac{4}{Kq \alpha}$, we obtain
		\begin{align}
			\notag
			\frac{1}{Kq}\sum_{k=0}^{Kq-1} \mathbb{E}\left[\left\|\nabla F^{\nu}_{\frac{1}{2\rho}}(\vx^{(k)})\right\|^2\right] \leq & \frac{4}{Kq \alpha} \left(\mathbb{E}\left[F^{\nu}_{\frac{1}{2\rho}}(\vx^{(0)})\right] -\mathbb{E}\left[F^{\nu}_{\frac{1}{2\rho}}(\vx^{(Kq)})\right]  \right) \\ &\quad\quad\quad + \frac{4  \beta^2m (l_g^2+\sigma_2^2)}{\nu^2} \left(\frac{\alpha^2 q L_F^2 L_g^2}{S_2}  + \frac{\sigma^2}{S_1}\right)+ 4{\alpha {\rho} L_F^2}{} \notag \\
			\le & \frac{4 \Delta}{Kq \alpha} + \frac{4  \beta^2 (l_g^2+\sigma_2^2)}{\nu^2} \left(\frac{\alpha^2 q L_F^2 L_g^2}{S_2}  + \frac{\sigma^2}{S_1}\right)+ 4{\alpha {\rho} L_F^2}{}.
			\label{eq:kktvio}
		\end{align}
		
		Now we bound the RHS of~\cref{eq:kktvio}. From $\alpha \leq \frac{\overline{\epsilon}^2}{16\rho L_F^2}$, it follows that $4{\alpha {\rho} L_F^2}\leq \frac{\overline{\epsilon}^2}{4}$. The choice $S_1=\lceil 16 \overline{\epsilon}^{-2}\beta^2m \nu^{-2} (l_g^2+\sigma_2^2) \sigma^2\rceil$ yields $\frac{4  \beta^2m (l_g^2+\sigma_2^2)}{\nu^2}  \frac{\sigma^2}{S_1}\leq \frac{\overline{\epsilon}^2}{4}$. The conditions $q= S_2
		$ 
		and $\alpha\leq \frac{\overline{\epsilon} \nu}{4 \beta \sqrt{m(l_g^2+\sigma_2^2)} L_gL_F}$ imply $\frac{4  \beta^2m (l_g^2+\sigma_2^2)}{\nu^2}  \frac{\alpha^2 q L_F^2 L_g^2}{S_2}\leq \frac{\overline{\epsilon}^2}{4}$. 
		Also, we obtain $\frac{4 \Delta }{Kq \alpha} 
		\leq \frac{\overline{\epsilon}^2}{4}$ from $K=\lceil 16\overline{\epsilon}^{-2}q^{-1}\alpha^{-1} \Delta  \rceil$. Thus we have the desired result.
	\end{proof}
	
	Under the conditions of~\Cref{cor:boundug2}, we have~\cref{eq:kktvio} with the vanishment of the term $\frac{\sigma^2}{S_1}$. Thus the following theorem can be {established} immediately. 
	
	\begin{theorem}
		\label{cor:comple}
		Under the conditions of~\Cref{cor:boundug2},  given $\overline{\epsilon}>0$, let $\{\vx^{(k)}\}_{k=0}^{Kq-1}$ and $\{\vu^{(k)}\}_{k=0}^{Kq-1}$ be generated by~\Cref{alg:dad2-2} with $S_1=N$ {(i.e., taking the whole data)}, and other parameters set to those in~\cref{eq:q}--\cref{eq:K}. Then~\cref{eq:bd-sum-grad} holds.
	\end{theorem}
	
	\subsection{Complexity result}\label{sec:all}
	In this subsection, we demonstrate that~\Cref{alg:dad2-2} can find an $(\epsilon,\epsilon)$-KKT point for problem~\cref{eq:model},  
	{using the results in \Cref{thm:comple} and \cref{thm:main}}.  
	From these two theorems, we can directly {derive} the complexity of~\Cref{alg:dad2-2} to produce a near-$\overline{\epsilon}$ stationary solution of~\cref{eq:smoothed} in expectation by choosing an iterate from $\{\vx^{(k)}\}_{k=0}^{Kq-1}$ uniformly at random.
	\begin{corollary}\label{cor:complexity-both}
		Under~\Cref{ass:problemsetup} and \Cref{ass:gradient},~\Cref{alg:dad2-2} can produce a near-$\overline{\epsilon}$ stationary solution of~\cref{eq:smoothed} in expectation with $Kq=O\left(\overline{\epsilon}^{-4}\max\{1, \frac{\overline{\epsilon}}{\nu}\}\right)$ stochastic subgradients and $K(S_1 + (q-1)S_2)=O\left(\overline{\epsilon}^{-5}\nu^{-1}\max\{1, \frac{\overline{\epsilon}}{\nu}\}\right)$ stochastic function evaluations on $\vg$. If \(\mathbb{P}_{\vg}\) is a uniform distribution over a finite dataset $\Xi$ of size $N$, the number of stochastic function evaluations becomes $O\left(N + \sqrt{N}\overline{\epsilon}^{-4}\max\{1, \frac{\overline{\epsilon}}{\nu}\}\right)$.
	\end{corollary}
	We then present the complexity results in the following proposition. 
	\begin{proposition}
		\label{cor:convexsto}
		Suppose that~\Cref{ass:problemsetup}, \cref{def:2} with $B_g>m\rho_g  D^2$, and~\Cref{ass:gradient} hold.
		Then, given $\epsilon>0$, %
		\Cref{alg:dad2-2} can produce an $(\epsilon,\epsilon)$-KKT point of problem~\cref{eq:model} in expectation 
		with 
		{\begin{itemize}
				\item[$\mathrm{(i)}$] $O(\epsilon^{-4})$ OGC/CGC and $ O(\epsilon^{-6})$ CFC if \(\mathbb{P}_{\vg}\) is a general distribution;
				\item[$\mathrm{(ii)}$] $O(\epsilon^{-4})$ OGC/CGC and $O(N+\sqrt{N} \epsilon^{-4})$ CFC if \(\mathbb{P}_{\vg}\) is a uniform distribution over a finite dataset $\Xi$ of size $N$.
		\end{itemize}} 
	\end{proposition}
	
	\begin{proof}
		Choose $\overline{\epsilon}$, $\beta$, and $\nu$ {so} 
		that the conditions in~\cref{eq:epsilon} hold.  In addition, set $S_1$ {as} 
		in~\cref{eq:q} if \(\mathbb{P}_{\vg}\) is a general distribution and $S_1=N$ {(i.e., taking the whole data)}, if \(\mathbb{P}_{\vg}\) is a uniform distribution over a finite dataset~$\Xi$ of size~$N$. Then set $S_2,$  $q$, $\alpha_k$ and $K$ to those in~\cref{eq:q} and~\cref{eq:K}. The claims now directly follow from~\Cref{thm:main}  and~\Cref{cor:complexity-both}.
	\end{proof}
	
	\begin{remark}
		\label{rem:deter} When $\vg$ is deterministic, our algorithm still converges to an $(\epsilon,\epsilon)$-KKT point of problem~\cref{eq:model} in expectation without requiring the condition \mbox{$B_g > m\rho_g  D^2$}. This condition appears only once in our analysis, specifically in the proof of inequality \eqref{eq:hnu}, to ensure well-definedness of parameters $\beta$ and $\nu$ in \eqref{eq:epsilon}. Under this assumption, inequality \eqref{eq:hnu} and the subsequent arguments guarantee that 
		$g_+(\widehat{\vx}^*) \le B_g$, which then enables us to use the Slater-type CQ assumption. 
		In contrast, for the deterministic case, the constraints $g_i(\vx)\le 0, \forall\, i\in[m]$ can be equivalently expressed as $\max_{i\in[m]} g_i(\vx)\le 0$. Leveraging this reformulation, we directly utilize the established convergence result for the special case of $m=1$, detailed in our arXiv first version~\cite{liu2025single}. This prior result ensures that, provided the initial point is sufficiently close to the feasible set $\cX$, all iterates produced by our method satisfy $\max_{i\in[m]} g_i(\widehat{\vx}^{(k)})\le B_g$. Consequently, the assumption $B_g > m\rho_g  D^2$ becomes unnecessary for this case, provided that we can obtain a (near) feasible initial point. 
	\end{remark} 
		\section{Numerical experiments}
			\label{sec:numer}
			
			In this section, we evaluate the empirical
			performance of the proposed~\Cref{alg:dad2-2} and  
			compare 
			to three state-of-the-art
			approaches, including two different   IPPs~\cite{boob2023stochastic,ma2020quadratically} and the SSG in~\cite{huang2023single}.
			All the experiments 
			are performed on iMac with Apple~M1, a 3.2 GHz 8-core processor, and 16GB of RAM running MATLAB R2024a.
			
			\subsection{Tested problems and Datasets}\label{eq:data}
		We conduct experiments on solving two fairness-constrained problems and one Neyman-Pearson classification problem, which are instances of~\cref{eq:model}.

			\textbf{Problem 1.} 
			The first tested problem is a classification problem with 
			ROC-based fairness~\cite{vogel2021learning}. 
			It is formulated in~\cref{eq:fair1}.
			We obtain the minimum $\Phi^*$ and a solution $\vx^*$ by 
            {applying} a stochastic subgradient method to a large   number of iterations on $\min_{ \vx \in\cX}\Phi(\vx)$.
			{Afterward, }
            we set $\kappa_1=0.001 \cdot \Phi^*$, $D=5\|\vx^*\|$, and let $\Theta$ consist of 400 points equally spaced {between $\uline{a}-0.5\left(\overline{a}-\uline{a}\right)$ and $\overline{a}+0.5\left(\overline{a}-\uline{a}\right)$,
            where $\overline{a} = \max _i (\vx^*)^{\top} \mathbf{a}_i$ and $\uline{a} = \min_i (\vx^*)^{\top} \mathbf{a}_i$.}
            {For all the compared methods, we initialize 
            \mbox{$\vx^{(0)}=\vx^*$.}}
            For this problem, we   have {$l_f = \frac{1}{n_{\mathrm{p}}} \sum_{i=1}^{n_{\mathrm{p}}}\|\va_i^{\mathrm{p}}\| + \frac{1}{n_{\mathrm{u}}} \sum_{i=1}^{4n_{\mathrm{u}}}\|\va_i^{\mathrm{u}}\|$, \mbox{$l_g=\frac{1}{n} \sum_{i=1}^{n }\|\va_i\| $}, 
            	$\rho_f= \frac{1}{n_{\mathrm{p}}} \sum_{i=1}^{4n_{\mathrm{p}}}\|\va_i^{\mathrm{p}}\|^2 + \frac{1}{n_{\mathrm{u}}} \sum_{i=1}^{4n_{\mathrm{u}}}\|\va_i^{\mathrm{u}}\|^2$, and  $\rho_g= 0$ from \cite{huang2023single}.} That is, $\Phi(\cdot)-\Phi^*-\kappa_1$ is convex{, and \cref{ass:problemsetup} is satisfied. Moreover, \cref{def:2} holds with parameters $B_g=1$, $\overline{\rho}=\frac{-\max_{i\in[m]} g_i(\vx^*)}{D^2}$,  and
            	$B=\frac{-\max_{i\in[m]} g_i(\vx^*)}{2}$.}

			\textbf{Problem 2.} 
			The second problem is a classification problem with demographic parity~\cite{agarwal2018reductions} and the smoothly clipped absolute deviation (SCAD) regularization term~\cite{fan2001variable}. 
			It is formulated as  
			\begin{equation}
				\min_{\vx: \|\vx\|_{\infty}\leq D} \Phi(\vx) + \lb \|\varphi(\vx)\|_1, \,\, \text{\st}\,\, \Psi_0(\vx)\leq \kappa_2,
			\end{equation}
			where $\varphi:\RR^d\mapsto\RR^d$, $\Psi_0:\RR^d\mapsto\RR$,  $\Phi$ is the same as that in~\cref{eq:fair1}, $D>0$, $\lb > 0$, $\kappa_2>0$ is  a slackness parameter. {The functions $\Psi_0$ and $\varphi$ are defined as follows:}
			\begin{align*}
			&\small{\Psi_{0}(\vx) : = \left| \frac{1}{n_{\mathrm{p}}} \sum_{i=1}^{n_{\mathrm{p}}} \sigma\left(\mathbf{x}^{\top} \mathbf{a}_i^{\mathrm{p}}\right)-\frac{1}{n_{\mathrm{u}}} \sum_{i=1}^{n_{\mathrm{u}}} \sigma\left(\mathbf{x}^{\top} \mathbf{a}_i^{\mathrm{u}}\right)  \right|, \text{ and }}\\ &\small{[\varphi(\vx)]_i=\begin{cases}2\left|x_i\right|, & 0 \leq\left|x_i\right| \leq 1, \\ -x_i^2+4\left|x_i\right|+1, & 1<\left|x_i\right| \leq 2, \\ 3, & \left|x_i\right|>2,\end{cases}
            \quad\text{ for }i=1,2,\ldots,d},
			\end{align*} {where $\varphi$ promotes sparsity in $\vx$.} 
            Here, $\Psi_{0}$ is a continuous relaxation of the demographic parity $ \left| \frac{1}{n_{\mathrm{p}}} \sum_{i=1}^{n_{\mathrm{p}}} \mathbb{I}\left(\mathbf{x}^{\top} \mathbf{a}_i^{\mathrm{p}}\geq 0\right)-\frac{1}{n_{\mathrm{u}}} \sum_{i=1}^{n_{\mathrm{u}}} \mathbb{I}\left(\mathbf{x}^{\top} \mathbf{a}_i^{\mathrm{u}}\geq 0\right)  \right|$.
			For this problem, {we discuss in \cref{appen:2} that its constraint satisfies \cref{def:2}, and from \cite{huang2023single}, we obtain $\rho_f= \rho_g= \max\big\{2\lb, \frac{1}{n_{\mathrm{p}}} \sum_{i=1}^{4n_{\mathrm{p}}}\|\va_i^{\mathrm{p}}\|^2 + \frac{1}{n_{\mathrm{u}}} \sum_{i=1}^{4n_{\mathrm{u}}}\|\va_i^{\mathrm{u}}\|^2\big\}$, and the Lipschitz modulus $l_g = \frac{1}{n_{\mathrm{p}}} \sum_{i=1}^{n_{\mathrm{p}}}\|\va_i^{\mathrm{p}}\| + \frac{1}{n_{\mathrm{u}}} \sum_{i=1}^{4n_{\mathrm{u}}}\|\va_i^{\mathrm{u}}\|$, $l_f = 5\lb +  \frac{1}{n } \sum_{i=1}^{n }\|\va_i \| $.}
			In the experiments, we set $D=5$, $\lb=0.02$, $\kappa_2=0.02$, 
			and 
            $\vx^{(0)}=\vzero$ 
			for all the compared methods.
			
			
			\textbf{Problem 3.} 
			{The third tested problem is a multi-class Neyman-Pearson classification problem formulated in \cref{eq:example1}.  We set $D=0.3$, $\kappa=4.5$, and let $\phi$ be the sigmoid function $\sigma$. Similar to those constants of Problem 2, we have 
			$l_f =  l_g= \max_{i\in[M]}\frac{1}{|\cD_i|} \sum_{\va_i\in\cD_i}\|\va_i\| $, and
			$\rho_f= \rho_g=    \max_{i\in[M]}\frac{1}{|\cD_i|} \sum_{\va_i\in\cD_i}\|\va_i\|^2$. { As discussed in \cref{appen:2}, the constraint of the tested problem also satisfies \cref{def:2}.} }
			
			\textbf{Datasets.} 
			The experiments of Problems 1--2 are conducted on two datasets: \verb|a9a|~\cite{kohavi1996scaling}  
			and \verb|COMPAS|~\cite{angwin2022machine}. 
			The \verb|a9a| dataset contains 48,842 data points with 123 features {and is} used to predict income levels; gender is the protected attribute, and fairness is assessed between males and females.  
			The \verb|COMPAS| dataset includes 6,172 data points and 16 features, focused on predicting recidivism risk, with fairness evaluated between Caucasian and non-Caucasian groups. We split each dataset into two subsets with a ratio of 2:1. The larger set is used as $\mathcal{D}=\left\{\left(\mathbf{a}_i, b_i\right)\right\}_{i=1}^{n}$, and {the smaller subgroup is further partitioned into two subsets based on a binary protected attribute:} $\mathcal{D}_{\mathrm{p}}=\left\{\left(\mathbf{a}^{\mathrm{p}}_i, b^{\mathrm{p}}_i\right)\right\}_{i=1}^{n_{\mathrm{p}}}$ and  $\mathcal{D}_{\mathrm{u}}=\left\{\left(\mathbf{a}^{\mathrm{u}}_i, b^{\mathrm{u}}_i\right)\right\}_{i=1}^{n_{\mathrm{u}}}$. 
						The experiments for Problem 3 are conducted using the \verb|MNIST| dataset~\cite{lecun1998mnist}, which contains 60,000 grayscale images of handwritten digits (0–9), each image having 
			28$\times$28 pixels and an associated ground-truth label.
			We set $M=10$ and define each subset $\cD_i$ as the collection of images corresponding to the image labeled by digit $\mathrm{mod}(i,10)$, for $i\in[M]$.
			
			\subsection{Implementation details}\label{sec:imple}
			For our method, we use $\beta =10$ and $\nu = 10^{-5}$ as the default setting.
			We implement both deterministic and stochastic versions of our method (denoted as 3S-Econ-D and 3S-Econ-S, respectively).
			For 3S-Econ-D, we set $q=1$ and $S_1=S_2=n$ (i.e., taking the whole data)  
			and {select} $\alpha_k=10^{-2}$  as our default setting.
			For 3S-Econ-S, we set $S_1=n$ and $S_2=q=\lceil\sqrt{n}\rceil$   
			and 
            {select} $\alpha_k = \frac{1}{100\lceil \sqrt{k/q}\rceil}$ as our default setting\footnote{To determine the default settings of the stepsize, we test 3S-Econ-D and 3S-Econ-S on Problem~1 using the a9a dataset, {with varying choices of  \( \alpha_k \). Specifically, for all $k\ge 0$, we choose  \( \alpha_k\) from  $\{0.001, 0.01, \frac{1}{100\lceil \sqrt{k/q}\rceil}, \frac{1}{1000\lceil \sqrt{k/q}\rceil} \} $}. Through our experiments, we observe that 3S-Econ-D achieves the best performance with  \(\alpha_k = 0.01, \forall\, k\). Similarly, for 3S-Econ-S,   \(\alpha_k = \frac{1}{100\lceil \sqrt{k/q}\rceil}, \forall\, k,\) yields the best performance. Based on these observations, we fix these parameters as the default settings and apply them across all datasets and tested problems.
			}.
			
			We adopt the tuning techniques in~\cite{boob2023stochastic,ma2020quadratically} for the two distinct  IPP methods and the techniques in~\cite{huang2023single} for the SSG. 
			A strongly convex constrained subproblem is solved approximately in each outer iteration of the IPPs.  For this subproblem, we employ the SSG from~\cite{huang2023single} and the ConEx in~\cite{boob2023stochastic} as inner solvers. 
			We refer to the two implemented IPPs as IPP-SSG and IPP-ConEx, respectively.
			The stepsize~$\alpha_k$ for IPP-SSG is selected from $\{2\times 10^{-4}, 5\times 10^{-4}, 10^{-3}, 5\times 10^{-3}\}$. Following the notation in~\cite{ma2020quadratically},  we set $\theta_t = \frac{t}{t+1}, \eta_t = c_1(t+1)$, and $\tau_t = \frac{c_2}{t+1}$ for IPP-ConEx, where $c_1$ is chosen from $\{20, 50, 100, 200\}$ and $c_2$ from $\{0.002, 0.005, 0.01, 0.02\}$. The optimal parameter set  for these methods is selected based on the smallest objective value after 5000 iterations.

			We report 
			the objective function value (FV), i.e., $f(\vx)$, the constraint violation (CVio), i.e., { $\|[\vg(\vx)]_+\|_1$}, 
			and the stationarity violation (SVio). 
            {At each iteration,} we solve 
			$$
			\min_{ \vx } \big\{f(\vx) + \rho_f \|\vx -\vx^{(k)}\|^2, \,\text{\st}\,g_i(\vx) + \rho_g \|\vx -\vx^{(k)}\|^2\leq 0, \, i\in[m]\big\}
			$$
			to a desired accuracy to obtain $\widehat{\vx}^{(k)}$ and then use 
			$\|\widehat{\vx}^{(k)}-{\vx}^{(k)}\|$ to measure SVio
			\footnote{Specifically, we solve the strongly convex problem by the SSG~\cite{huang2023single} to a tolerance of $10^{-2}$.
			}. 
		
		{	
			\subsection{Performance of \mbox{3S-Econ-D} with different $\beta$ and $\nu$}\label{sec:algbeta} 
			In this subsection, we investigate the numerical performance of \mbox{3S-Econ-D} with different values 
			of \(\beta\) and~\(\nu\) on solving Problem~1 by using the \texttt{a9a} dataset, respectively.  We vary the penalty parameter $\beta \in \{1,10,100,1000\}$ with fixed $\nu = 10^{-5}$, and subsequently vary the smoothing parameter $\nu \in \{10^{-5}, 10^{-4}, 10^{-3}, 10^{-2}\}$ with fixed $\beta=10$.  The results are depicted in~\cref{fig:betanu}.  From the figure, we observe that 		 
	when \(\beta=1\), the CVio remains relatively large, indicating inadequate penalization,  
			while excessively large \(\beta\)  
			slows down convergence. Among the varying values, $\beta=10$ yields the most balanced trade-off between convergence speed and constraint satisfaction. 
			In addition, we observe that the  
			CVio closely aligns with  
			\(\nu\), consistent with our theoretical results.  
			On all subsequent experiments, we select $\beta=10$ and \(\nu=10^{-5}\) as our default settings.

\begin{figure}  
				\begin{center}
				\begin{tabular}{ccc}
					\includegraphics[width=0.31\textwidth]{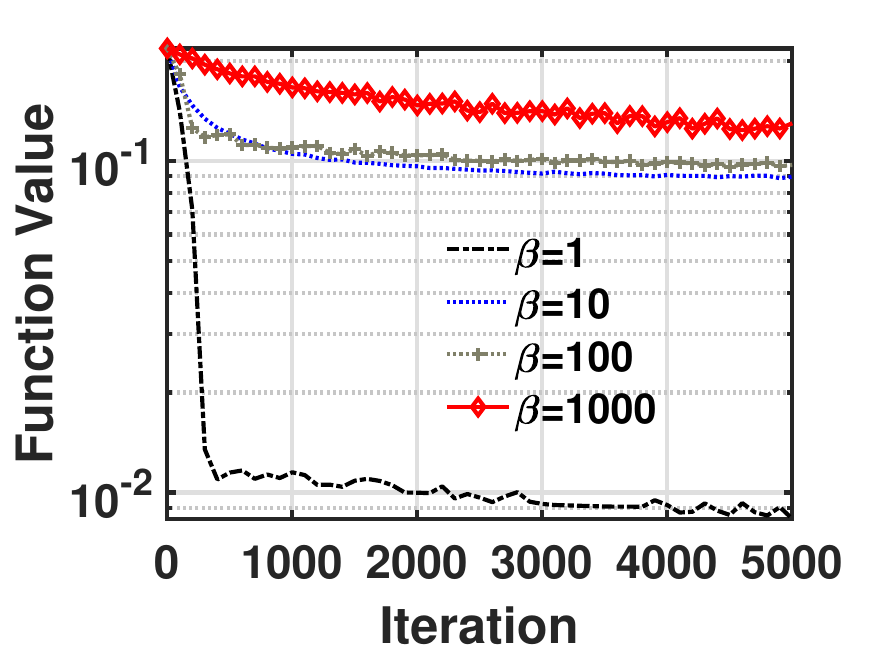} &
					\includegraphics[width=0.31\textwidth]{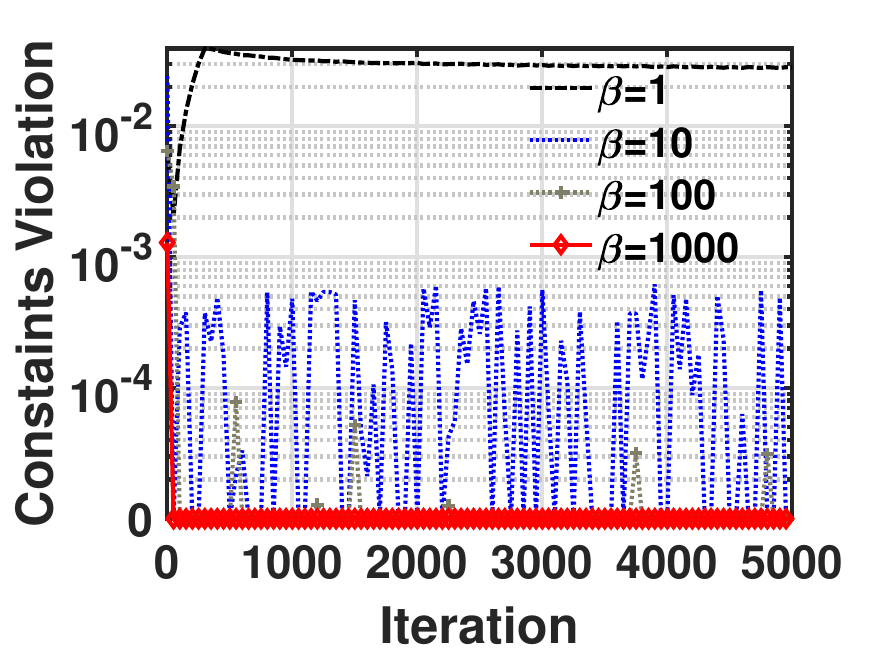}&
					\includegraphics[width=0.31\textwidth]{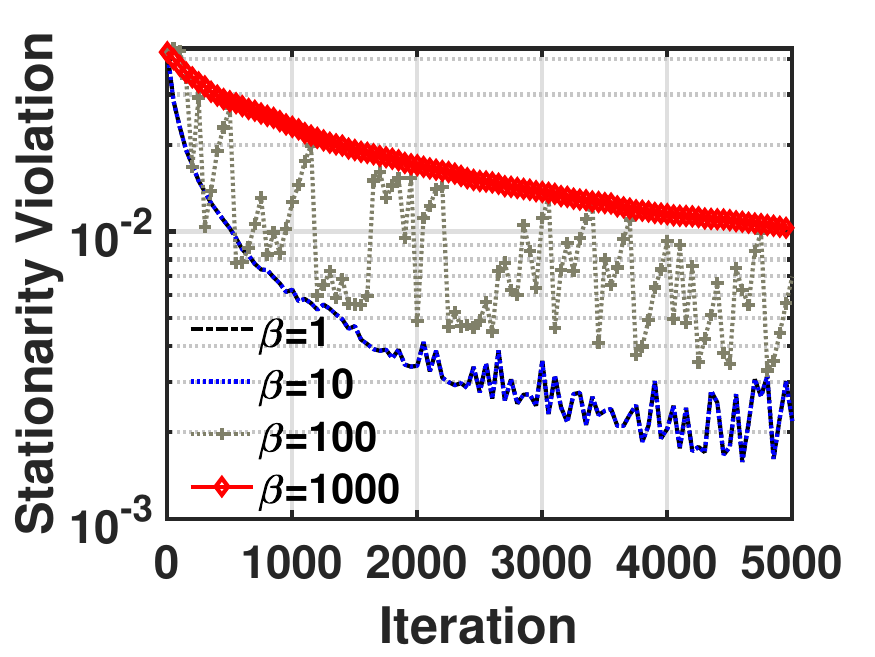} \\
					\includegraphics[width=0.31\textwidth]{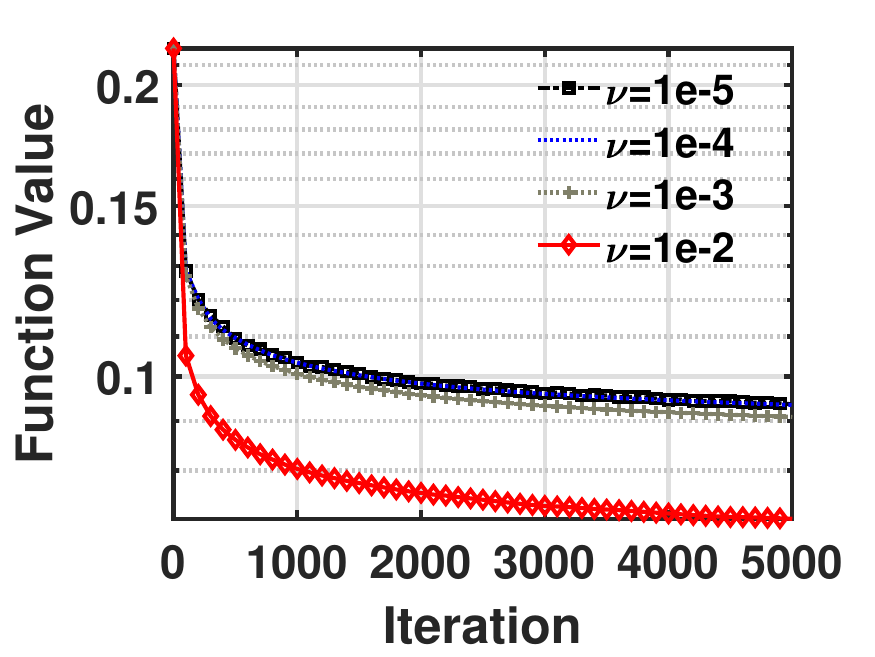} &
					\includegraphics[width=0.31\textwidth]{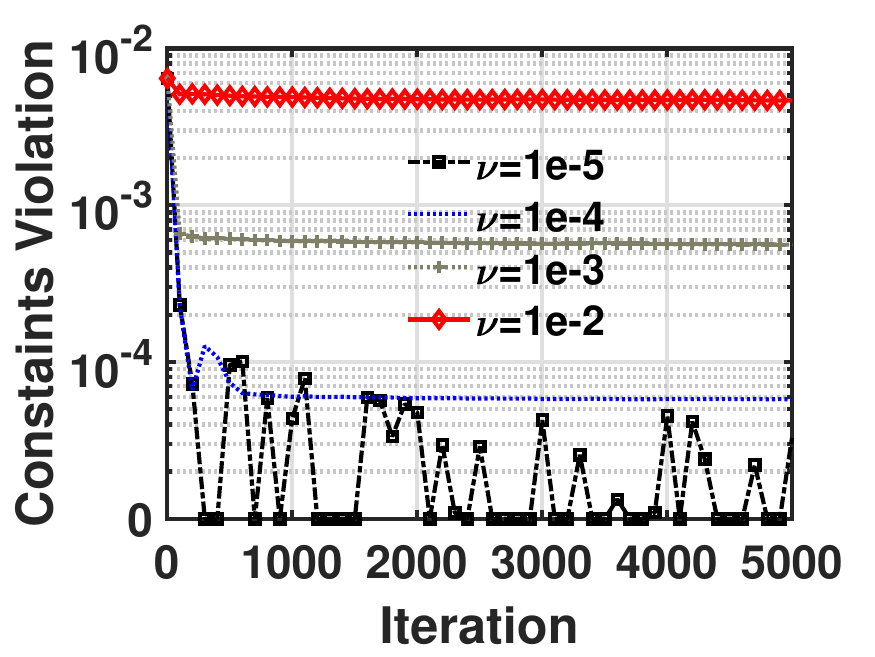} &
					\includegraphics[width=0.31\textwidth]{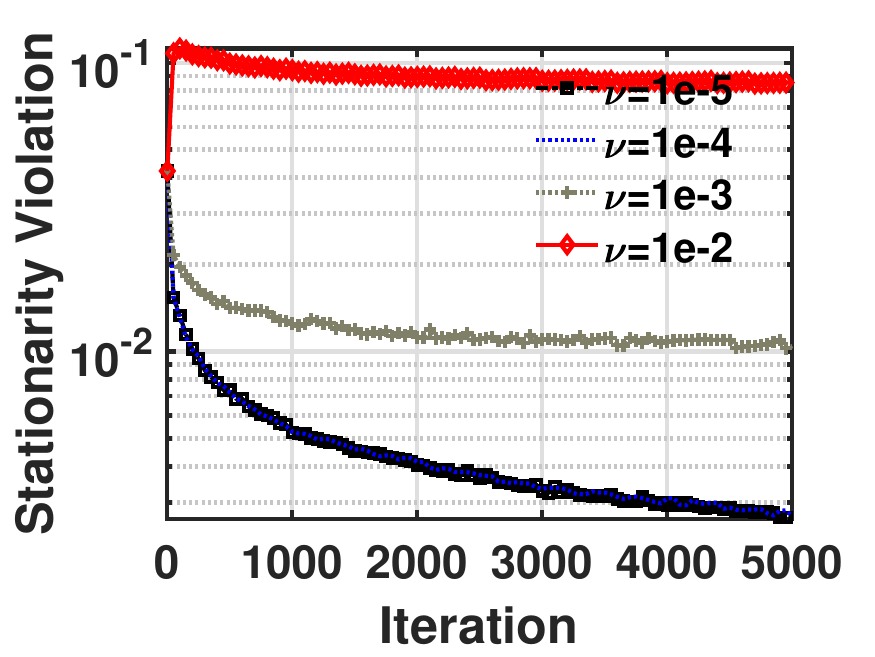}
					\end{tabular}
				\end{center}
				\caption{Performance of \mbox{3S-Econ-D} with different $\beta$ and $\nu$ on solving Problem 1 by using the \texttt{a9a} dataset. Top row: different values of $\beta$ and fixed $\nu = 10^{-5}$; Bottom row: different values of $\nu$ and fixed $\beta=10$.}\label{fig:betanu}
			 
			\end{figure}
}
			
			\subsection{Comparisons among different methods}
			In~\cref{fig:all_deter_pro2}, we compare our \mbox{3S-Econ-D} method with three state-of-the-art deterministic approaches on solving Problems~1--2 {over iterations\footnote{Here, for the two IPP methods, the iteration count refers specifically to the inner-loop iterations.}}.  
			It is clear that  3S-Econ-D  
			outperforms the others in FVs and SVio values, not only converging faster but also achieving lower errors. In terms of CVio,  
			3S-Econ-D shows more fluctuations. Nevertheless, we observe that it remains feasible for most of the iterations. Since for the deterministic version, we can explicitly check the feasibility. By keeping feasible iterates, 3S-Econ-D will  significantly outperform the other three deterministic methods.  
			\begin{figure} 
				\begin{center}
					\subfloat[Problem 1, a9a]{\includegraphics[width=44mm]{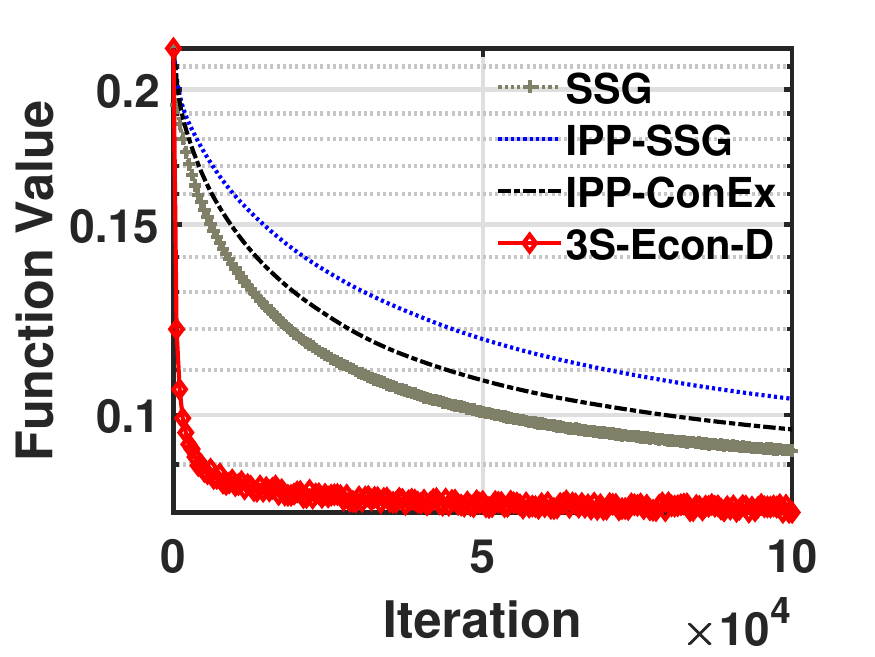}}
					\subfloat[Problem 1, a9a]{\includegraphics[width=44mm]{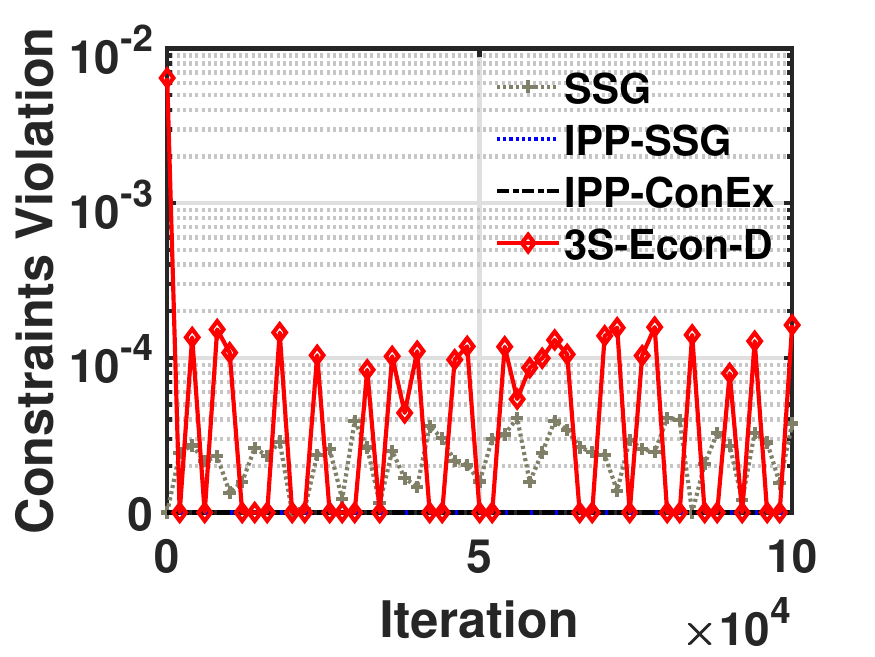}}
					\subfloat[Problem 1, a9a]{\includegraphics[width=44mm]{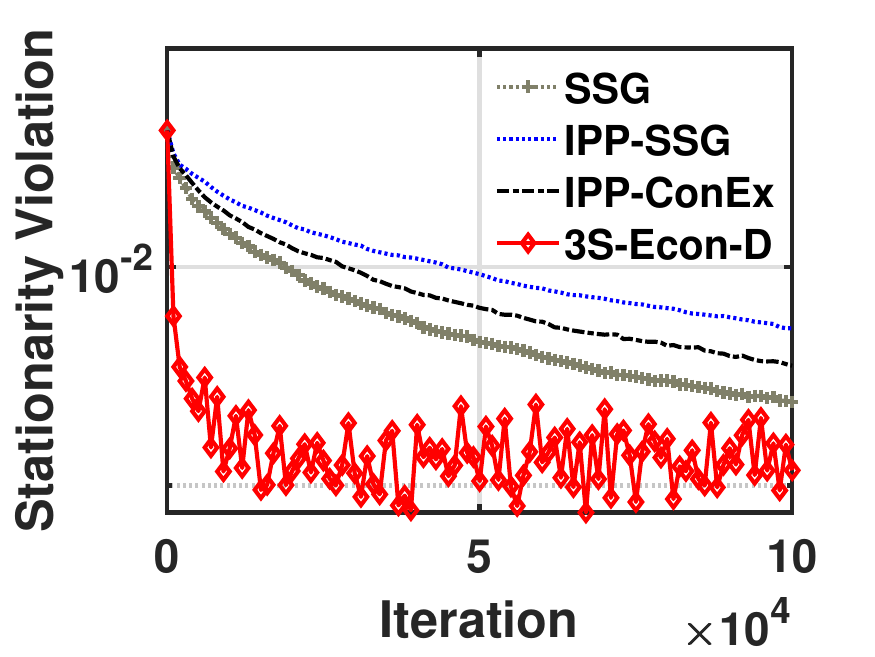}}
					\vspace{-3mm}
					\\
					\subfloat[Problem 1, COMPAS]{\includegraphics[width=44mm]{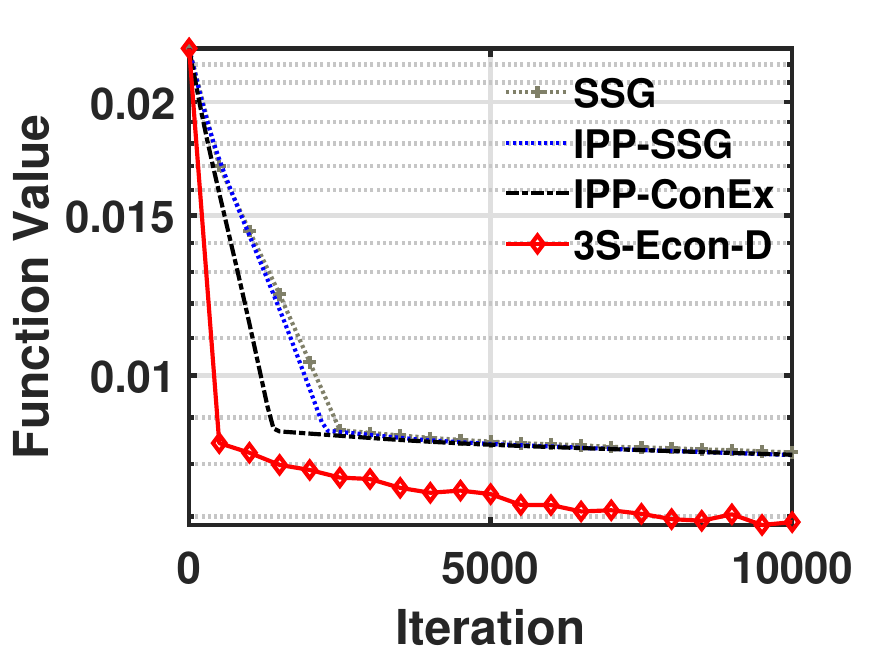}}
					\subfloat[Problem 1, COMPAS]{\includegraphics[width=44mm]{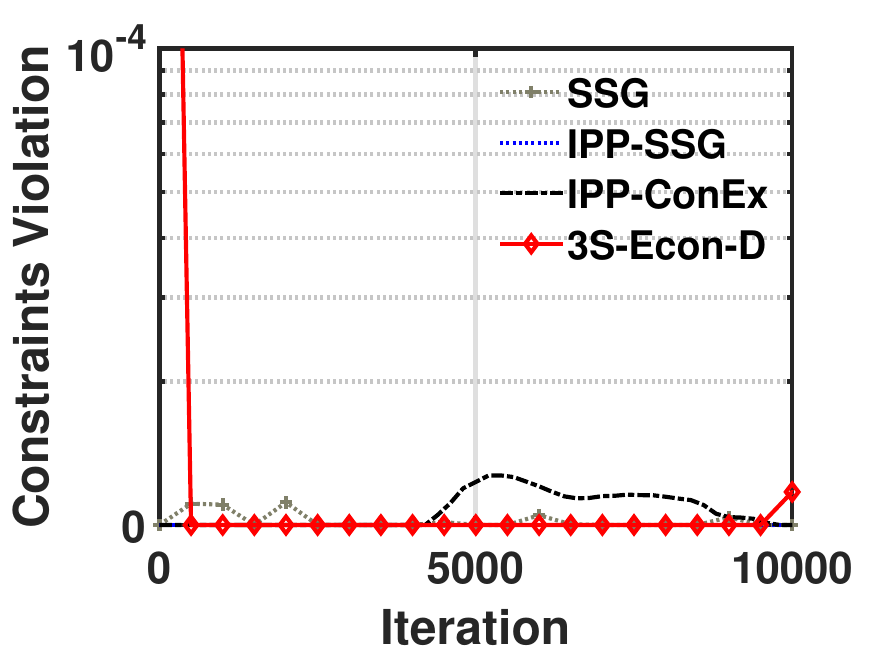}}
					\subfloat[Problem 1, COMPAS]{\includegraphics[width=44mm]{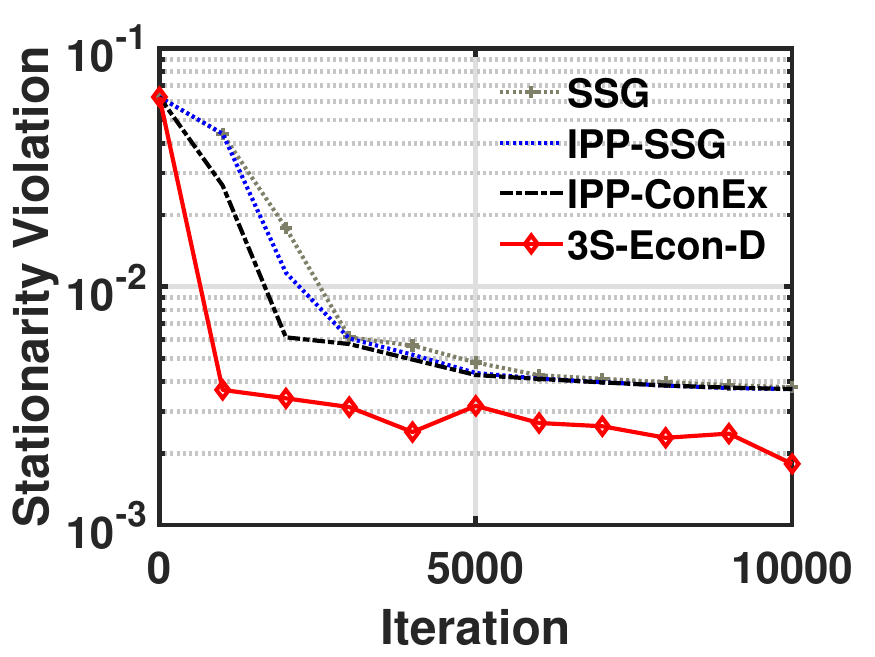}}
					\vspace{-3mm}
					\\         
			
					\subfloat[Problem 2, a9a]{\includegraphics[width=44mm]{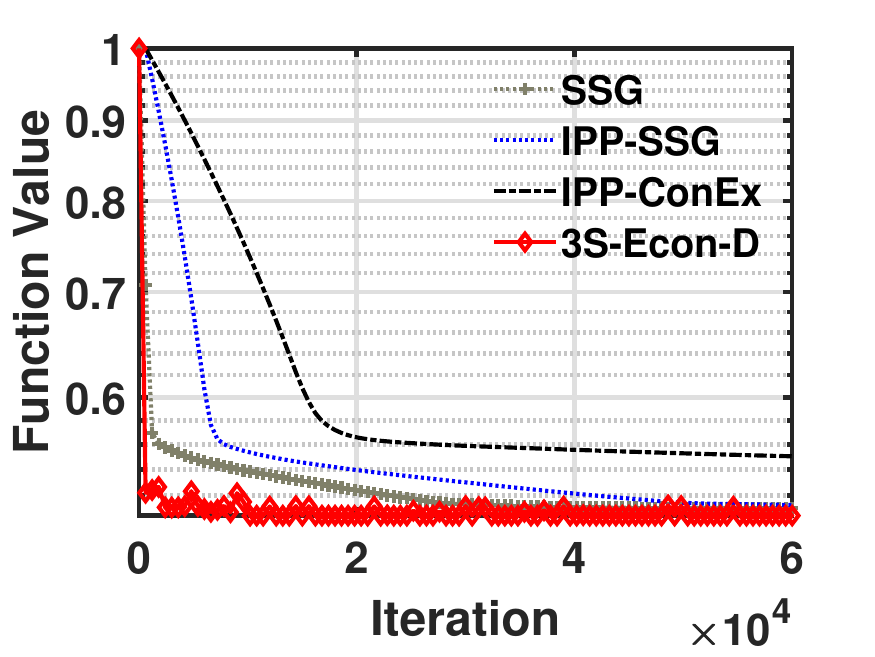}}
					\subfloat[Problem 2, a9a]{\includegraphics[width=44mm]{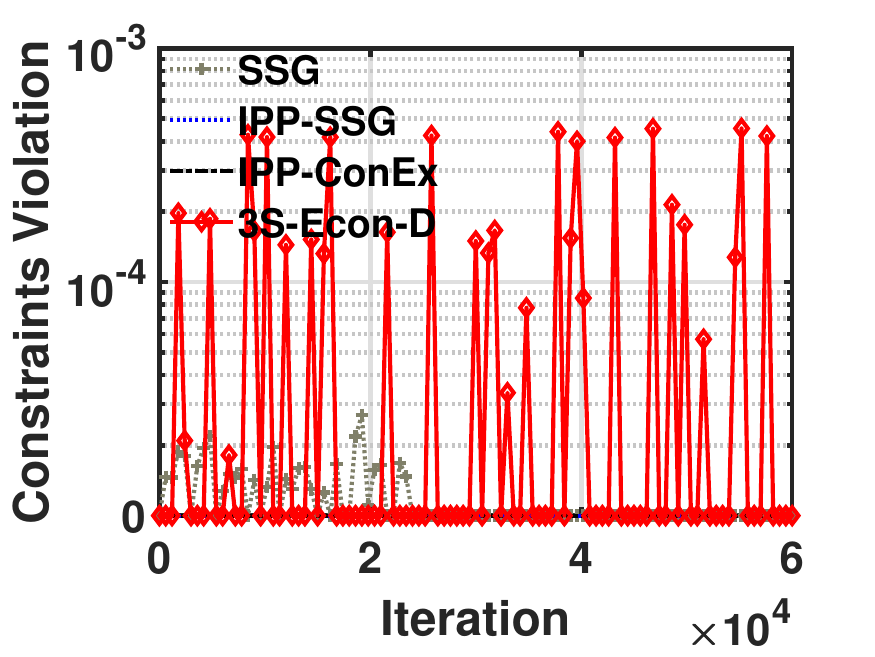}}
					\subfloat[Problem 2, a9a]{\includegraphics[width=44mm]{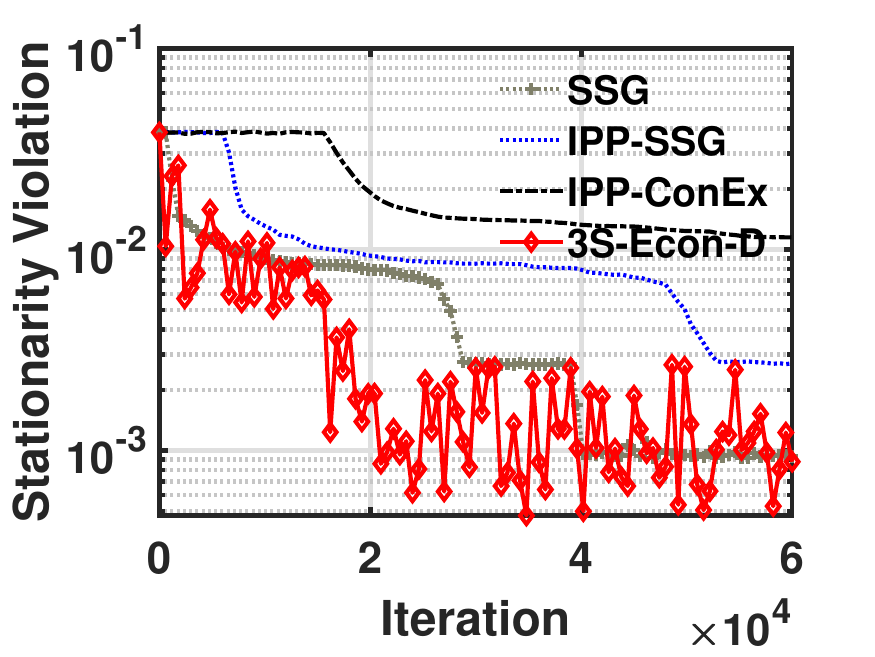}}
					\vspace{-3mm}
					\\
					\subfloat[Problem 2, COMPAS]{\includegraphics[width=44mm]{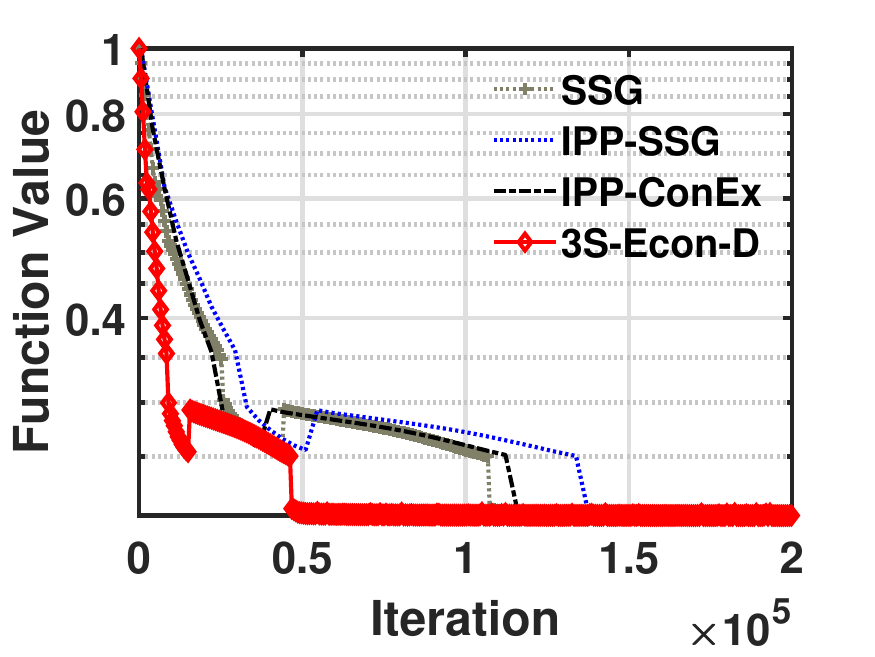}}
					\subfloat[Problem 2, COMPAS]{\includegraphics[width=44mm]{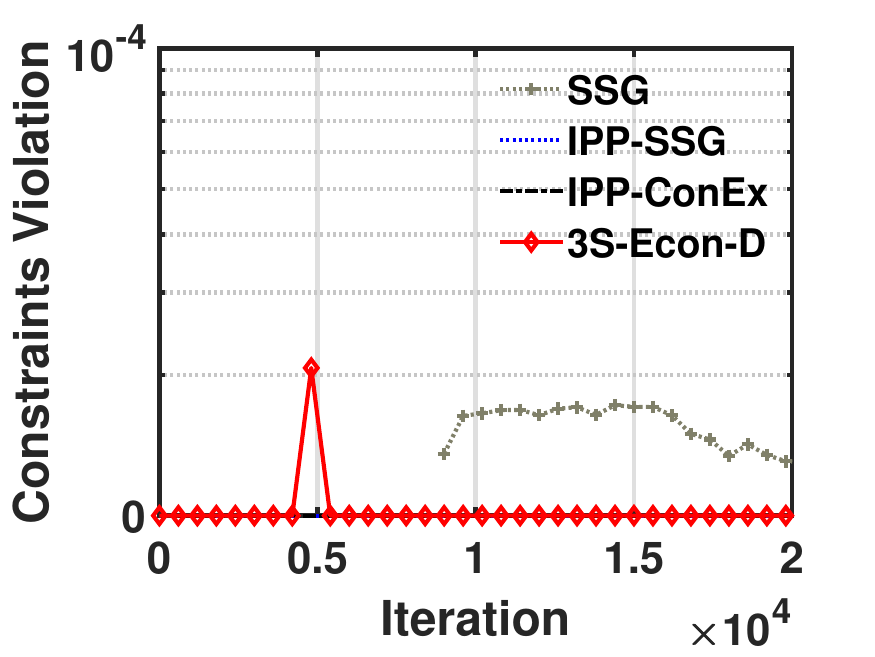}}
					\subfloat[Problem 2, COMPAS]{\includegraphics[width=44mm]{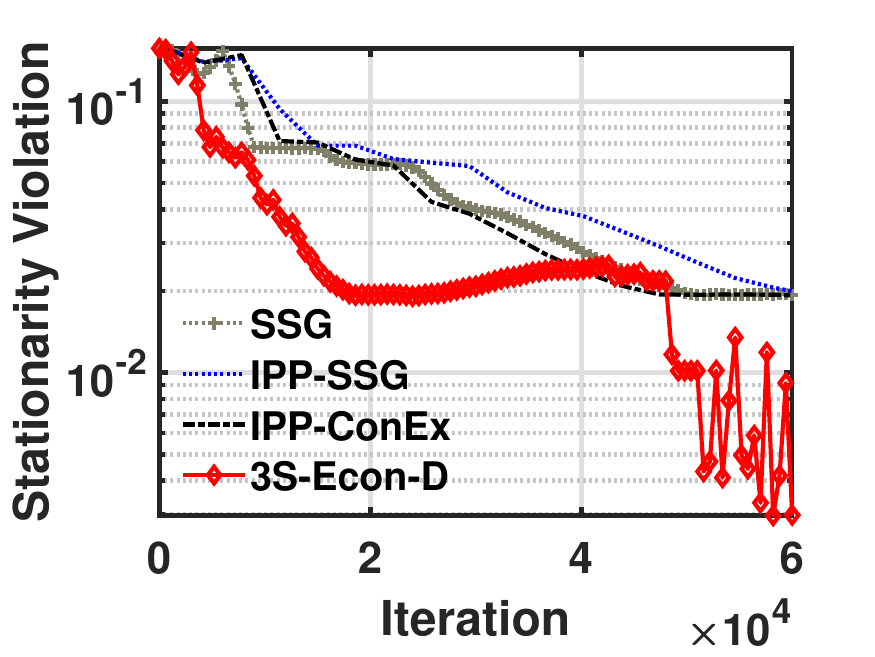}}
					\vspace{-3mm}
				\end{center}
				\caption{Comparisons between 3S-Econ-D (deterministic version of our method) and  other approaches: SSG~\cite{huang2023single}, 
					IPP-SSG~\cite{ma2020quadratically}, and IPP-ConEx~\cite{boob2023stochastic}.}\label{fig:all_deter_pro2}
			\end{figure}

			To evaluate the numerical behavior of the stochastic version of our method (denoted as 3S-Econ-S), we first compare it with 3S-Econ-D.
			The FV and CVio values on Problems 1-2 with two datasets are shown in~\cref{fig:all_sto_pro1}, 
			``data passes'' for the \(x\)-axis label is a shorthand for ``data passes ($\vg$)'' (DP ($\vg$)), 
			referring to the number of times to 
			access all the data involved in the constraint function. 
			Despite more oscillations due to the stochasticity, 
			3S-Econ-S can produce comparable results to 
			3S-Econ-D with significantly fewer data passes. 
			\begin{figure}[h] 
				\begin{center}
					\subfloat{\includegraphics[width=33mm]{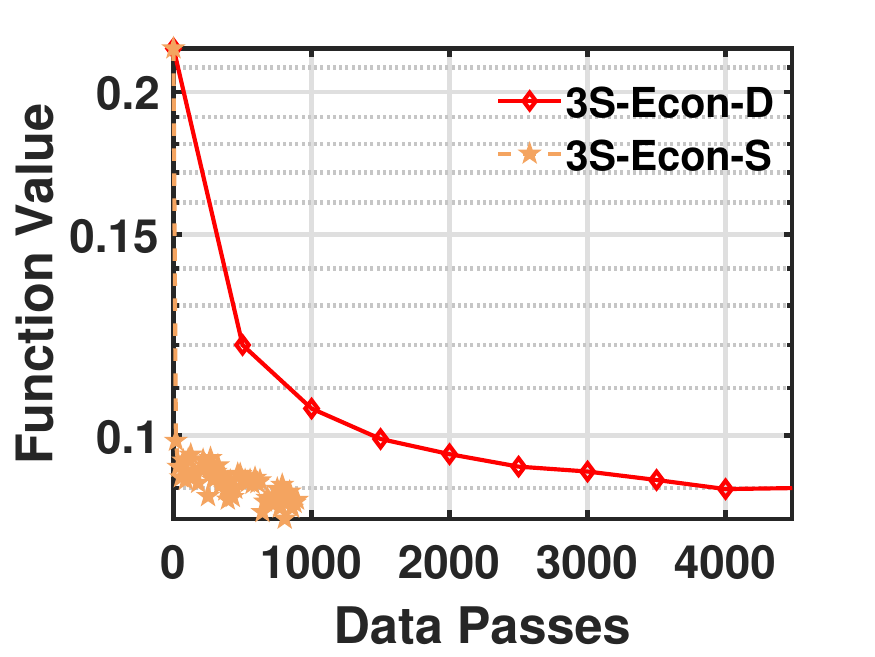}}
					\subfloat{
						\includegraphics[width=33mm]{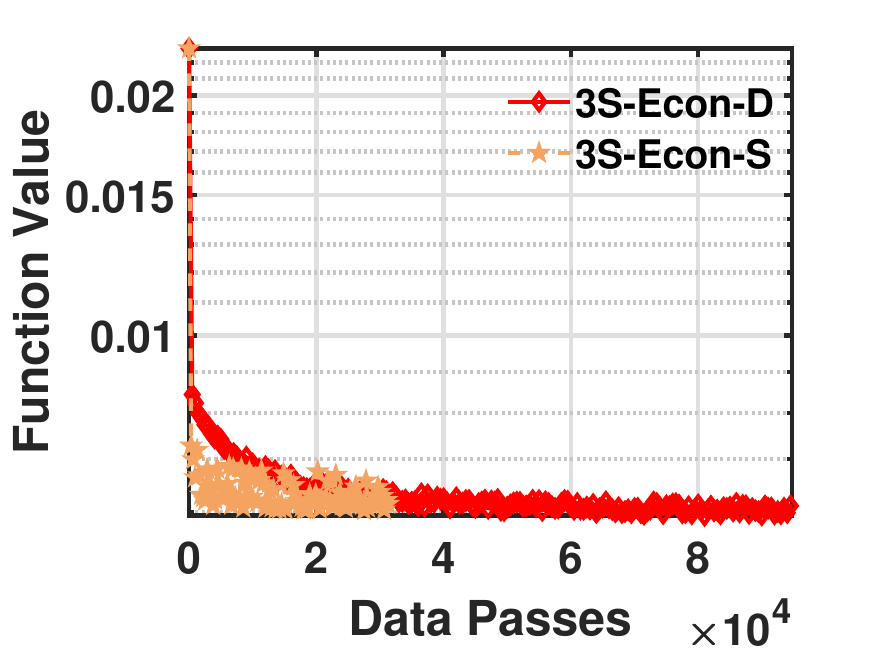}}
					\subfloat{\includegraphics[width=33mm]{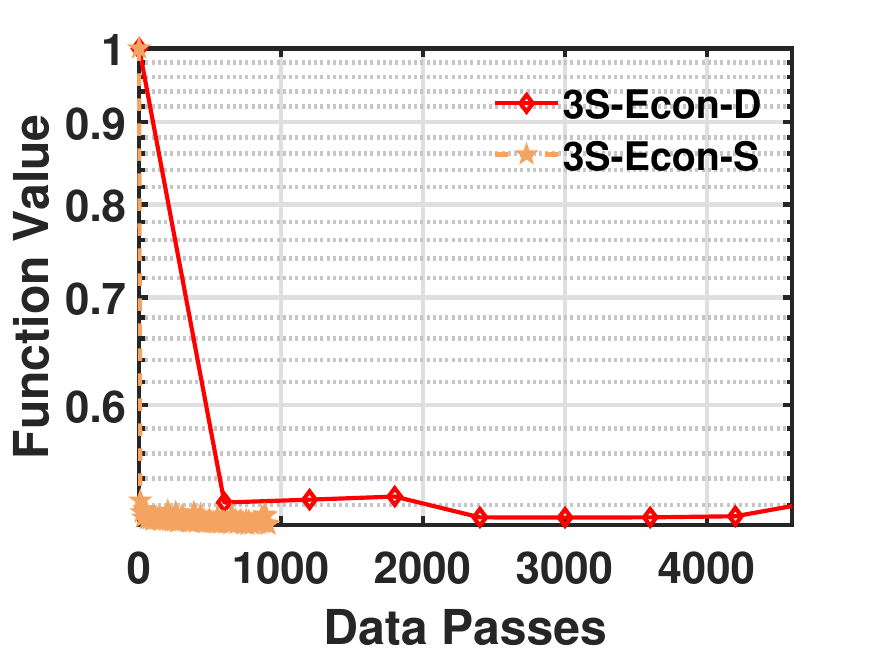}}
					\subfloat{\includegraphics[width=33mm]{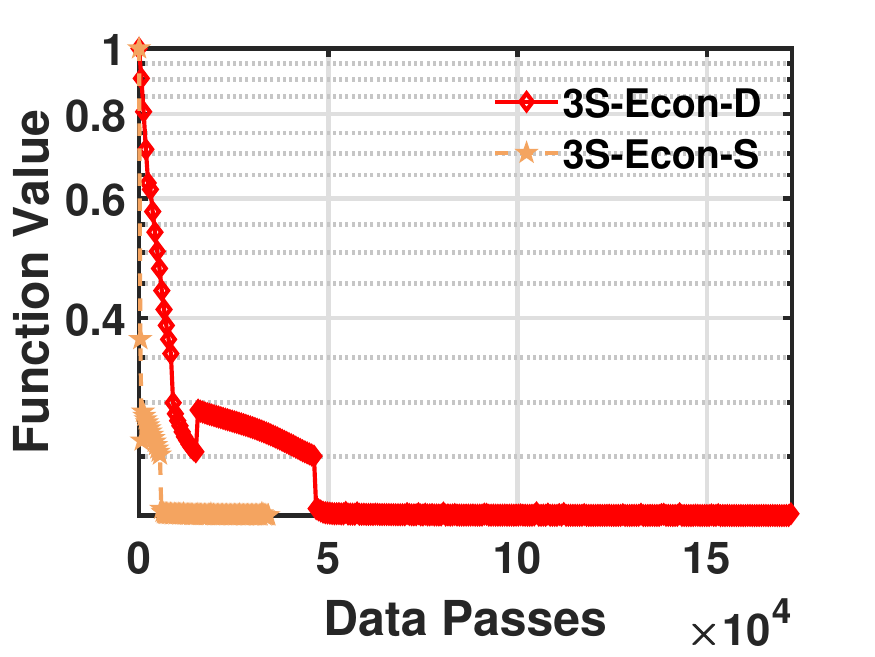}}
					\\
					\setcounter{subfigure}{0}
					\subfloat[Prob. 1, a9a]{\includegraphics[width=33mm]{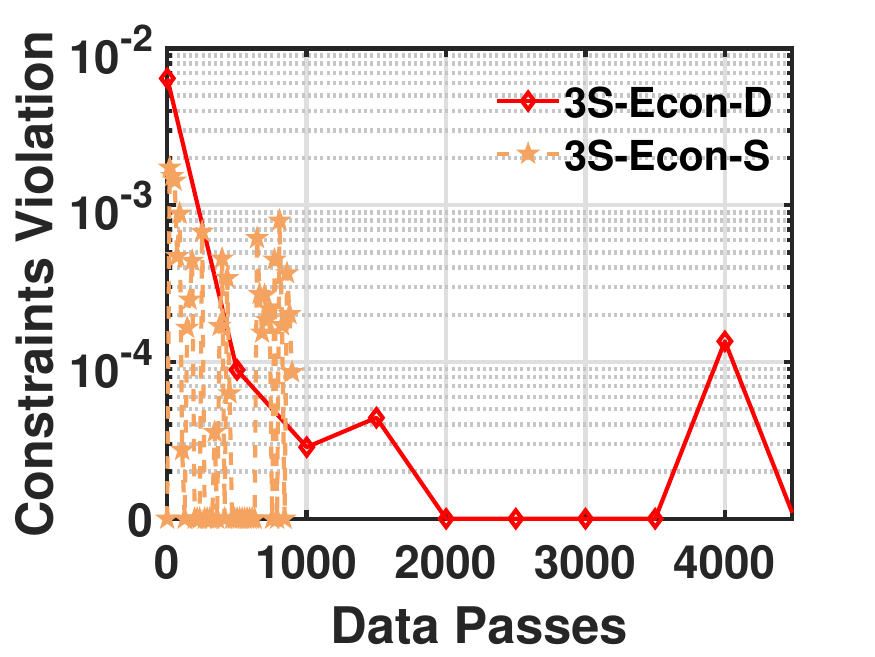}}
					\subfloat[Prob. 1, COMPAS]{
						\includegraphics[width=33mm]{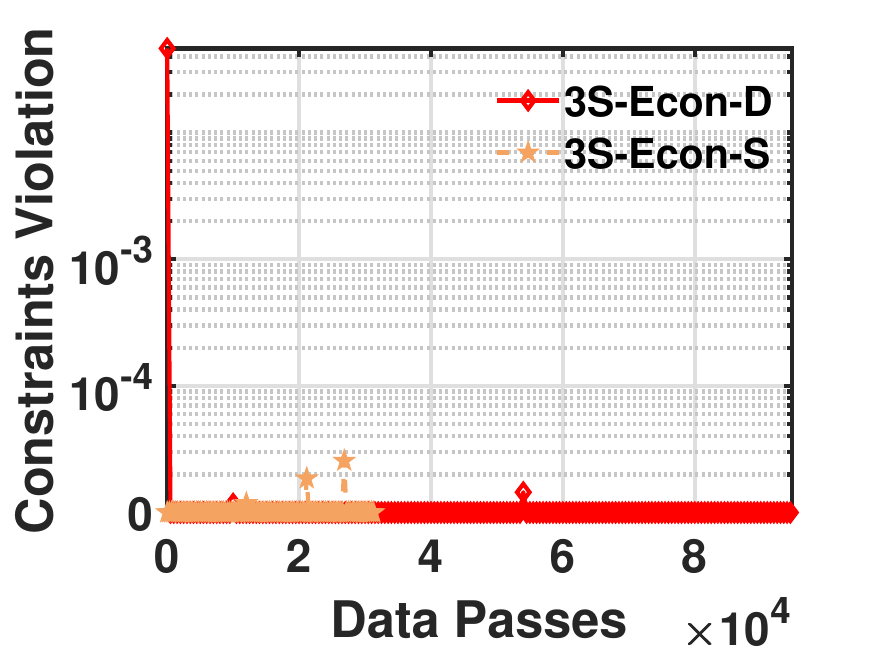}}
					\subfloat[Prob. 2, a9a]{\includegraphics[width=33mm]{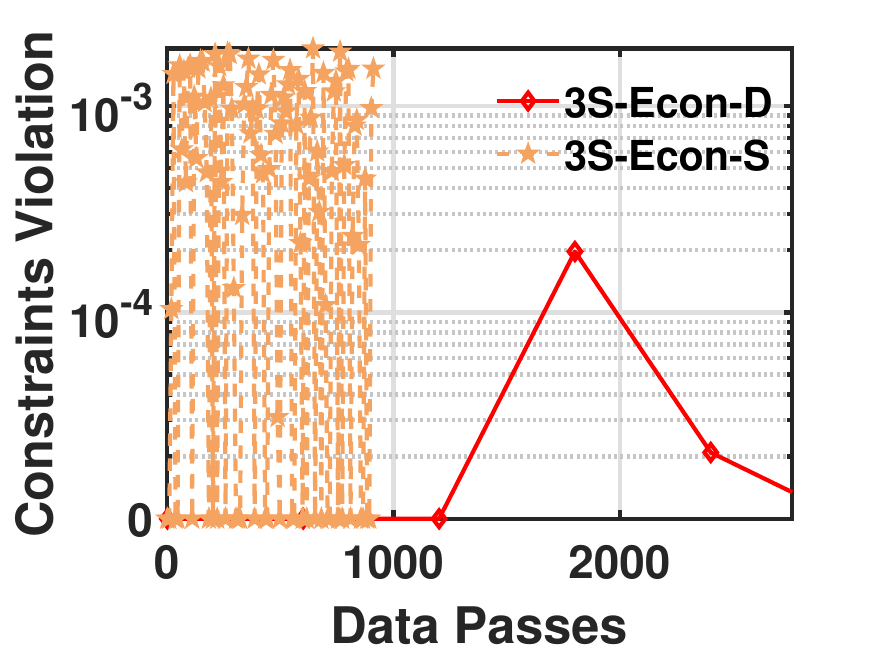}}
					\subfloat[Prob. 2, COMPAS]{\includegraphics[width=33mm]{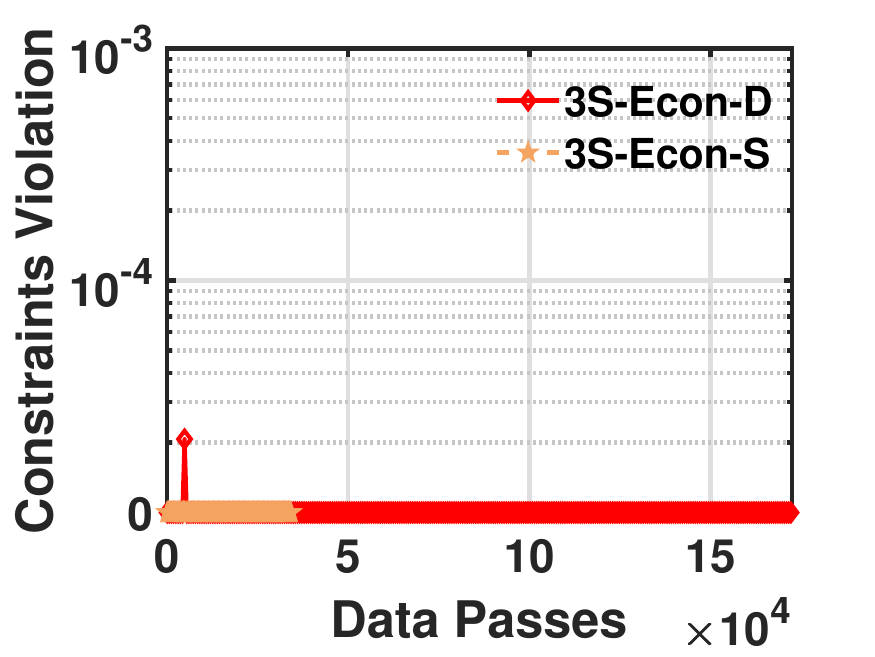}}
				\end{center}
				\caption{Comparisons between the deterministic and stochastic versions of our method}\label{fig:all_sto_pro1}
			\end{figure}
			
			In~\cref{tab:2}, we show more comparison results of 3S-Econ-D and 3S-Econ-S 
			with IPP-SSG, IPP-ConEx, SSG, and the stochastic version of SSG (denoted as \mbox{SSG-S})\footnote{It is worth noting that the stochastic versions of IPP-SSG and IPP-ConEx produce numerical results inferior to their deterministic counterparts when solving Problems 1--2. Thus, we only present results for their deterministic versions.
				This observation can be explained by the fact that, although {their stochastic variants} 
                utilize the stochastic subgradient of \(f\) rather than the full subgradient, both methods still require accessing all data in the constraints during each iteration to access a constraint stochastic subgradient and the constraint {deterministic} function value. 
			}.
			We terminate all the deterministic methods if SVio is less than \(10^{-3}\), while for the stochastic algorithms, the threshold is set to \(5\times 10^{-3} \).
			In addition, we impose a maximum ``DP ($\vg$)'' of 200,000 for Problem~1 and 720,000 for Problem 2. 
			Since computing~\(\widehat{\vx}^{(k)}\) with high precision at each iteration $k$ is time-consuming, we calculate SVio every few steps. 
			Upon termination, our proposed methods demonstrate superior performance in terms of both FV and SVio compared to other tested algorithms in most cases. In terms of ``DP ($\vg$)'', which consistently dominates ``DP ($f$)'' for all methods, we observe that for Problem~1, 3S-Econ-D achieves speedups ranging from 1.5 to 5 times compared to other deterministic algorithms, while 3S-Econ-S delivers even more impressive speedups, approximately 108 to 466 times faster.
			When considering the metric ``(CPU)  Time (s)'', both 3S-Econ-D and 3S-Econ-S remain significantly faster than their counterparts. Although the acceleration factor of 3S-Econ-S in terms of runtime is less pronounced compared to that measured by \mbox{``DP ($\vg$)''}, this discrepancy primarily arises from implementation-related overhead inherent in MATLAB, such as  matrix-vector multiplications.


			\begin{table}\footnotesize
				\centering
				\begin{tabular}{|@{}c|c|c|c|c|c|c|c|c|}
					\hline
					&  & measures & SSG  & 
					IPP-SSG &
					IPP-ConEx &
					3S-Econ-D&
					SSG-S 
					&
					3S-Econ-S  \\
					\hline
					\multirow{12}{*}{\rotatebox{90}{Problem 1}} &  \multirow{6}{*}{\rotatebox{90}{a9a}}  & iteration & 1.00e+05 & 1.00e+05 & 1.00e+05 & 1.50e+04 & 1.99e+05 & 8.24e+04
					\\
					& &DP ($f$)  & 1.00e+05 & 1.00e+05 & 1.00e+05 & 1.50e+04  & 1.10e+03 & 455 
					\\
					& &DP ($\vg$)  & 2.00e+05 & 2.00e+05 & 2.00e+05 & 3.00e+04  & 2.00e+05 & 910 \\
					& &{Time (s)}  & 1.38e+03 & 1.37e+03 & 3.70e+03 & 582  & 1.03e+03 & 537 \\
					&	& FV& 9.26e-02 & 1.03e-01 & 9.70e-02 & 8.47e-02 &8.08e-02 &  8.48e-02 \\
					&	& CVio& 3.76e-05 & 6.57e-06 &  0 & 1.41e-04 &6.34e-04 & 0 \\
					&	& SVio& 2.41e-03 & 5.25e-03 & 3.60e-03 & 9.56e-04 & 1.36e-02 & 4.69e-03 \\
					\cline{2-9}
					&  \multirow{6}{*}{\rotatebox{90}{COMPAS}}& iteration  & 1.00e+05 & 1.00e+05 & 1.00e+05 & 3.70e+04 & 1.96e+05 & 6.00e+04 \\
					& &DP ($f$)  & 1.00e+05 & 1.00e+05 & 1.00e+05 & 3.70e+04  & 3.06e+03 & 938 
					\\
					&  &DP ($\vg$)& 2.00e+05 & 2.00e+05 & 2.00e+05 & 7.40e+04 & 2.00e+05 & 1.85e+03\\
					& &{Time (s)}  & 285 & 178 & 236 & 127 & 273 & 32.9 \\
					&	& FV& 6.91e-03 & 6.79e-03 & 6.85e-03 & 6.31e-03 & 6.49e-03 & 6.41e-03 \\
					&	& CVio& 9.27e-06 & 5.65e-06 & 2.17e-06 &  1.49e-05 & 5.66e-04 & 0\\
					&	& SVio& 1.65e-03 & 1.52e-03 & 1.60e-03  & 9.99e-04 & 1.43e-02   &4.77e-03 \\
					\hline
					\multirow{12}{*}{\rotatebox{90}{Problem 2}} &  \multirow{6}{*}{\rotatebox{90}{a9a}} &iteration& 4.02e+04 & 3.60e+05 & 3.60e+05 & 2.10e+04 & 5.10e+04 & 9.90e+03 \\
					& &DP ($f$)  & 4.02e+04 & 3.60e+05 & 3.60e+05 & 2.10e+04 & 282 & 55
					\\
					&	  &DP ($\vg$)  & 8.04e+04 & 7.20e+05 & 7.20e+05 & 4.20e+04 & 5.13e+04 & 110 \\
					& &{Time (s)} & 305 & 2.35e+03 & 7.76e+03 & 217 & 134 & 35.6 \\
					&	& FV& 5.11e-01 & 5.07e-01 & 5.14e-01 & 5.05e-01 & 5.07e-01 & 5.10e-01 \\
					&	& CVio& 5.11e-05 & 0  & 0 & 0 & 0 &  0 \\
					&	& SVio& 9.58e-04 & 2.72e-03 & 1.15e-02 & 8.58e-04 & 1.99e-03 & 2.19e-03 \\
					\cline{2-9}
					&  \multirow{6}{*}{\rotatebox{90}{COMPAS}} & iteration & 2.89e+05 & 3.60e+05 & 3.10e+05 &  1.54e+05 & 1.15e+05 & 1.39e+05 \\
					& &DP ($f$)   & 2.89e+05 & 3.60e+05 & 3.10e+05 &  1.54e+05 & 1.80e+03 &  2.17e+03
					\\
					&  &DP ($\vg$)  & 5.78e+05 & 7.20e+05 & 6.20e+05 &  3.07e+05 & 1.17e+05 &  4.35e+03 \\
					& &{Time (s)}  & 60.0  & 57.5 & 156 & 40.9 & 27.3 & 23.8 \\
					&	& FV&2.04e-01 & 2.04e-01 & 2.04e-01 & 2.04e-01  &2.06e-01 & 2.06e-01 \\
					&	& CVio& 0  & 0 & 0& 0  &1.24e-05 & 0  \\
					&	& SVio& 8.54e-04 & 8.30e-04 & 7.99e-04 &  9.43e-04  &4.06e-03 &  4.70e-03 \\
					\hline
				\end{tabular}
				\caption{Comparisons  among different methods on solving Problems 1 and 2.}\label{tab:2}
			\end{table}

						{In~\cref{fig:all_deter_pro3}, we solve Problem 3 on the \verb|MNIST| dataset and present 
						the values of FV, CVio, and SVio versus CPU time for SSG, IPP-ConEx, IPP-SSG, and our proposed   methods. 
						 Overall, our methods 
						reach lower FV and SVio values more quickly compared to other methods, while {achieving feasibility at most iterates}. Specifically, 3S-Econ-S attains an SVio of $1.29\times 10^{-3}$ at   21.9 seconds, and 3S-Econ-D reaches an SVio of $1.53\times 10^{-3}$ at   172 seconds,  significantly outperforming SSG, which achieves its lowest SVio of $1.07\times 10^{-2}$ at about 225 seconds. Additionally, IPP-ConEx remains stagnant at the zero solution, likely due to the zero point being a stationary solution of its proximal subproblem. }
							\begin{figure} 
							\begin{center}
								{\includegraphics[width=0.31\textwidth]{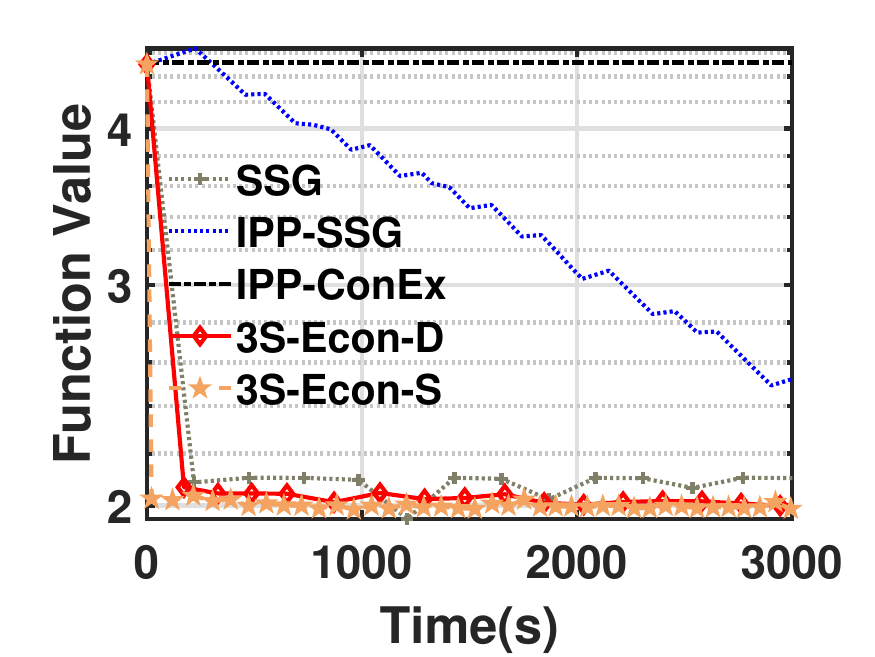}}
								{\includegraphics[width=0.31\textwidth]{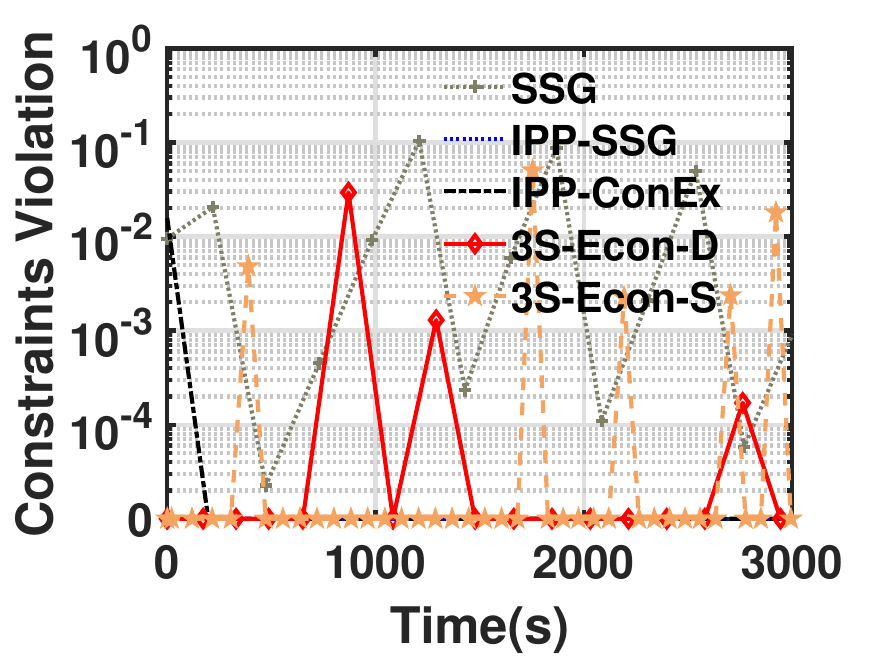}}
								{\includegraphics[width=0.31\textwidth]{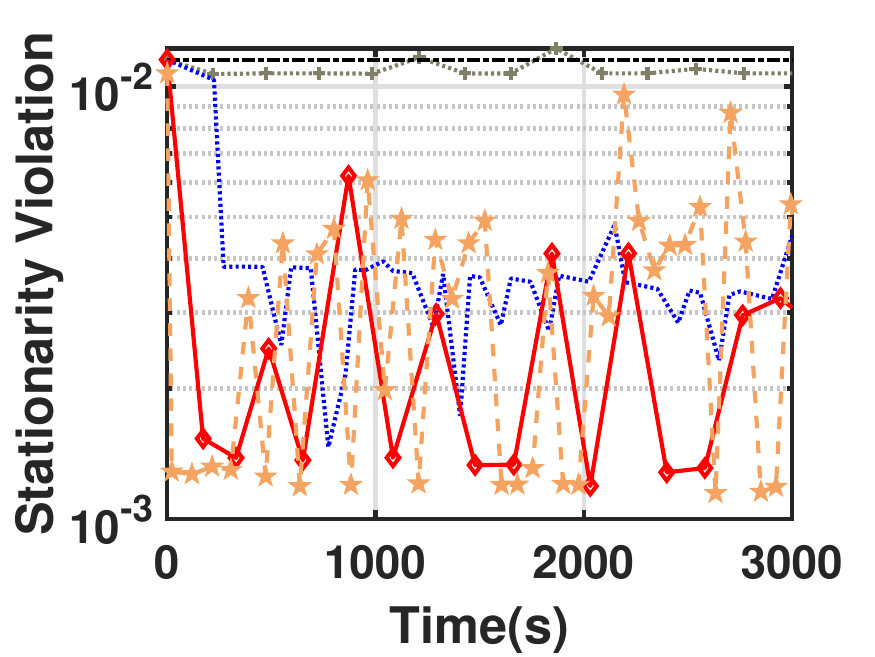}}
							\end{center}
							\caption{Comparisons of different methods on Problem 3 by using \texttt{MNIST} dataset.}\label{fig:all_deter_pro3}
						\end{figure}

\section{Conclusions and Discussions}\label{sec:conclu}
We have presented a novel FOM for solving nonconvex nonsmooth stochastic optimization with expectation constraints. Built on an exact penalty formulation that applies a smoothed penalty function on the inequality expectation constraint, our method incorporates the SPIDER-type estimation of the constraint function into the stochastic subgradient method. Through establishing the equivalence between the near-stationary solution of the penalty problem and the near KKT solution of the original problem, we achieve an optimal $O(\epsilon^{-4})$ iteration complexity result to produce an $(\epsilon,\epsilon)$-KKT point. In terms of sample complexity, our result is lower than a few state-of-the-art results by a factor of $\epsilon^{-2}$ in either constraint sample subgradient or constraint function values, or by a factor of $\sqrt{N}$ when the constraint is the average of $N$ sample functions. Moreover, our method delivers significantly superior numerical performance over state-of-the-art methods on solving fairness-constrained problems and Neyman-Pearson classification problem.

{Arjevani et al.~\cite{arjevani2023lower} establish a lower bound of order $\Omega(L\,\epsilon^{-4})$ for finding an $\epsilon$-stationary point (in expectation) for unconstrained smooth nonconvex stochastic optimization with $L$-Lipschitz gradients. For our penalized objective, assume in addition that $f$ and each $g_i$ are smooth with $L_f$- and $L_g$-Lipschitz gradients. A straightforward calculation shows that $\nabla F^{\nu}$ is $(L_f + \beta(l_g^2 + D L_g l_g)/\nu)$-Lipschitz continuous. Because our exact penalization requires $\nu=O(\epsilon)$, the effective Lipschitz constant scales as $\Omega(1/\epsilon)$; substituting it into the $\Omega(L\epsilon^{-4})$ bound suggests a lower bound of order $\Omega(\epsilon^{-5})$ for finding an $(\epsilon,\epsilon)$-KKT point of problem \cref{eq:model}, matching the upper bound results in~\cite{li2024stochastic}. This derivation relies on additional smoothness assumptions, but we conjecture that a total $\Omega(\epsilon^{-5})$ lower bound holds more broadly for penalty–based methods on problem \cref{eq:model}.
This observation is informative when compared with our upper bounds: it indicates that the iteration/gradient complexity (on the order of $\epsilon^{-4}$) can be substantially smaller than the \emph{total} complexity---$\Omega(\epsilon^{-5})$ if the conjectured lower bound holds, and $ {O}(\epsilon^{-6})$ in our analysis. These discussions raise two directions to pursue in the future: (i) proving tight lower bounds of  total complexity for stochastic FOMs on solving nonsmooth problems in the form of~\cref{eq:model}; and (ii) designing FOMs with $ {O}(\epsilon^{-5})$ CFC and ${O}(\epsilon^{-4})$ CGC/OGC. 
}

			\section*{Acknowledgement}				
			We would like to express our sincere gratitude to Dr.~Yankun Huang and Prof. Qihang Lin for providing their code in~\cite{huang2023single}. 
			We also thank the three anonymous referees and the associate editor for their constructive comments.

			\bibliographystyle{siamplain}
			\bibliography{optim}

			\appendix
			\normalsize

			\section{Proof of \cref{lem:necessary}}\label{appen:1}
			 
			\begin{proof}
			
			{ Let $\mathcal{J}_3 = \{ i :  g_i(\vx) \ge 0\}$. Then by the condition on $\cJ_1$, it holds $\mathcal{J}_1 \cap \mathcal{J}_3 \neq \emptyset$, and the set difference $\mathcal{J}_1 \cap \mathcal{J}_3^c \subset \{ i : g_i(\vx) < 0 \}$.}
			Since $0<g_+(\vx) \leq B_g$, then by \cref{def:2}, there exists $\mathbf{y} \in \mathrm{relint}(\mathcal{X})$ such that 
			$
			g_i(\mathbf{y}) + \frac{\overline{\rho}}{2}\|\mathbf{y} - \mathbf{x}\|^2 \leq -B
			$ for all $i \in [m]$.
			Now, for any $\boldsymbol{\zeta}_{g_i} \in \partial [g_i(\vx)]_+, {i \in \mathcal{J}_1}$, any $ \boldsymbol{\zeta}_{g_i} \in \partial g_i(\vx), {i \in \mathcal{J}_2}$,  and any $\vu \in \mathcal{N}_{\mathcal{X}}(\vx)$, we have
				\begin{align} \label{eq:lem:necessary-ineq1}	
					& \sum_{i\in\mathcal{J}_1}  [g_i(\mathbf{x})]_++\sum_{i\in\mathcal{J}_2} \lb_i g_i(\mathbf{x})
					+   \left \langle  \sum_{i\in J_1}\vzeta_{g_i}+\sum_{i\in\cJ_2}\lb_i \vzeta_{g_i}+\vu, \mathbf{y}-\mathbf{x} \right \rangle
					+ \frac{\overline{\rho}-\rho_g}{2}\|\mathbf{y}-\mathbf{x}\|^2 \nonumber
					\\
					=&\sum_{i\in\mathcal{J}_1\cap \cJ_3}  \left(\textstyle g_i(\mathbf{x})+\frac{\overline{\rho}}{2}\|\mathbf{x}-\mathbf{x}\|^2 \right)+  \textstyle \left \langle \sum_{i\in \cJ_1\cap\cJ_3}\vzeta_{g_i}+\vu, \mathbf{y}-\mathbf{x} \right\rangle
					+\frac{\overline{\rho}-\rho_g}{2}\|\mathbf{y}-\mathbf{x}\|^2 \nonumber
					\\
					& \quad +\sum_{i\in\mathcal{J}_2} \lb_i \left( \textstyle g_i(\mathbf{x})+\frac{\overline{\rho}}{2}\|\mathbf{x}-\mathbf{x}\|^2 \right) 
					+ \left \langle  \textstyle \sum_{i\in\cJ_2}\lb_i \vzeta_{g_i}, \mathbf{y}-\mathbf{x} \right \rangle \nonumber
					 \\
					 \leq &\sum_{i\in\mathcal{J}_1\cap\cJ_3} \left(g_i(\mathbf{y})+\frac{\overline{\rho}}{2}\|\mathbf{y}-\mathbf{x}\|^2\right) + \sum_{i\in\mathcal{J}_2} \lb_i\left(g_i(\mathbf{y})+\frac{\overline{\rho}}{2}\|\mathbf{y}-\mathbf{x}\|^2\right) \leq-B,
				\end{align}
				where the first equality follows from $[g_i(\mathbf{x})]_+= g_i(\mathbf{x}), \forall \, i\in\mathcal{J}_1\cap \cJ_3,$ and $[g_i(\mathbf{x})]_+= 0$, $\vzeta_{g_i}=\vzero, \forall\, i\in\mathcal{J}_1 \cap \mathcal{J}_3^c$, the first inequality comes from the $(\overline{\rho}-\rho_g)$-strongly convexity of $\sum_{i\in\mathcal{J}_1\cap \cJ_3}  (g_i(\cdot)+\frac{\overline{\rho}}{2}\|\cdot-\mathbf{x}\|^2 )  + \iota_{\mathcal{X}}(\cdot)$ and the convexity of  $\sum_{i\in\mathcal{J}_2} \lb_i \big(g_i(\cdot) + \frac{\overline{\rho}}{2}\|\cdot-\mathbf{x}\|^2\big)$,
				as $\cJ_1\cap \cJ_3\neq \emptyset$ and $\lb_i\ge 0, \forall\,{i\in J_2}$, the last inequality holds by $\cJ_1\cap \cJ_3\neq \emptyset$, $\lb_i\ge 0, \forall\,{i\in J_2}$, and $ g_i(\mathbf{y})+\frac{\overline{\rho}}{2}\|\mathbf{y}-\mathbf{x}\|^2 \leq - B$ for all $i\in[m]$. 
				
				The inequality \cref{eq:lem:necessary-ineq1}, together with Young's inequality, gives 
                \begin{align*}
                	\textstyle
                &-\frac{ \left\|  \sum_{i\in J_1}\vzeta_{g_i} + \vu + \sum_{i\in\cJ_2}\lb_i \boldsymbol{\vzeta}_{g_i}\right\|^2}{2(\overline{\rho}-\rho_g)} \\
              \textstyle  \leq  &\left\langle  \sum_{i\in J_1}\vzeta_{g_i}+\sum_{i\in\cJ_2}\lb_i \partial g_i(\vx)+\vu, \mathbf{y}-\mathbf{x} \right\rangle
                + \frac{\overline{\rho}-\rho_g}{2}\|\mathbf{y}-\mathbf{x}\|^2  \\
             \textstyle   \leq & -B - \sum_{i\in\mathcal{J}_1}  [g_i(\mathbf{x})]_+ - \sum_{i\in\mathcal{J}_2} \lb_i g_i(\mathbf{x}) \leq -B,
                \end{align*} 
                where the last inequality holds by $g_i(\vx)\ge 0, \forall\, i\in\mathcal{J}_1\cap\cJ_3$, and $\lb_i g_i(\vx) \ge 0, \forall\, i\in J_2$. 
			Since the choice of subgradients and $\vu$ is arbitrary, the proof is completed.
			\end{proof}
			
		 	\section{Problems 2--3 in \cref{eq:data} satisfy \cref{def:2} and the strong feasibility condition in~\eqref{eq:str-feas}}\label{appen:2}
			We first  verify in this section that the fairness-constrained logistic‐loss problem in~\cref{eq:data} satisfies~\cref{def:2}.  
		Let $\rho_g$ and $l_g$ denote the weak convexity modulus and Lipschitz continuity modulus of $\vg$ in Problem~2, respectively.
			Under mild conditions, the following lemma 
			establishes the Slater-type CQ. 
			 \begin{lemma}[{\cite[Theorem B.1]{huang2023single}}]
			 	Suppose that \([\va_i^{\mathrm p}]_1 = 1\) for $i=1,\dots,n_{\mathrm p}$ and
			 	$[\va_i^{\mathrm u}]_1 = -1$ for $i=1,\dots,n_{\mathrm u}$.
			 	Let
			 	\(
			 	\mathcal K
			 	:=
			 	\bigl\{\,\lvert \frac{i}{n_{\mathrm p}} - \frac{j}{n_{\mathrm u}}\rvert :
			 	i=0,\dots,n_{\mathrm p},\;
			 	j=0,\dots,n_{\mathrm u}\bigr\},
			 	\)
			 	and pick $\kappa \notin \mathcal{K}$ such that $\kappa$ is strictly greater than the smallest element of $\mathcal{K}$.
			 	Let $\underline\kappa,\bar\kappa\in\mathcal K$ be the two grid points that bracket $\kappa$ with no other $\mathcal K$-points in $(\underline\kappa,\bar\kappa)$.   
			 	Define
			 	$
			 	\overline\Psi(\vx)
			 	:=
			 	\frac1{n_{\mathrm p}}\sum_{i=1}^{n_{\mathrm p}}\sigma(\vx^\top\va_i^{\mathrm p})
			 	-
			 	\frac1{n_{\mathrm u}}\sum_{i=1}^{n_{\mathrm u}}\sigma(\vx^\top\va_i^{\mathrm u}), \text{ and }
			 	$
			 	$$
			 	q(\vx)
			 	:=\max\left\{
			 	\max_{i=1,\dots,n_{\mathrm p}}
			 	\!\frac{\sigma(\vx^\top\va_i^{\mathrm p})(1-\sigma(\vx^\top\va_i^{\mathrm p}))}{n_{\mathrm p}},
			 	\max_{i=1,\dots,n_{\mathrm u}}
			 	\!\frac{\sigma(\vx^\top\va_i^{\mathrm u})(1-\sigma(\vx^\top\va_i^{\mathrm u}))}{n_{\mathrm u}}\right\}.
			 	$$
			Then the following statements hold:
			\begin{itemize} 
				\item[$\mathrm{(i)}$] 
				Let
				$
				q_{\mathrm{min}}=\inf \{q(\mathbf{x})| | \overline{\Psi}(\mathbf{x}) \mid \in[(\kappa+\underline{\kappa}) / 2,(\kappa+\bar{\kappa}) / 2]\} .
				$
				It holds $q_{\mathrm{min}}>0$. 
				\item[$\mathrm{(ii)}$] 
				Let $\underline{q}\in (0, q_{\mathrm{min}}]$. Then Problem 2 in \cref{eq:data} satisfies the Slater-type CQ with any $\bar{\rho}>\rho_g$, and with parameters  
				$
				0<B\leq \min \left\{\frac{\kappa-\underline{\kappa}}{2}, \frac{\underline{q}^2}{\rho_g+4\bar{\rho}}, \frac{\underline{q}^2 \kappa}{l_g^2}\right\}
				$ and
				$
				0<B_g \leq \min \left\{\frac{\bar{\kappa}-\kappa}{2}, \frac{\underline{q}^2}{\rho_g+4\bar{\rho}}, \frac{\underline{q}^2 \kappa}{l_g^2}\right\}.
				$
			\end{itemize}
			\end{lemma} 
			We note that though $q_{\mathrm{min}}$ is difficult to compute, for our tested Problem 2, 
			we can directly establish a 
			positive lower bound on $q(\vx)$ uniformly about $\vx\in \cX$.
	In particular, let \(D_a=\min\left(\min_{1\le i\le n_{\mathrm p}}\|\va_i^{\mathrm p}\|, \min_{1\le j\le n_{\mathrm u}}\|\va_j^{\mathrm u}\|\right)\), i.e., there is $\va_i^{\mathrm p}$ or $\va_j^{\mathrm u}$ such that $\|\va_i^{\mathrm p}\|=D_a$ or $\|\va_j^{\mathrm u}\|=D_a$. 
	Since $\|\vx\|\le D,\forall\, \vx\in \cX$, it holds	\mbox{$
		|\vx^{\top}\va^{\mathrm p}_i|\le D_aD$}    {or}   $|\vx^{\top}\va^{\mathrm u}_j|\le D_aD.
		$
		Consequently, we have $q(\vx) \ge \frac{\sigma(D_aD)\bigl(1-\sigma(D_aD)\bigr)}{\max\{n_{\mathrm p},\,n_{\mathrm u}\}} =: \tilde{q},\forall\, \vx\in \cX$, and thus $q_{\mathrm{min}}\ge \tilde{q}$. 
		Therefore, we can set $\underline{q}= \tilde{q}$ in the above Statement (ii).
		
		Next, we verify that Problem~3 defined in~\cref{eq:data} also satisfies~\cref{def:2}. For each $i \in \{2,\ldots,M\}$, let $g_{i-1}$ be the component constraint function that corresponds to $\cD_i$ 
		in Problem~3, and suppose that 
		$[\va_i]_1 = 1, \forall\, \va_i \in \cD_i$. Under this assumption, one can easily show that for any $i \in \{2,3,\ldots,M\}$,
		\begin{align}
			\|\nabla g_{i-1}(\mathbf{x})\| 
			&\geq \frac{1}{\left|\mathcal{D}_i\right|} \sum^M_{\substack{l=2 \\ l \neq i}} \sum_{\va_i \in \mathcal{D}_i}  \sigma\left(\mathbf{x}_i^{\top} \va_i-\mathbf{x}_l^{\top} \va_i\right)\left(1-\sigma\left(\mathbf{x}_i^{\top} \va_i-\mathbf{x}_l^{\top} \va_i\right)\right)\left[\mathbf{a}_i\right]_1,
		\end{align}
		by the fact $\|\nabla g_{i-1}(\mathbf{x})\|\ge \|\nabla_{[\vx_i]_1} g_{i-1}(\mathbf{x})\|$.
		Let $D_a=\max_{i=2,3,\ldots,M}\{\min_{\va_i\in\cD_i}\|\va_i\|\}$, i.e., for each $i\in\{2,3,\ldots,M\}$, there is $\va_i\in\cD_i$  such that $\|\va_i\|\le D_a$. Since we have $\|\vx\|\le D,\forall\, \vx\in \cX$, it holds	
 $|\mathbf{x}_i^{\top} \va_i-\mathbf{x}_l^{\top} \va_i|\leq 2D_a D$. By an analogous argument, we  establish that for all $i \in \{2,3,\ldots,M\}$,
		\[
		\|\nabla g_{i-1}(\mathbf{x})\| \geq \frac{\sigma(2D_a D)\bigl(1-\sigma(2D_a D)\bigr)}{\max_{i=2,3,\ldots,M}\{|\mathcal{D}_i|\}}=: \tilde{q}.
		\]
		Hence, following steps (83)--(85) in the proof of~{\cite[Theorem B.1]{huang2023single}}, we are able to show that Problem~3 satisfies the Slater-type  CQ as defined in~\cref{def:2}. We do not repeat the details here. 
		
		Finally, we provide sufficient conditions ensuring that Problems~2 and~3 in~\cref{eq:data} satisfy the strong feasibility condition in~\eqref{eq:str-feas}.
			
			Specifically, for Problem~2, it suffices to choose
			\[
			\kappa_2 \;\ge\; \Psi_0(\vx_{\mathrm{feas}}) + 2\rho_g D^2
			\;=\; \Psi_0(\vx_{\mathrm{feas}}) + 50 \max\bigg\{2\lb, \frac{1}{n_{\mathrm{p}}}\sum_{i=1}^{4n_{\mathrm{p}}}\|\va_i^{\mathrm{p}}\|^2 + \frac{1}{n_{\mathrm{u}}}\sum_{i=1}^{4n_{\mathrm{u}}}\|\va_i^{\mathrm{u}}\|^2\bigg\}
			\]
			for some $\vx_{\mathrm{feas}} $ satisfying $\|\vx_{\mathrm{feas}}\|_{\infty}\leq D$.
			
			For Problem~3, it suffices to choose
			\[
			\kappa \;\ge\; \frac{1}{\left|\mathcal{D}_i\right|}\sum_{\substack{l=2,\\ l\neq i}}^M\sum_{\va_i\in\mathcal{D}_i}\phi\left((\mathbf{x}^{(0)}_i)^{\top}\va_i-(\mathbf{x}^{(0)}_l)^{\top}\va_i\right)
			\;+\;2\rho_g D^2
			\]
			for some $\vx^{(0)}$ satisfying $\|\vx^{(0)}\|_{\infty}\leq D$.
			This condition explicitly becomes
			\[
			\kappa \;\ge\;\frac{1}{\left|\mathcal{D}_i\right|}\sum_{\substack{l=2,\\ l\neq i}}^M\sum_{\va_i\in\mathcal{D}_i}\phi\left((\mathbf{x}^{(0)}_i)^{\top}\va_i-(\mathbf{x}^{(0)}_l)^{\top}\va_i\right)
			\;+\;0.18\max_{i\in[M]}\frac{1}{|\cD_i|}\sum_{\va_i\in\cD_i}\|\va_i\|^2.
			\]

	\end{document}